\documentclass[9pt,letterpaper]{article}
\usepackage{amsmath,amssymb,fullpage}
\usepackage{mathtools}
\usepackage{etoolbox}
\usepackage{enumerate}
\usepackage{dsfont}
\usepackage{mathrsfs}
\usepackage{tabularray}
\usepackage{linegoal}
\usepackage{ifthen}
\usepackage{multicol}
\usepackage{svg}
\usepackage{caption}
\usepackage{subcaption}

\usepackage{tcolorbox}
\tcbuselibrary{breakable}

\usepackage{pdfpages}

\everymath{\displaystyle}
\numberwithin{equation}{section}

\UseTblrLibrary{diagbox}
\UseTblrLibrary{booktabs}
\usepackage{textcomp, marvosym}

\usepackage[noadjust, nosort]{cite}
\usepackage{graphicx}

\usepackage{nameref, hyperref}

\usepackage[right]{lineno}

\usepackage{microtype}
\DisableLigatures[f]{encoding = *, family = * }

\usepackage{array}

\usepackage{here}
\usepackage{soul}

\date{}

\usepackage{zref-xr}
\zxrsetup{tozreflabel=false, toltxlabel=true, verbose}
\zexternaldocument*{./Rein_Cell_Systems_SI_Revised_v1}

\providecommand{\blue}[1]{\color{black}{#1}\color{black}\hspace{0pt}}
\providecommand{\black}[1]{\color{black}{#1}\color{black}\hspace{0pt}}

\newtheorem{theorem}{Theorem}[section]
\newtheorem{corollary}[theorem]{Corollary}
\newtheorem{proposition}[theorem]{Proposition}
\newtheorem{lemma}[theorem]{Lemma}
\newtheorem{define}[theorem]{Definition}
\newtheorem{example}[theorem]{Example}

\newtheorem{assumption}[theorem]{Assumption}

\def\Xz{\boldsymbol{X}}

\def\Zz{\boldsymbol{Z}}
\def\Yz{\boldsymbol{Y}}
\providecommand\phib{\boldsymbol{\emptyset}}
\providecommand\X[1]{\boldsymbol{X_{#1}}}
\providecommand\Z[1]{\boldsymbol{Z_{#1}}}

\def\adj{\textnormal{Adj}}

\def\d{\mathrm{d}}

\def\ds{\mathrm{d}s}
\def\d{\mathrm{d}}

\def\llongrightarrow{\relbar\joinrel\relbar\joinrel\relbar\joinrel\rightarrow}

\providecommand{\rarrow}[1]{\stackrel{#1}{\llongrightarrow}}
\providecommand{\rarrowl}[1]{\underset{#1}{\llongrightarrow}}

\DeclareMathOperator*{\col}{col}

\newcommand\needref[1]{\ifthenelse{\equal{#1}{}}
{\textcolor{red}{\textbf{[REF NEEDED]}}}
{\textcolor{red}{\textbf{[REF NEEDED: #1]}}}}

\DeclareMathOperator{\eps}{\varepsilon}
\DeclareMathOperator*{\trace}{trace}
\DeclareMathOperator{\rank}{rank}
\DeclareMathOperator*{\diag}{diag}
\DeclareMathOperator*{\sign}{sign}
\def\mn{\mathrm{min}}

\newenvironment{proof}{{\it Proof :~}}{\hfill$\diamondsuit$\\}

\renewenvironment{bmatrix}{\left[\begin{array}{ccccccccccccccccccc}}{\end{array}\right]}

\newcounter{boxcounter}
\usepackage{tikz}
\usetikzlibrary{arrows.meta}

\tikzset{
    regulations/.cd,
    act/.style={-{Stealth}}, % Style for activation
    rep/.style={-{Bar}}, % Style for repression
}

\newcommand{\regulationarrow}[1][act]{%
    \tikz[baseline] {\draw[regulations/#1] (0,0.5ex) --++ (1.5em,0);}%
}

\newcommand{\nocontentsline}[3]{}
\let\origcontentsline\addcontentsline
\newcommand\stoptoc{\let\addcontentsline\nocontentsline}
\newcommand\resumetoc{\let\addcontentsline\origcontentsline}

\begin{document}
\vspace*{0.2in}

\begin{flushleft}
{\Large Structural Stability Properties of Antithetic Integral (Rein) Control with Output Inhibition
\textbf\newline{} % Please use "sentence case" for title and headings (capitalize only the first word in a title (or heading), the first word in a subtitle (or subheading), and any proper nouns).
}
\newline
\\
Corentin Briat\footnote{Corentin Briat is now with the School of Life Sciences, University of Applied Sciences Northwestern Switzerland, Switzerland. email: corentin@briat.info, url: www.briat.info.} and Mustafa Khammash\footnote{email: mustafa.khammash@bsse.ethz.ch, url: https://bsse.ethz.ch/ctsb}
%with the Lorem Ipsum Consortium\textsuperscript{\textpilcrow}
\\
\bigskip
\textbf{Department of Biosystems Science and Engineering, ETH-Z\"urich, Switzerland}\\
%
%\textbf{1} Affiliation Dept/Program/Center, Institution Name, City, State, Country
%\\
%\textbf{2} Affiliation Dept/Program/Center, Institution Name, City, State, Country
%\\
%\textbf{3} Affiliation Dept/Program/Center, Institution Name, City, State, Country
%\\
\bigskip

% Insert additional author notes using the symbols described below. Insert symbol callouts after author names as necessary.
%
% Remove or comment out the author notes below if they aren't used.
%
% Primary Equal Contribution Note

% Additional Equal Contribution Note
% Also use this double-dagger symbol for special authorship notes, such as senior authorship.
%\ddag These authors also contributed equally to this work.

%\textcurrency Current Address: Dept/Program/Center, Institution Name, City, State, Country % change symbol to "\textcurrency a" if more than one current address note
% \textcurrency b Insert second current address
% \textcurrency c Insert third current address


% Group/Consortium Author Note
%\textpilcrow Membership list can be found in the Acknowledgments section.

% Use the asterisk to denote corresponding authorship and provide email address in note below.

\end{flushleft}
% Please keep the abstract below 300 words NOW 302
\section*{Abstract}

Perfect adaptation is a well-studied biochemical homeostatic behavior lying at the core of biochemical regulation. While the concepts of homeostasis and perfect adaptation are not new, their underlying mechanisms and associated biochemical regulation motifs are not yet fully understood. Insights from control theory unraveled the connections between perfect adaptation and integral control, a prevalent engineering control strategy. In particular, the recently introduced Antithetic Integral Controller (AIC) has been shown to successfully ensure perfect adaptation properties to the network it is connected to. The complementary structure of the two molecules the AIC relies upon allows for a versatile way to control biochemical networks, a property which gave rise to an important body of literature pertaining to mathematically elucidating its properties, generalizing its structure, and developing experimental methods for its implementation. The Antithetic Integral Rein Controller (AIRC), an extension of the AIC in which both controller molecules are used for control, holds many promises as it supposedly overcomes certain limitations of the AIC. We focus here on an AIRC structure with output inhibition that combines two AICs in a single structure. We theoretically demonstrate its superior properties, which are linked to the intrinsic properties of a specific AIC structure with output inhibition. This controller ensure structural stability and structural perfect adaptation properties for the controlled network under mild assumptions, meaning that this property is independent of the parameters of the network and the controller. The results are very general and valid for the class of unimolecular mass-action networks as well as more general networks, including cooperative and Michaelis-Menten networks. We also provide a systematic and accessible computational way for verifying whether a given network satisfies the conditions under which the structural property would hold. Finally, we also propose a possible implementation of such a regulation mechanism using an intein-based synthetic circuit.

\stoptoc
\section{Introduction}

%\begin{itemize}
%\item \red{Set-point admissibility is meaningless without stability}
%\item Why do we care about different set-points.
%\item better motivate the problem and anticipate the main results of the paper
% \item Take-away message
% \item Bring set-point admissibility at the forefront
%\end{itemize}

In his pioneering works, Walter B. Cannon laid the groundwork for studying biological and physiological regulation and coined the term ``homeostasis''. This concept, which he introduced in 1929, describes the ability of living organisms to regulate their internal processes \cite{Cannon:29}. Homeostasis encompasses many physiological processes, such as the regulation of body temperature, blood sugar, and cholesterol \cite{Karin:16,Scheepers:23} as well as neuronal function in neuroscience \cite{Yang:23}. Similar ideas were later adapted and further developed in the seminal works of Jacob, Monod, and co-workers  \cite{JacobMonod:61,Monod:63}.\\

A stringent type of homeostatic regulation of particular recent interest is perfect adaptation.  It refers to the ability of a network that is subjected to persistent stimuli (or perturbations) to adapt to that stimulus by having the molecular concentration of one or more of its constituent species return to their pre-stimulus level. Robust perfect adaptation is achieved when the adaptation property still holds in spite of perturbations of the network's parameters or topology. The phenomenon of robust perfect adaptation has been observed in biological networks \cite{Barkai:97,Yi:00,ElSamad:02,Alon:07,Muzzey:09}, and the elucidation of its underlying mechanisms, design principles, and topological requirements has been extensively reported using both theoretical and computational approaches \cite{Yi:00,Alon:07, Drengstig:08, Ma:09, Ni:09, Briat:15e, Araujo:18, Khammash:21, Gupta:22, Araujo:23, Hirono:23, Briat:20:Structural, Blanchini:21, Khammash:21}, among others.  It is important to note that the concept of robust perfect adaptation is not uniform across all scenarios, and variations exist depending on the exact class of perturbations to which adaptation is robust. It is therefore essential to specify the particular type of perturbations being considered, as this can influence the underlying conditions needed to achieve perfect adaptation.\\

The \emph{perfect adaptation design problem}, is the problem of designing mechanisms that ensure perfect adaptation for a given biochemical network. This problem has recently attracted a lot of attention, partly due to its parallels with the so-called \emph{regulation problem} and \emph{disturbance rejection problem} in control engineering. These problems involve designing industrial controllers that enable the controlled system to compensate for unmeasured environmental variations and perturbations (disturbances), while keeping the controlled variables at a pre-defined, chosen level (set-point). A simple control strategy known to structurally solve the regulation and the disturbance rejection problems is the so-called \emph{integral controller}, which forms the core of Proportional-Integral-Derivative (PID) controllers, a prevalent control algorithm found in a vast range of applications \cite{Astrom:95}. This connection between biology, control engineering, and mathematics echos Wiener's Cybernetics \cite{Wiener:61}, which proposed the existence of control motifs in living systems and suggested that they be studied using the same tools used to study man-made control systems. \\

The Antithetic Integral Controller (AIC) \cite{Briat:15e} is one such generic motif. It employs a simple strategy that ensures robust perfect adaptation properties in biochemical systems in both the deterministic and stochastic settings through the implementation of an integral feedback module. Its core mechanism is that of molecular sequestration \cite{Chen:12}, a method for comparing two molecular levels. The theoretical properties of the Antithetic Integral Controller and its leaky variants have now been well-studied, in both the deterministic \cite{Briat:17ACS, Qian:19,Olsman:19,Olsman:19b,Briat:19:Logistic, Briat:20:Structural, Filo:21a, Filo:21b} and the stochastic \cite{Briat:15e, Briat:18:Interface, Aoki:19} settings. \blue{Notably, the stochastic model of a specific AIC type -- referred to in this paper as naAIC and illustrated in Figure~\ref{main:fig:AICtable} -- exhibits structural stability properties under very mild conditions. Interestingly, this property is absent in its deterministic counterpart, which requires significantly stronger conditions to achieve stability. A key requirement for effective control is \emph{set-point admissibility}, which ensures that the topology of the controlled network allows molecular levels of the controlled species to reach a desired set-point, potentially through appropriate tuning of the control parameters. As discussed in Box 1 for the specific case of a gene expression network and more generally in \cite{Briat:15e,Briat:19:Logistic}, the original AIC -- referred to here as naAIC and defined in \eqref{main:eq:AIC:naAIC} -- is inherently limited to regulating the controlled species above its basal expression level. Sub-basal levels are unattainable, consistent with the fact that the naAIC structure influences the transcription rate of mRNA or, more broadly, acts as a direct or indirect activator of the controlled species. In contrast, a less explored AIC structure -- denoted here as niAIC and defined in \eqref{main:eq:AIC:niAIC}, also discussed in the context of gene expression regulation in Box 1 -- operates as an inhibitor of the controlled species. This structure allows for regulation of protein levels below their basal expression level. This distinction highlights the critical impact of input selection and controller network topology on the set of admissible set-points for the controlled molecular species.}\\

\blue{A natural question arises: can a controller topology be designed to both expand the set of admissible set-points and ensure the necessary stability properties of the controlled network? Developing a general strategy for this purpose would not only enhance the theoretical appeal of the design but also improve its practical robustness to changes in the process. For instance, one could easily envision scenarios where environmental or cellular context changes render a previously admissible set-point non-admissible. Such changes would inevitably destabilize the controlled network, potentially resulting in oscillations, saturations, or even cell death.}\\

The Antithetic Integral Rein Controller (AIRC), introduced in \cite{Gupta:19}, exploits the bimolecular structure of the AIC by utilizing its two controller species in separate channels for network modulation and giving rise to a wide variety of possible control topologies, as can be seen in Figure~\ref{main:AIRC:AIRC}. The specific AIRC considered in \cite{Gupta:19}, which we refer here as an \emph{aiAIRC with output inhibition} and depicted in Figure~\ref{main:AIRC:aiAIRC}, employs the first controller species as an indirect activator of the controlled species and the second as a direct inhibitor of the controlled species. This configuration offers certain advantages, including the possibility to reach a wider range of set-points, compared to the naAIC described above, and potentially improving the transient response of the network \cite{Gupta:19}. \blue{This controller network topology could potentially expand the range of admissible set-points -- a possibility that, to date, has not been rigorously proven in general. Similarly, the analysis of its essential stability properties has largely been limited to the context of gene expression \cite{Gupta:19}, providing only limited insight into more general network configurations. While set-point admissibility is an important prerequisite for adaptation, the central concern is stability. Stability ensures that the controlled network reliably converges to the steady state associated with the desired set-point for the controlled molecular species. Once again, both the choice of inputs to the controlled network and the controller topology play a crucial role in determining stability. These decisions must be made carefully, in accordance with the system's property in order to ensure the derired properties for the controlled network.}\\

Our aim in this work is to elucidate the properties of both the aiAIRC with output inhibition and its component, the niAIC with output inhibition. We first show that, for unimolecular networks, the aiAIRC with output inhibition yields nonnegative equilibria regardless of the value of the set-point,  implying that, in contrast to the naAIC, all positive set-points are admissible for that controller. This finding is established here more broadly and simply than in previous studies \cite{Gupta:19,Plesa:23}. Additionally, we show that the \blue{aiAIRC} with output inhibition exhibits an interesting equilibrium switching behavior in the strong sequestration regime, which corresponds to the regime where the controller species sequestration rate is large. This regime has been extensively studied in theoretical analyses of antithetic integral control structures \cite{Qian:18}. In this regime, the AIRC toggles between activation and output inhibition of the controlled species based on the set-point's relation to the basal expression level, behaving alternately like an naAIC or an niAIC with output inhibition.  This behavior underscores the importance of studying the much less explored niAIC with output inhibition.\\

We first establish that the niAIC with output inhibition locally behaves like a filtered Proportional-Integral (PI) controller, in contrast to the naAIC's sole integral control action, as discussed in similar context in \cite{Filo:21a,Filo:21b}. \blue{A key difference of PI controllers over purely integral controller, as implemented by the naAIC in \cite{Briat:15e,Briat:19:Logistic}, is that the proportional action has the capacity of stabilizing the dynamics of a controlled system and improving its transient behavior, as also pointed out in \cite{Briat:18:Interface}. We further demonstrate that this class of controllers solves the perfect adaptation problem and provide necessary conditions on the network for adaptation that are similar in flavor to those obtained for the naAIC in \cite{Briat:19:Logistic}. We then leverage ideas from control theory, dynamical systems, and linear algebra to show that when the network is unimolecular and stable, the closed-loop network will also be stable, provided that the set-point is admissible (i.e. below the basal expression level) for all possible values for the gain of the controller and its sequestration rate. Exploiting the stabilization properties of the controller, we then extend these results to a more general class of networks, which we refer to \emph{output unstable networks}, for which the same result holds but with no restriction on the set-point value. These results are further extended to the AIRC with output inhibition.}\\

Addressing general nonlinear networks presents a challenge due to their diversity and complex technical aspects. Despite these hurdles, we successfully establish structural stability principles applicable to these networks. These principles can be assessed on case-by-case manner using readily available computational methods, which we outline. Interestingly, we find that for certain significant subclasses, such as cooperative and Michaelis-Menten networks, those extensive additional analyses are unnecessary. Indeed, in these case we can directly apply  general results similar to those established for the unimolecular case.\\

Concluding our study, we propose a feasible implementation of the controller using intein-based systems \cite{Nanda:20,Wang:22}. This approach draws on existing knowledge in synthetic biology and previous developments in intein-mediated controllers, most notably the work of Anastassove et al. \cite{Anastassov:23}.\\

\blue{The main takeaway of this paper is that AIRC networks offer a powerful framework for controlling reaction networks, possibly in the presence of non-monotonic reaction paths that may alternate between activation and inhibition depending on the network configuration. These controller networks address set-point limitations and ensure structural stability, both of which are critical for controlling highly uncertain systems using controllers that face similar constraints and for which fine-tuning may not be feasible. This paper further complete past works from the same authors and provides a mathematical foundation for these claims in the context of niAIC with output inhibition and its associated aiAIRC structure. We believe this work marks a significant step forward in advancing a robust mathematical control theory for reaction networks, systems biology, and synthetic biology.}

\subsection*{Reaction Networks}

Reaction networks are a very powerful modeling paradigm that can be used to represent any population system, such as those arising in biology, ecology, epidemiology, etc.\cite{Goutsias:13}. A reaction network $(\Xz,\mathcal{R})$ consists of a set of $n$ molecular species $\Xz=\{\X{1},\ldots,\X{n}\}$ that interact through $K$ reaction channels $\mathcal{R}=\{\mathcal{R}_1,\ldots,\mathcal{R}_K\}$ denoted as
\begin{equation}\label{main:eq:RN}
 \mathcal{R}_k:\ \sum_{i=1}^n\zeta_{k,i}^l\X{i}\rarrow{\lambda_k}\sum_{i=1}^n\zeta_{k,i}^r\X{i},\ k=1,\ldots,K
\end{equation}
where $\zeta_{k,i}^\ell,\zeta_{k,i}^r\in\mathbb{Z}^n_{\ge0}$ are the left and right stoichiometric vectors. The stoichiometric vector of reaction  $\mathcal{R}_k$ is given by $\zeta_k:=\zeta_k^r-\zeta_k^\ell\in\mathbb{Z}^n$ where $\zeta_k^r=\col(\zeta_{k,1}^r,\ldots,\zeta_{k,n}^r)$ and $\zeta_k^l=\col(\zeta_{k,1}^l,\ldots,\zeta_{k,n}^l)$. %The stoichiometry matrix $S\in\mathbb{Z}^{n\times K}$ is defined as $S:=\begin{bmatrix}  \zeta_1&\ldots&\zeta_K\end{bmatrix}$.
Each reaction $\mathcal{R}_k$ is also described by its propensity function $\lambda_k$ that describes, in the deterministic setting, the rate at which this reaction occurs. Such rates may take multiple forms such as mass-action, Hill, Michaelis-Menten, etc. \cite{Voit:00,Alon:07}. In particular, when the network only has mass-action kinetics, only the reaction rates are indicated on top of the arrow in \eqref{main:eq:RN}. This will be explicitly mentioned when this is the case. In all those cases, we define $\mathcal{P}$ to be the set of the network parameters, that is, the set of all parameters describing the reaction rates or, more generally, all the parameters of the propensity functions.\\

In the deterministic setting, reaction networks are quantitatively described in terms of a vector of molecular concentrations, denoted here by $x(t)$, which evolves on a state-space $\mathcal{S}\subseteq\mathbb{R}_{\ge0}^n$. As a result, the propensity functions $\lambda_k:\mathcal{S}\mapsto\mathbb{R}_{\ge0}$ are defined in such a way that $\mathcal{S}$ is forward invariant; i.e. for all $x_0\in\mathcal{S}$, we have that $x(t)\in\mathcal{S}$ for all $t\ge0$. This will be tacitly assumed to be the case in the rest of the paper. The dynamical model representing the deterministic reaction network \eqref{main:eq:RN} is, therefore, given by the Reaction Rate Equation (RRE)
\begin{equation}\label{main:eq:RRE}
\begin{array}{rcl}
  \dot{x}(t)&=&\displaystyle\sum_{k=1}^K\zeta_k\lambda_k(x(t)),\ t\ge0\\ % S\lambda(x(t))=
  x(0)&=&x_0.
\end{array}
\end{equation}
The point $x^*$ is said to be an equilibrium point for the above dynamics if $\textstyle\sum_{k=1}^K\zeta_k\lambda_k(x^*)=0$.

\subsection*{Perfect Adaptation}

The perfect adaptation property is the property of a reaction network that certain molecular counts will return to the previous equilibrium levels after the appearance of environmental changes and network perturbations. This is one of the many types of homeostatic behaviors that can be found in living organisms \cite{Cannon:29} or that can be theoretically constructed \cite{Alon:07,Ma:09,Briat:15e,Golubitsky:17,Khammash:21,Gupta:22}. The version we consider in this paper is the one below:
\begin{define}[Perfect Adaptation]\label{main:def:RPA}
  Consider a reaction network $(\Xz,\mathcal{R})$ and a species $\Yz\in\Xz$. The species $\Yz$ is said to exhibit the perfect adaptation property in the network $(\Xz,\mathcal{R})$ if
  \begin{enumerate}
  %\item the concentrations $x(t)$ of the molecular species $\Xz$ are bounded at all times,
  %\item The steady-state value of the controlled species is equal to the set-point $y^*$, provided that this set-point is admissible, that is $y(t)\to y^*$ as $t\to\infty$.
  \item the equilibrium point $x^*$ is (locally) asymptotically stable for the dynamics of the network \eqref{main:eq:RRE}, and
  \item the equilibrium value of the concentrations of the species $\Yz$, denoted by $y^*$, is independent of all the parameters in a subset of $\mathcal{P}$.
  \end{enumerate}
\end{define}

The definition of perfect adaptation we consider consists of two components. The stability component is here to guarantee that the network goes back to the same equilibrium after a state perturbation. The second component is the equilibrium perfect adaptation property which indicates that if we perturb the parameters in some subset of $\mathcal{P}$, the equilibrium for $\Yz$, denoted by $y^*$, remains the same. Note, however, that the equilibrium states may not be the same. \blue{Stronger definitions of perfect adaptation exist. For example, \emph{robust perfect adaptation} requires that the perfect adaptation property remains valid under sufficiently small perturbations to certain network parameters. On the other hand, the property is termed \emph{structural} when it holds for all positive values of specific network parameters. Structural perfect adaptation can be seen as the ultimate form of robustness, which is ensured by the structure of the network rather than by the values of its parameters. We will particularly focus on this latter variant as it is particularly desirable in biological contexts where network uncertainties are significant and fine-tuning is challenging, if not impossible. }\\

Perfect adaptation and its variants are rather strong properties: endogenous and engineered networks may not immediately satisfy them by design. Analogously, industrial systems are not naturally self-regulating (e.g. robots) and require additional components to make them fully functional and reliable. The design of such components is the purpose of control engineering and, keeping that exact same state of mind, we may pose the following question: how to turn a non-adapting network into a perfectly adapting one? There are different possible ways to do so depending on how the problem is formulated and the different constraints. The version of the problem we consider is the one below:

\begin{define}[Perfect Adaptation \blue{Design} Problem]\label{main:def:RPAP}
  Consider a reaction network $(\Xz,\mathcal{R})$ and a species $\Yz\in\Xz$. The perfect adaptation problem consists of finding a controller network $(\Zz,\mathcal{R}^c)$ such that
  \begin{enumerate}
  \item $\Yz$ exhibits the perfect adaptation property in the network $(\Xz\cup\Zz,\mathcal{R}\cup\mathcal{R}^c)$, and
  \item $y^*$ is equal to a tunable set-point which is a known function of the controller parameters only (i.e. it is independent of the parameters of the network $(\Xz,\mathcal{R})$).
  \end{enumerate}
\end{define}

The above problem is a reaction network analogue of the more standard engineering control problems in which controllers are designed to ensure certain specifications for the controlled systems. In the present case, the controllers are to be implemented within cells using biological components, whence the name  \emph{in-vivo controller}, rather than in a computer. Due to the nature of the substrate, specific constraints and challenges, usually absent in more standard engineering problems, are to be considered to ensure the proper operation of the controllers; see e.g. \cite{Khammash:22:Cyber,Briat:23:Annual}. The requirement that the set-point be adjustable is facultative here but allows for an increased flexibility and the possible optimization of the overall controlled network, as discussed in \cite{Briat:17ACS} in the context of the control of metabolic networks.\\

When the set-point needs to be adjusted, the question of which set-point values are compatible with the network arises. This leads us to the concept of set-point admissibility:
\begin{define}[Set-point admissibility]
  Consider a reaction network $(\Xz,\mathcal{R})$ and a species $\Yz\in\Xz$. Moreover, decompose $\mathcal{P}$ as $\mathcal{P}_f\cup\mathcal{P}_t$ where $\mathcal{P}_f$ is a set of fixed network parameters and $\mathcal{P}_t$ is a set of tunable network parameters.

  We say that a set-point $r$ is admissible for the species  $\Yz$ if, given the values for the parameters in $\mathcal{P}_f$, we can select values for the parameters in $\mathcal{P}_t$ such that $x^*\ge0$ and $y^*=r$ where $x^*$ is an equilibrium point for the  dynamics \eqref{main:eq:RRE}. The admissible set for the output $\Yz$  is the set of its admissible values.
\end{define}

\blue{The motivation for dividing the parameter set \( \mathcal{P} \) into two subsets, \( \mathcal{P}_t \) and \( \mathcal{P}_f \), is to distinguish between parameters that can be actuated and those that are inherent to the network's structure. In this context, the admissibility of a chosen set-point is a necessary condition for the proper functioning of the controlled network around that set-point. Setting a non-admissible set-point effectively asks the system to reach an unattainable value, potentially leading to unstable or diverging state trajectories.}\\

\blue{Set-point admissibility depends on the network's structure, the specific output considered, and the choice of tunable parameters. Selecting different tunable parameters results in different admissible sets for the output. This is illustrated in Box 1, where we examine a gene expression network first with the transcription rate as the tunable parameter and then with the protein degradation rate. In this paper, we use set-point admissibility as a preliminary criterion for deciding which parameters should be used to actuate the network. Once these parameters are selected, we can analyze the network's properties and propose a suitable controller structure to solve the perfect adaptation design problem.}\\

\blue{Finally, although set-point admissibility is essential for the perfect adaptation design problem, it is meaningless without stability. Set-point admissibility only indicates the possibility of attaining a certain value for the controlled species at steady state. For this to be meaningful, the state of the controlled network must also converge to that steady state. Without convergence, the system may exhibit diverging or oscillatory trajectories, rendering the concept of perfect adaptation meaningless, as no actual adaptation occurs in the controlled network. Establishing the conditions for stability and structural stability is therefore the main goal of this study.}\\

\subsection*{Antithetic Integral (Rein) Controllers}

The Antithetic Integral Controller (AIC), introduced in \cite{Briat:15e} and further discussed in \cite{Briat:19:Logistic,Olsman:19,Olsman:19b}, is one of the few, structurally simple, integral controllers that can be implemented in terms of chemical reactions and that can operate in both the deterministic and the stochastic settings. Other integral controllers of course exist, such as zeroth-order integral controllers \cite{Ni:09} and autocatalytic integral controllers \cite{Drengstig:12,Briat:16a,Xiao:18,Briat:19:Logistic}, but are unfortunately only functional in the deterministic setting. The AIC relies on molecular sequestration as core principle, which consists of two complementary molecular species strongly sequestering each other. Sequestration is a well-known mechanism that has been naturally selected to achieve certain pivotal functions within living organisms, such as bacterial stress response through the use of sigma-factors \cite{Trevino:13}. On a more theoretical level, the sequestration reaction plays an essential role in the solution of realization problems \cite{Oishi:10b,Fages:17} as it can be interpreted as a molecular implementation of a subtraction operator \cite{Briat:15e,Briat:19:Logistic,Cuba:23,Filo:23b}. The internal bimolecular structure of the AIC makes it a very flexible controller that can easily implement negative and positive feedback loops using both activation or inhibition interactions, as illustrated in Figure~\ref{main:fig:AICtable} in Box 1. More advanced structures built upon the AIC can be generated through a clever addition of reactions and species in order to achieve more complex behaviors and performance requirements; see e.g. \cite{Frei:22,Filo:22,Anastassov:23}. The common denominator to all those designs is that the controlled species should repress itself through a sequence of reactions involving controller species, a mechanism that makes the AIC a member of the family of negative feedback controllers.\\

A particular extension of the AIC is the so-called Antithetic Integral Rein Controller (AIRC), which exploits the bimolecular structure of the AIC to implement a second actuation channel acting in opposition to the first one. The concept of rein control has been introduced by Saunders in \cite{Saunders:98} as a central mechanism for glucose regulation and was shown to have enhanced robustness properties over a unidirectional control mechanism. Therefore, it is expected that the AIRC benefits from the same advantages over the AIC which only acts on the controlled species in a unidirectional manner and may offer limited performance in some scenarios. As shown in  Figure~\ref{main:AIRC:AIRC}, the AIRC admits multiple possible configurations depending on the role of the controller species (activator or inhibitor) and the reaction paths between the actuated and the controlled species. As AIRCs can be interpreted as the superposition of two AICs, the same rules apply and one must select roles for the controller species in a way that makes the controller a negative feedback controller.\\

The aiAIRC with output inhibition, shown in Figure~\ref{main:AIRC:aiAIRC}, is a special case of the aiAIRC, which consist of the superposition of an nAIC and an niAIC with output inhibition. This structure was considered in \cite{Gupta:19} in the special case of gene expression control, where some of its properties were discussed. The objective of this paper is to address  this problem in a much more general setting.

\begin{figure}[H]
     \centering
     \begin{subfigure}[b]{0.55\textwidth}
         \centering
    \includegraphics[width=\textwidth]{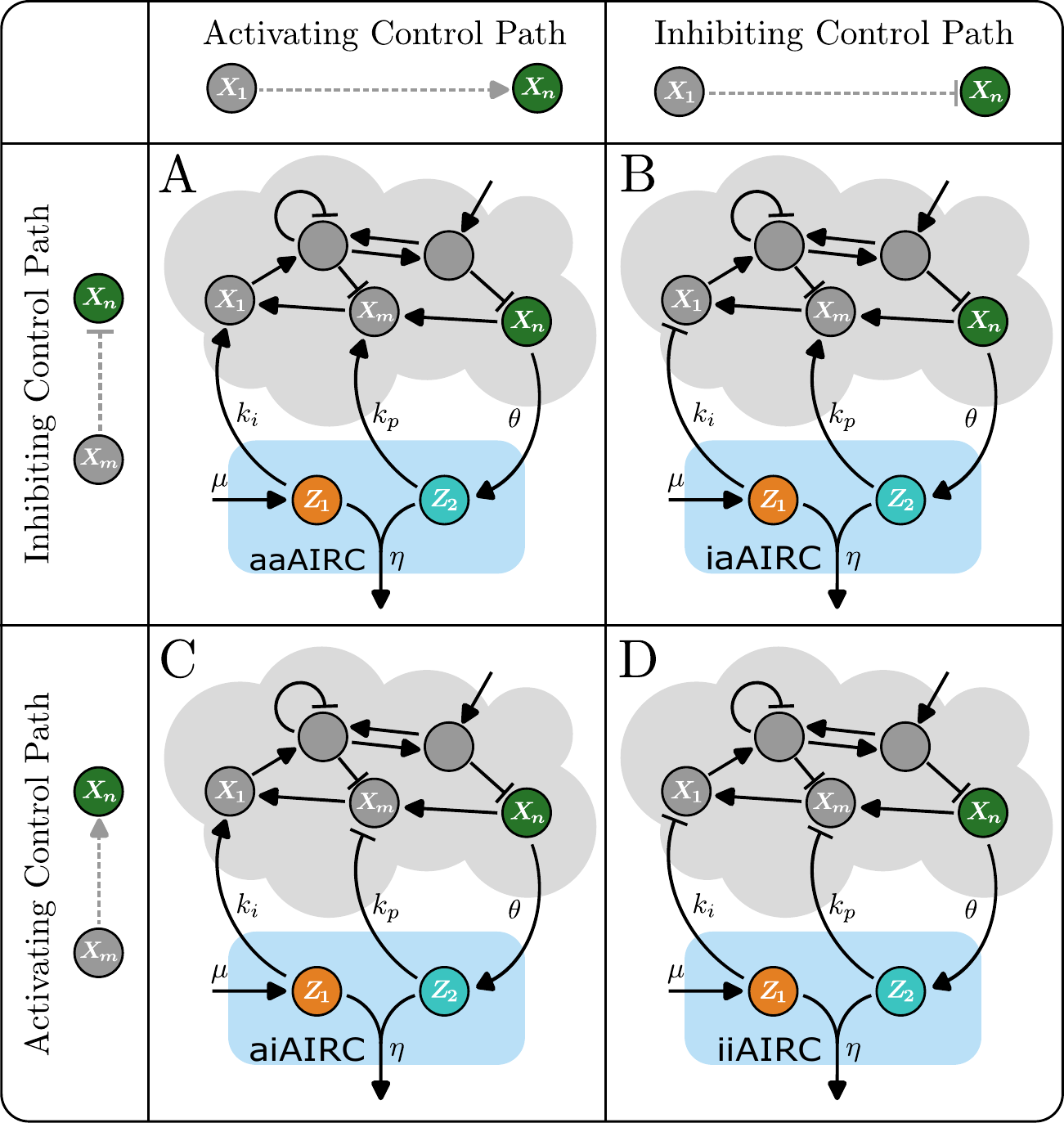}
    \caption{Example of possible AIRCs.}\label{main:AIRC:AIRC}
     \end{subfigure}
     \hfill
     \begin{subfigure}[b]{0.4\textwidth}
         \centering
         \includegraphics[width=\textwidth]{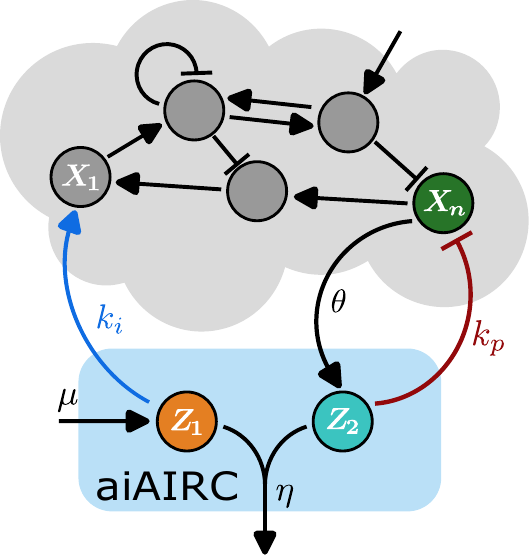}
    \caption{The aiAIRC with output inhibition.}\label{main:AIRC:aiAIRC}
     \end{subfigure}
     \caption{(a). Different possible structures of the AIRC depending on whether the controller species $\Z{1}$ and $\Z{2}$ act as activators or inhibitors on the actuated species $\X{1}$ and $\X{m}$ and whether the actuated species act as activators or inhibitors for the controlled species $\X{n}$. In the first column, $\X{1}$ is an activator for $\X{n}$ whereas it is an inhibitor in the second one. Similarly, in the first row $\X{m}$ is an inhibitor of $\X{n}$ where it is an activator in the second one. The roles of the controller species is chosen so that the overall feedback loop is a negative feedback loop and have opposite action on the controlled species. For instance, when $\X{1}$ activates $\X{n}$ and $\X{m}$ inhibits it, a suitable AIRC is given by the aaAIRC, since that $\X{n}\regulationarrow[act]\Z{2}\regulationarrow[rep]\Z{1}\regulationarrow[act]\X{1}\regulationarrow[act]\X{n}$ and $\X{n}\regulationarrow[act]\Z{2}\regulationarrow[act]\X{m}\regulationarrow[rep]\X{n}$, showing that $\Z{1}$ acts as an activator and $\Z{2}$ as a repressor, as required to make the topology a rein controller. One can also observe that the two feedback loops implemented by the aaAIRC are negative feedback loops. Analogous conclusions hold for all the topologies. (b) The aiAIRC with output inhibition consists of two AICs: an naAIC  that activates the production of $\X{1}$ (blue pointed arrow) and an niAIC that directly inhibits the output $\X{n}$ (red blunt head arrow).}
\end{figure}

\section{Results}

\subsection*{The aiAIRC with output inhibition allows for an arbitrary set-point.}

As discussed in Box 1, the naAIC and niAIC induce mutually exclusive set-point admissibility conditions in line with their activatory or inhibitory nature and in the unidirectionality of their actuation mechanism. The intuition is that those mutually exclusive conditions could be perhaps made compatible through the consideration of a combination of those controllers, that is through the consideration of an aiAIRC with output actuation described by the reaction network

\begin{equation}\label{main:eq:AIC:aiAIRCoi}
  \phib\rarrow{\mu}\Z{1},\ \Yz\rarrow{\theta}\Yz+\Z{2},\ \Z{1}+\Z{2}\rarrow{\eta}\phib,\  \Z{1}\rarrow{k_i}\Z{1}+\X{1},\ \X{n}+\Z{2}\rarrow{k_p}\Z{2}
\end{equation}
where the reaction rates $\mu,\theta,\eta,k_p,k_i>0$ are the parameters of the controller and are assumed to be freely adjustable. It is assumed here that the controlled species is $\X{n}$. The closed-loop network consisting of the interconnection of a unimolecular mass-action network of the form \eqref{main:eq:RN} and the aiAIRC \eqref{main:eq:AIC:aiAIRCoi} is described by the dynamical system
\begin{equation}\label{main:eq:mainsystCL}
  \begin{array}{rcl}
    \dot{x}(t)&=&Ax(t)+e_1k_izu_1(t)-e_n  x_n(t) k_pz_2(t)+b_0\\
    \dot{z}_1(t)&=&\mu-\eta z_1(t)z_2(t)\\
    \dot{z}_2(t)&=&\theta x_n (t)-\eta z_1(t)z_2(t)
  \end{array}
\end{equation}
where $x\in\mathbb{R}_{\ge0}^n$ is the vector of concentrations of the species of the network whereas $z_1,z_2\in\mathbb{R}_{\ge0}$ are the concentrations of the species of the controller. The matrix $A$ is the matrix describing the dynamics of the network to be controlled whereas $b_0$ is the vector of basal rates \blue{which are obtained from applying \eqref{main:eq:RRE} to the case of unimolecular networks, that is, we have $\textstyle\sum_{k=1}^K\zeta_k\lambda_k(x)=:Ax+b_0$.} Moreover, in that case, the set of all fixed parameters $\cal{P}_f$ are those in the matrix $A$ and the vector $b_0$ whereas the set of tunable parameters are those of the controller network, that is, $\mathcal{P}_t=\{k_i,k_p,\mu,\theta,\eta\}$. For the sake of convenience, we define $r:=\mu/\theta$ which, as we shall see later, coincides with the set-point. The vector $e_i$ is the vector of all zeros except at the index $i$ where the entry is equal to one. \blue{Practically speaking, requiring that the network is unimolecular and mass-action is equivalent to saying that the dynamics of the network is linear. While their consideration may appear limited in scope, they have the benefits of being simple to introduce and to allow for a direct, explicit mathematical analysis which may be extrapolated to the more general nonlinear setting. More details about this are given in Section S4 of the SI. }\\

The following result is a simplified version of Proposition \ref{prop:eqpt} of the SI:
%
%\begin{proposition}\label{main:prop:eqpt}
%  Assume that $A$ is Hurwitz stable\footnote{A matrix is Hurwitz stable if all its eigenvalues have negative real part.} and that $-e_n ^TA^{-1}e_1\ne0$. Then, the equilibrium point $(x^*,z_1^*,z_2^*)$ of the closed-loop system \eqref{main:eq:mainsystCL} is unique and nonnegative. It is, moreover, given by
%  \begin{equation}
%    \begin{array}{rcl}
%      x^*&=&-A^{-1}\left(e_1k_iz_1^*-\dfrac{e_n  r k_p\mu}{\eta z_1^*}+b_0\right)\\
%      z_2^*&=&\dfrac{\mu}{\eta z_1^*}
%    \end{array}
%  \end{equation}
%  where $z_1^*$ is the unique positive root to the polynomial
%  \begin{equation}\label{main:eq:P}
%  P_1(z_1):=\eta g_1 k_iz_1^{2}+(g_0-r)\eta z_1-g_n  k_p\mu r,
%\end{equation}
%  where $r:=\mu/\theta$, $g_1:=-e_n ^TA^{-1}e_1$, $g_n :=-e_n ^TA^{-1}e_n $, and $g_0:=-e_n ^TA^{-1}b_0$.
%\end{proposition}
\begin{proposition}\label{main:prop:eqpt}
  Assume that $A$ is Hurwitz stable\footnote{A matrix is Hurwitz stable if all its eigenvalues have negative real part.} and that $e_n ^TA^{-1}e_1\ne0$. Then, the equilibrium point $(x^*,z_1^*,z_2^*)$ of the closed-loop system \eqref{main:eq:mainsystCL} is unique, nonnegative and such that $x_n^*=r:=\mu/\theta$.
\end{proposition}
%
%\red{This result states that under the assumptions of the result on the matrix $A$, all set-points $r>0$ are admissible for the system \eqref{main:eq:mainsystCL} -- which shows that our initial intuition is verified -- as the assumptions are independent of the set-point value $r$ and that this set-point value can be set through an appropriate choice of the tunable parameters, here $\mu$ and $\theta$. The assumption that $A$ is Hurwitz stable is equivalent to saying that the solution of the uncontrolled network dynamics $\dot{x}=Ax+b_0$ is stable and globally converging to its unique nonnegative equilibrium point $-A^{-1}b_0$. While such a behavior is not uncommon in practice, some networks may exhibit sustained oscillations and, in such cases, the proposed controller structure may not be adequate for solving the associated perfect adaptation problem. The above result can be adapted to address the more general case where $A$ is not Hurwitz stable, as stated in Proposition \ref{prop:eqpt} of the SI.}
\blue{This result indicates that under the assumptions regarding the matrix \( A \), all set-points \( r > 0 \) are admissible for the system \eqref{main:eq:mainsystCL}. This confirms our initial intuition, as these assumptions are independent of the set-point value \( r \), and this value can be set through an appropriate choice of the tunable parameters \( \mu \) and \( \theta \). The assumption that \( A \) is Hurwitz stable is equivalent to stating that the solution of the uncontrolled network dynamics \( \dot{x} = A x + b_0 \) is stable and globally converges to its unique nonnegative equilibrium point \( -A^{-1} b_0 \). While such behavior is common in practice, some networks may exhibit sustained oscillations. In such cases, the proposed controller structure may not be adequate for solving the associated perfect adaptation problem. However, the above result can be adapted to address the more general case where \( A \) is not Hurwitz stable, as stated in Proposition \ref{prop:eqpt} of the SI.}

\subsection*{The aiAIRC with output inhibition behaves as a filtered proportional-integral controller}

This controller locally acts as a filtered Proportional-Integral (PI) controller consisting of the combination of a proportional action, which depends directly on the output/error, and an integral action, which depends on the integral of the output/error. To show this, consider the linearized dynamics of \eqref{main:eq:mainsystCL} about the unique equilibrium point
\begin{equation}\label{main:eq:mainsystCL:linear}
\begin{array}{rcl}
  \dot{\tilde{x}}(t)&=&A\tilde{x}(t)-e_n k_pz_2^*\tilde{y}(t)+e_1k_i\tilde{z}_1(t)-e_n  k_pr \tilde{z}_2(t)\\
  \dot{\tilde{z}}_1(t)&=&-\eta z_2^*\tilde{z}_1(t)-\eta z_1^*\tilde{z}_2(t)\\
  \dot{\tilde{z}}_2(t)&=&\theta \tilde{y}(t)-\eta z_2^*\tilde{z}_1(t)-\eta z_1^*\tilde{z}_2(t)
\end{array}
\end{equation}
where $\tilde{x}=x-x^*$, $\tilde{z}_1=z_1-z_1^*$, and $\tilde{z}_2=z_2-z_2^*$. One can observe that the controller acts at three levels: the first one lies at the level of the dynamics of the output through the term \blue{$-e_nk_pz_2^*\tilde{y}(t)$}, the second one at the level of the production of the species $\X{1}$ through the term $e_1k_i\tilde{z}_1(t)$, and the last one at the level of the degradation of the output through the term $-e_n  k_pr \tilde{z}_2(t)$. The first term corresponds the proportional action as it depends proportionally on the output, the second and the third terms correspond to two integral actions with the difference that the former is activating whereas the latter is inhibiting. The presence of the PI action can be further emphasized by formulating the overall system into the interconnection of the network and the controller in the Laplace domain. In this domain, the network can be represented as a map from the input $\widehat{u}(s)$ to the controlled output $\widehat{x}_n(s)$ as
\begin{equation}
  \widehat{x}_n(s)=\underbrace{e_n^T(sI-A)\begin{bmatrix}
  e_1 & -e_n
\end{bmatrix}}_{\mbox{Network}}\widehat{u}(s),
\end{equation}
while the control input can be expressed as
\begin{equation}\label{eq:PILaplace}
  \widehat{u}(s)=\left(\underbrace{\begin{bmatrix}
  0\\
  -k_pz_2^*
\end{bmatrix}}_{\mbox{Proportional Action}}+\underbrace{\dfrac{1}{(s+\eta(z_1^*+z_2^*))}\dfrac{1}{s}\begin{bmatrix}
  -k_i\eta\theta z_1^*\\
  \theta k_pr(s+\eta z_2^*)
\end{bmatrix}}_{\mbox{Filtered Integral Action}}\right)\widehat{x}_n(s).
\end{equation}
%\red{As indicated above, the controller is composed of two components. The first one is the proportional component that depends linearly on the output $\X{n}$, and is only present in the second component of the control input, that is, the channel that acts on the degradation of the controlled species. This term is absent in the naAIC architecture discussed in Box 1 and in previous works on the integral control of reaction networks \cite{Briat:15e,Briat:19:Logistic}. The second term is a low-pass filtered integral term which consists of an integration operation -- corresponding to the term $1/s$ -- in series with a low pass-filter operation - corresponding to the term $1/(s+\eta(z_1^*+z_2^*))$. The presence of the low-pass filter does not negatively impact the integral action and zero steady-state error will be preserved. Based on the above expression, we can also observe that the filtered PI structure is preserved as long as $k_p\ne0$. When this is not the case, the controller becomes an integral controller of the form naAIC, as already discussed in \cite{Briat:15e,Briat:19:Logistic} and in Box 1. When $k_i=0$, on the other hand, the controller loses its naAIC arm b but preserves an integral action through the arm associated with the niAIC.}
\blue{As previously mentioned, the controller comprises two components. The first is a proportional component that depends linearly on the output species \( \X{n} \) and is present only in the second part of the control input -- the channel that influences the degradation of the controlled species. Notably, this proportional term is absent in the naAIC architecture discussed in Box 1 and in earlier works on integral control of reaction networks \cite{Briat:15e,Briat:19:Logistic}. The second component is a low-pass filtered integral term, consisting of an integration operation -- represented by \( 1/s \), which is the Laplace transform of the integration operation  -- followed by a low-pass filter operation, corresponding to \( 1/(s + \eta(z_1^* + z_2^*)) \). The inclusion of the low-pass filter does not adversely affect the integral action, so zero steady-state error and, therefore, perfect adaptation are maintained provided that the dynamics of the controlled network is stable. We can observe that the filtered PI structure is preserved as long as \( k_p \ne 0 \). If \( k_p = 0 \), the controller reduces to an antithetic integral controller of the naAIC form, as discussed in \cite{Briat:15e,Briat:19:Logistic} and in Box 1. Conversely, when \( k_i = 0 \), the controller loses its naAIC arm but retains integral action through the arm associated with the niAIC, therefore preserving the perfect adaptation property under the assumption that the controlled network is stable.}

\subsection*{The aiAIRC with output inhibition has a switching behavior in the strong sequestration regime.}

A successful technique in the analysis of antithetic structures consists of letting the sequestration rate go to infinity, a regime we call here the \emph{strong sequestration regime}. This procedure allows one to perform model reduction yielding a simplification of the dynamical and equilibrium equations from which insights could be more easily drawn from. In fact, it can be shown that in this strong sequestration regime yields the expressions $z_1(t)\approx \max\{0,I(t)\}$, $z_2(t)\approx -\min\{0,I(t)\}$, and $\dot{I}(t)=\mu-\theta y(t)$ in which we can observe  the presence of the integral action.\\

However, we are more interested here in the behavior of the equilibrium point of the closed-loop network \eqref{main:eq:mainsystCL} in that regime. The following result, which is a simplified version of Proposition \ref{prop:switching} in the SI, provides the description of this behavior:
\begin{proposition}\label{main:prop:switching}
  Assume that $A$ is Hurwitz stable, then the basal output expression level $g_0$ for the system is given by $g_0=-e_n^TA^{-1}b_0$. Assuming further $e_n ^TA^{-1}e_1\ne0$ yields the following statements for the equilibrium point $(x^*,z_1^*,z_2^*)$ of the closed-loop system \eqref{main:eq:mainsystCL}:
  \begin{enumerate}
  \item If the set-point $r$ is larger than the basal output expression $g_0$, then $(z_1^*,z_2^*)\rarrowl{\eta\to\infty}(u_1^*/k_i,0)$ where $u_1^*=(r-g_0)/g_1$.
  \item If the set-point $r$ is smaller than the basal output expression $g_0$, then $(z_1^*,z_2^*) \rarrowl{\eta\to\infty} (0,u_2^*/k_p)$ where $u_2^*=(g_0-r)/(g_n  r)$.
  %
%  \item If $r=g_0$, then
%  \begin{equation}
%    z_1^*=\sqrt{\dfrac{g_n  k_p\mu r}{\eta g_1k_i}}\textnormal{ and }z_2^*=\sqrt{\dfrac{\theta g_1k_i}{\eta g_n  k_p}}
%  \end{equation}
%  and they both tend to 0 when $\eta\to\infty$ while keeping the product $\eta z_1^*z_2^*$ constant and equal to $\mu$.
\end{enumerate}
\end{proposition}

This result can easily be interpreted as follows: when the set-point $r$ is larger than the basal level $g_0$, only the naAIC component of the controller is active whereas in the opposite scenario only the niAIC component is active. As already mentioned, the naAIC has been extensively studied and the case where $r>g_0$ can be considered now as well-understood. However, the niAIC structure is much less understood and clarifying its properties will be one of the objectives of this paper.

\subsection*{The niAIC with output inhibition structurally stabilizes stable unimolecular mass-action networks under some set-point admissibility condition}

The niAIC with output inhibition is described by the following reaction network
\begin{equation}\label{main:eq:AIC:niAICoi}
  \phib\rarrow{\mu}\Z{1},\ \Yz\rarrow{\theta}\Yz+\Z{2},\ \Z{1}+\Z{2}\rarrow{\eta}\phib,\ \X{n}+\Z{2}\rarrow{k_p}\Z{2},
\end{equation}
which yields the following model for the closed-loop network  \eqref{main:eq:RN}-\eqref{main:eq:AIC:niAICoi}
\begin{equation}\label{main:eq:mainsystCL2}
  \begin{array}{rcl}
    \dot{x}(t)&=&Ax(t)-e_n  x_n (t) k_pz_2(t)+b_0\\
    \dot{z}_1(t)&=&\mu-k_p\eta z_1(t)z_2(t)\\
    \dot{z}_2(t)&=&\theta x_n (t)-k_p\eta z_1(t)z_2(t)
  \end{array}
\end{equation}
where we have changed $\eta$ into $\eta k_p$. This latter modification does not change the nature of the results but dramatically simplifies their derivation \cite{Briat:19:Logistic}. The first part of the solution to the perfect adaptation problem stated in Definition \ref{main:def:RPAP} consists of establishing conditions under which a set-point is admissible. This is formulated in the following result:
\begin{proposition}\label{main:prop:eqpt}
 Assume that $A$ is Metzler\footnote{A real, square matrix is Metzler if its off-diagonal entries are nonnegative.} and Hurwitz stable\footnote{A square matrix is Hurwitz stable if its eigenvalues have negative real part.}, then the equilibrium point $(x^*,z_1^*,z_2^*)$ of the closed-loop network \eqref{main:eq:mainsystCL2}
  %given by
%  \begin{equation}\label{main:eq:eqpt}
%    (x^*,z_1^*,z_2^*)=\begin{pmatrix}
%      -A^{-1}(-e_nru_*+b_0), & \dfrac{\mu}{\eta u_*}, & \dfrac{u_*}{k_p}
%    \end{pmatrix},\ u_*=\dfrac{g_0-r}{g_n  r}
%  \end{equation}
  is unique, nonnegative and such that $x_n^*=r$ if and only if $r<g_0$.
 \end{proposition}

\blue{As previously mentioned, assuming that \( A \) is Hurwitz stable is equivalent to stating that the trajectories of the uncontrolled network globally converge to a unique equilibrium point. While this is a simplifying assumption to ease the exposition of the main results, Theorem \ref{lem:positiveunstable} in the SI demonstrates that it can be relaxed to accommodate cases where \( A \) may exhibit a certain forms of instability.}\\

The second part of the solution to the perfect adaptation problem consists of establishing conditions under which the equilibrium point $(x^*,z_1^*,z_2^*)$ of the closed-loop network \eqref{main:eq:mainsystCL2} is locally asymptotically stable. This can be achieved through the study of the asymptotic stability of the linearized system
 \begin{equation}\label{main:eq:mainsystCL2:linear}
  \begin{bmatrix}
    \dot{\tilde{x}}(t)\\
    \dot{\tilde{z}}_1(t)\\
    \dot{\tilde{z}}_2(t)
  \end{bmatrix}=\underbrace{\begin{bmatrix}
    \bar{A} & 0 & -e_n  k_pr\\
    0 & -\eta u_* & -\mu k_p/u_*\\
    \theta e_n ^T & -\eta u_* & -\mu k_p/u_*
  \end{bmatrix}}_{\mbox{$A_\ell$}}\begin{bmatrix}
    \tilde{x}(t)\\
    \tilde{z}_1(t)\\
    \tilde{z}_2(t)
  \end{bmatrix}
\end{equation}
where $\bar A:=A-e_n  e_n ^Tu_*$, $\tilde{x}=x-x^*$, $\tilde{z}_1=z_1-z_1^*$, and $\tilde{z}_2=z_2-z_2^*$. This linearized system is asymptotically stable if and only if all the eigenvalues of the matrix $A_\ell$ have negative real parts. In order to formulate conditions in terms of  the data of the problem (i.e. the matrix $A$ of the network to be controlled) and the tunable controller parameters $\mu,\theta,k_p,\eta>0$ under which the linearized dynamics is asymptotically stable, we rewrite the system as the interconnection depicted in Figure \ref{main:fig:NyqInt}.B where
\begin{equation}
  H_1(s)=\dfrac{k_p}{s},\quad \textnormal{and}\quad H_2(s)= e_n^T(sI-\bar A)^{-1}e_n\mu+\dfrac{\mu s}{u_*(s+\eta u_*)}.
\end{equation}
Using the concepts and tools described in Box 2 yields the following result
\begin{theorem}\label{main:th:structstab1}
  Assume that $\mu/\theta<g_0$ and that the matrix $A$ is Metzler and Hurwitz stable. Then, the unique equilibrium point $(x^*,z_1^*,z_2^*)$ of the system \eqref{main:eq:mainsystCL2} lies in the nonnegative orthant, is such that $x_n^*=\mu/\theta$,  and is locally exponentially stable for all $\eta,k_p>0$.
\end{theorem}

This result means that all state trajectories will converge to the equilibrium point provided that they start sufficiently close from it and that the conditions of the result are met. The fact that the conditions do not depend on the parameters  $\eta,k_p>0$ makes the implementation of the controller an easier task as fine tuning parameters in while implementation is usually not an option in synthetic biology. Note that those parameters do still have an impact on the dynamics of the system \eqref{main:eq:mainsystCL2} and its quantitative properties such as speed of convergence, oscillations, overshoot, etc.\\

As an illustrative example, consider the following unimolecular mass-action gene expression network with protein maturation and positive feedback:
  \begin{equation}\label{main:eq:RN:maturation}
  \begin{matrix}
    \phib&\rarrow{k_0}&\X{1} & \X{1}&\rarrow{k_{21}}&\X{1}+\X{2} & \X{i}&\rarrow{\gamma_i}\phib,i=1,2,3\\
    \X{2}&\rarrow{k_{32}}&\X{3}& \X{3}&\rarrow{k_{13}}&\X{1}+\X{3}& \X{3}&\rarrow{\nu}\X{3}+\X{3},
  \end{matrix}
  \end{equation}
  where $\X{1},\X{2}$, and $\X{3}$ are the mRNA, protein, and matured protein species, respectively, and where the parameters of the reactions of positive real numbers, which are in the set of fixed parameters $\mathcal{P}_f$. We assume here that we would like to control the maturated species, that is, $\Yz=\X{3}$. The model of the system is given by
\begin{equation}\label{main:eq:RN:maturation:model}
\begin{bmatrix}
  \dot{x}_1\\
    \dot{x}_2\\
    \dot{x}_3
\end{bmatrix}=\begin{bmatrix}
    -\gamma_1 & 0 & k_{13}\\
    k_{21} & -\gamma_2 & 0\\
    0 & k_{32} & \nu-\gamma_3
  \end{bmatrix}\begin{bmatrix}
    x_1\\
    x_2\\
    x_3
  \end{bmatrix}+\begin{bmatrix}
    k_0\\
    0\\
    0
  \end{bmatrix}
\end{equation}
where $x_1,x_2,x_3$ are the mRNA, protein, and matured protein concentrations, respectively. The matrix describing the dynamics of the network is Metzler, by construction, and is Hurwitz stable provided that $\gamma_1\gamma_2\gamma_3-k_{13}k_{32}k_{21}>0$ (Routh-Hurwitz criterion). We also have that
  \begin{equation}
  g_0=\dfrac{k_0k_{21}k_{32}}{\gamma_1\gamma_2(\gamma_3-\nu)-k_{21}k_{32}k_{13}}>0.
\end{equation}
By virtue of Theorem \ref{main:th:structstab1}, assuming that the network \eqref{main:eq:RN:maturation} is stable, we can conclude that the closed-loop network  \eqref{main:eq:RN:maturation}, \eqref{main:eq:AIC:niAICoi}  is locally exponentially stable for all $0<\mu/\theta<g_0$ and all $k_p>0,\eta>0$.

%\begin{figure}[H]
%  \centering
%  \includegraphics[width=0.50\textwidth]{./Matlab/pAIC_Unimolecular_Gene_expression_maturation_stable/pAIC_stable_heatmap.pdf}
%  \caption{Spectral abscissa of the system associated with the \eqref{main:eq:RN:maturation}, \eqref{main:eq:AIC:niAIC_output} with the parameters $\gamma_1=1$, $\gamma_2=\gamma_3=2$, $k_{21}=1$, $k_{32}=2$, $b_0=10$, $k_{13}=1$, $\eta=100$, and $\theta=1$ and for various values for $\mu$ and $k_p$. In this case, we have that $g_0=10$. Simple calculations show that the spectral abscissa of $A$, is $\alpha(A)=-0.3044$ while the spectral abscissa of the closed-loop system may reach smaller values, which indicates that this controller is able to improve the convergence properties of the system near the equilibrium point. }
%\end{figure}

\begin{figure}[H]
  \centering
\begin{minipage}[h]{0.45\linewidth}
    \begin{center}
  \includegraphics[width=\textwidth]{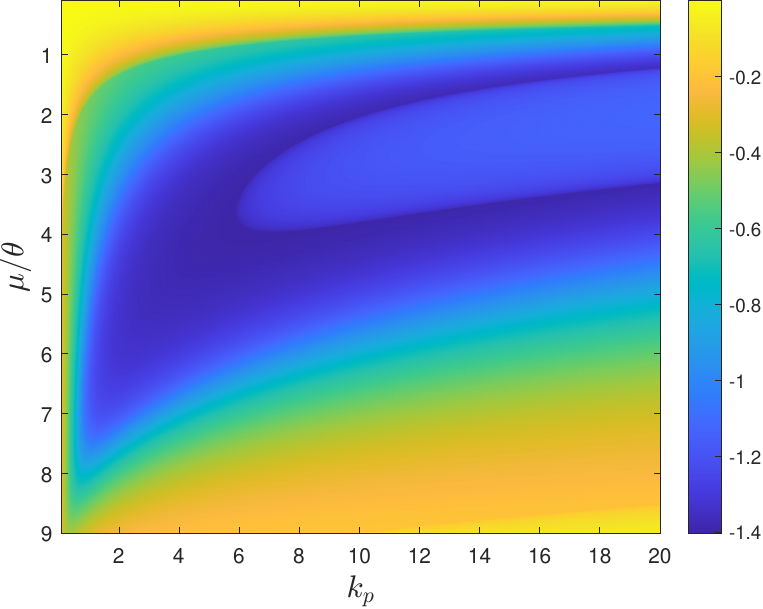}
  \end{center}
\end{minipage}
\begin{minipage}[h]{0.45\linewidth}
    \begin{center}
  \includegraphics[width=\textwidth]{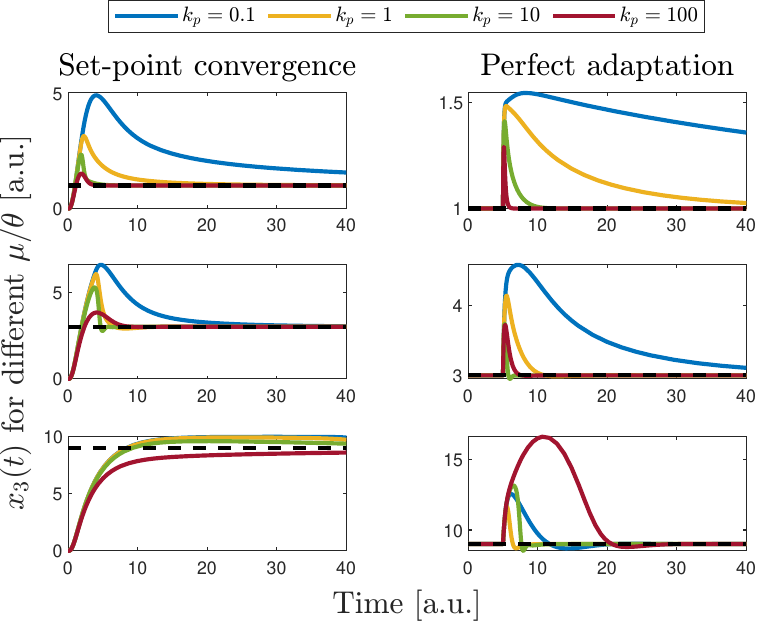}
  \end{center}
\end{minipage}
\caption{\textbf{Left.} Spectral abscissa (i.e. real-part of the rightmost eigenvalue) of the system associated with the \eqref{main:eq:RN:maturation}, \eqref{main:eq:AIC:niAICoi} with the parameters $\gamma_1=1$, $\gamma_2=\gamma_3=2$, $k_{21}=1$, $k_{32}=2$, $k_0=10$, $k_{13}=1$, $\nu=0$, $\eta=100/k_p$, and $\theta=1$ and for various values for $\mu$ and $k_p$. We have that $g_0=10$ and calculations show that the spectral abscissa of $A$ is $-0.3044$ while the spectral abscissa of the closed-loop system may reach smaller values, which indicates that this controller is able to improve the convergence properties of the system near the equilibrium point. \textbf{Right.} Time domain evolution of the concentrations of $\X{3}$ for various values for the set-point $\mu/\theta$ and controller gains $k_p$. The left column depicts simulation results for zero initial conditions (convergence properties) whereas the right column depicts the response of the closed-loop network when the parameter $k_{32}$ changes from 2 to 3 at $t=5$ (perfect adaptation property).}
\end{figure}

\subsection*{The niAIC with output inhibition structurally stabilizes a class of unstable unimolecular mass-action networks under no set-point restriction}

%\red{Put reasoning behind relaxing the last species, and also mention that the only stability of the stability of the upper-left part is required in the proof o the main result}

In the previous section, we considered the case where the matrix $A$ describes the dynamics of a stable network. In the present section, we relax this assumption by allowing a more general case where the network is allowed to be unstable in a very specific way, a property that we call "output instability". \blue{While the precise definition of output instability provided in the SI is rather technical, it can be intuitively motivated by analyzing the structure of the niAIC. Since the niAIC acts on the network by actively degrading the controlled species, it seems natural to believe that this controller could still achieve its goal even when the controlled species exhibits unstable behavior. Output instability of the matrix \( A \) is therefore a rigorous characterization of those matrices where only the controlled species is allowed to be unstable -- hence the term "output unstable." The rest of this section demonstrates that relaxing the stability of the network is not only possible but also leads to surprising results, notably regarding set-point admissibility, as shown below:}

%\red{While the actual definition of output instability, given in the SI, is rather technical, it can be motivated by a rather intuitive analysis of the structure of the niAIC. As the niAIC acts on the network by actively degrading the controlled species, then it seems natural to believe that this controller would still achieve its goal even when the controlled species exhibits an unstable behavior. Output instability of the matrix $A$ is, therefore, a rigorous characterization of all those matrices $A$ in which only the controlled species is allowed to have an unstable behavior. whence the name "output unstable". The rest of this section demonstrates that relaxing the stability of the network is not only possible but also leads to uprising results, notably regarding set-point admissibility, as shown below:}

\begin{proposition}\label{main:prop:eqpt2}
Assume that $A$ is Metzler and output unstable, then the equilibrium point $(x^*,z_1^*,z_2^*)$ of the closed-loop network \eqref{main:eq:mainsystCL2} is unique, nonnegative and such that $x_n^*=r$ if and only if $g_0\ne0$.
 \end{proposition}

We can observe here that, by relaxing the stability of the network, we allow for more flexibility at the level of the set-point, which is a dramatic change over the previous case for which the set-point value is more restricted. The reason for the emergence of this property is that, as opposed to the stable case, there is no equilibrium output basal level for the controlled species. In this regard, one can just let the output spontaneously increase beyond the set-point value before starting degrading it and making it converge to the desired set-point.\\

Having a suitable equilibrium point is only one part of the solution of the perfect adaptation problem. The second one is the stability of that equilibrium point, which is addressed in the result below:
\begin{theorem}\label{main:th:main:unstable}
  Assume that $A$ is Metzler and output unstable, and that $g_0\ne 0$. Then, the unique equilibrium point  $(x^*,z_1^*,z_2^*)$ of the closed-loop network \eqref{main:eq:mainsystCL2} lies in the nonnegative orthant, is such that $x_n^*=\mu/\theta$, and is locally exponentially stable for all $\eta,k_p,\mu,\theta>0$.
\end{theorem}
This result tells us that all state trajectories will converge to the equilibrium point provided that they start sufficiently close from it and that the conditions of the result are met. In the present case, the stability of the equilibrium point holds unconditionally of the controller parameters, which a drastic change over Theorem \ref{main:th:structstab1} where the set-point value is required to be smaller than the output basal expression level. \blue{In this case, the controller functions both as a stabilizer for the system dynamics and as a regulator for the controlled species. Specifically, the stabilization effect arises from the proportional action in the controller, as described in equation  \eqref{eq:PILaplace}, while the regulation effect is due to the filtered integral action. This dual role is consistent with standard results in the literature on PID control \cite{Astrom:95}.}\\

%\red{Therefore, in this case, the controller acts both as a stabilizer for the dynamics and as a regulator for the controlled species. In fact, it can be shown that the stabilization effect is due to the proportional action in the controller, as described in \eqref{eq:PILaplace}, while the regulation effect is due to the (filtered) integral action of the controller. This is in agreement with standard results in the literature on PID control \cite{Astrom:95}.}\\

To illustrate this discussion, consider back the network \eqref{main:eq:RN:maturation} with the model \eqref{main:eq:RN:maturation:model}, which can be shown to be output unstable if and only if
\begin{equation}
  \gamma_1\gamma_2(\nu -\gamma_3)+k_{21}k_{32}k_{13}>0.
\end{equation}
We also have the following expression for $g_0$
\begin{equation}
  g_0=\dfrac{-k_0k_{21}k_{32}}{\gamma_1\gamma_2(\nu-\gamma_3)+k_{21}k_{32}k_{13}}<0,
\end{equation}
and observe that it is negative. This is due to the fact that the network is output unstable, and that this value does not represent the output basal expression level anymore.\\

By virtue of Theorem \eqref{main:th:main:unstable}, assuming that the network \eqref{main:eq:RN:maturation} is output unstable, we can conclude that the closed-loop network \eqref{main:eq:RN:maturation}, \eqref{main:eq:AIC:niAICoi} is locally exponentially stable for all $\mu,\theta,k_p>0,\eta>0$. This statement is illustrated by simulation in Figure \ref{fig:persistent:1} and Figure \ref{fig:persistent:2}.

\begin{figure}[H]
  \centering
  \includegraphics[width=\textwidth]{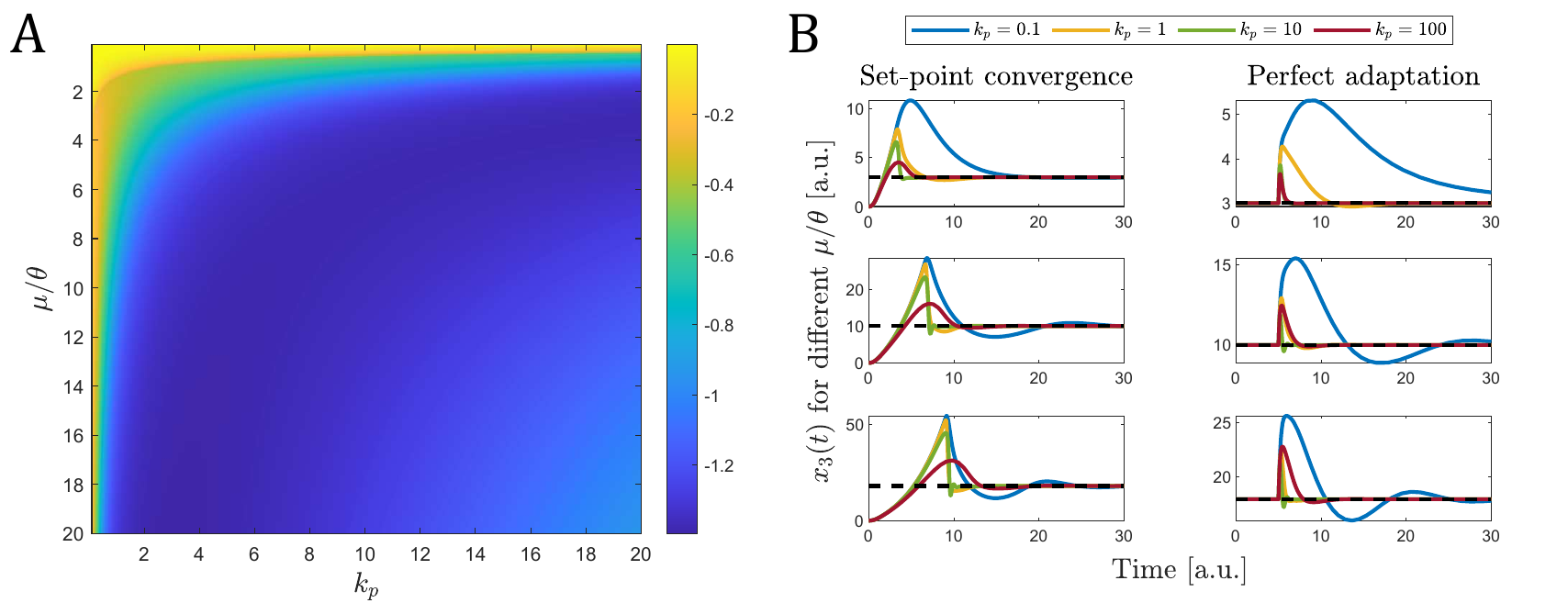}
%  \begin{minipage}[h]{0.45\linewidth}
%    \begin{center}
%    \includegraphics[width=\textwidth]{./Matlab/pAIC_Unimolecular_Gene_expression_maturation_unstable/niAIC_unstable_heatmap_rotated_0.pdf}
%  \end{center}
%\end{minipage}
%\begin{minipage}[h]{0.45\linewidth}
%    \begin{center}
%  \includegraphics[width=\textwidth]{./Matlab/pAIC_Unimolecular_Gene_expression_maturation_unstable/niAIC_unstable_trajectories_test_0.pdf}
%  \end{center}
%\end{minipage}
%
\caption{\textbf{A.} Spectral abscissa (i.e. real-part of the rightmost eigenvalue) of the system associated with the \eqref{main:eq:RN:maturation}, \eqref{main:eq:AIC:niAICoi} with the parameters $\gamma_1=1$, $\gamma_2=\gamma_3=2$, $k_{21}=1$, $k_{32}=2$, $k_0=10$, $k_{13}=3$, $\nu=0$, $\eta=100/k_p$, and $\theta=1$ and for various values for $\mu$ and $k_p$. We have that $g_0=-10$, $g_n=-1$, and calculations show that the spectral abscissa of $A$ is $0.2188$ while the spectral abscissa of the closed-loop system may reach smaller values, which indicates that this controller is able to stabilize the network and improve the convergence properties of the system near the equilibrium point. \textbf{B.} Time domain evolution of the concentrations of $\X{3}$ for various values for the set-point $\mu/\theta$ and controller gains $k_p$. The left column depicts simulation results for zero initial conditions (convergence properties) whereas the right column depicts the response of the closed-loop network when the parameter $k_{32}$ changes from 2 to 3 at $t=5$ (perfect adaptation property).}\label{fig:persistent:1}
\end{figure}

\begin{figure}[H]
  \centering
    \includegraphics[width=\textwidth]{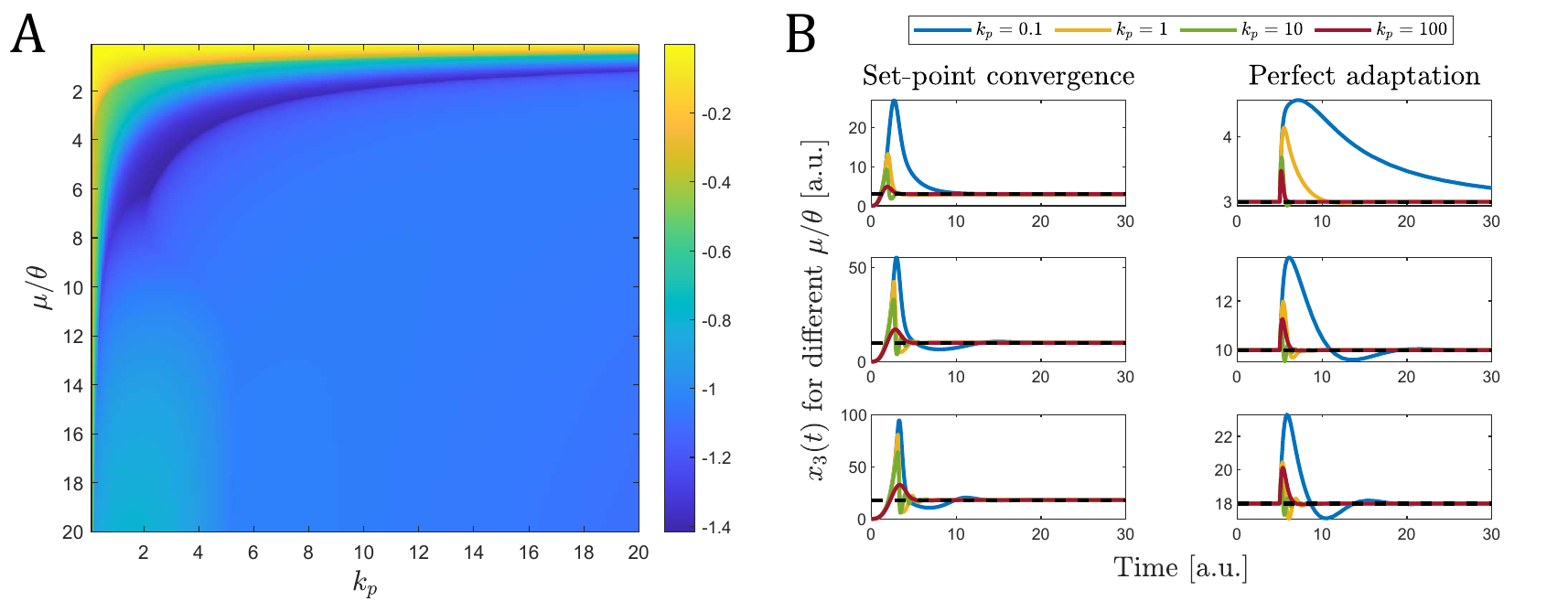}
    %\begin{minipage}[h]{0.45\linewidth}
%    \begin{center}
%   \includegraphics[width=\textwidth]{./Matlab/pAIC_Unimolecular_Gene_expression_maturation_unstable/niAIC_unstable_heatmap_rotated_1.pdf}
%  \end{center}
%\end{minipage}
%\begin{minipage}[h]{0.45\linewidth}
%    \begin{center}
%  \includegraphics[width=\textwidth]{./Matlab/pAIC_Unimolecular_Gene_expression_maturation_unstable/niAIC_unstable_trajectories_test_1.pdf}
%  \end{center}
%\end{minipage}
%
\caption{\textbf{A.} Spectral abscissa (i.e. real-part of the rightmost eigenvalue) of the system associated with the \eqref{main:eq:RN:maturation}, \eqref{main:eq:AIC:niAICoi} with the parameters $\gamma_1=1$, $\gamma_2=\gamma_3=2$, $k_{21}=1$, $k_{32}=2$, $k_0=10$, $k_{13}=1$, $\nu=3$, $\eta=100/k_p$, and $\theta=1$ and for various values for $\mu$ and $k_p$. We have that $g_0=-5$, $g_n=-1/2$, and calculations show that the spectral abscissa of $A$ is $1.2695$ while the spectral abscissa of the closed-loop system may reach smaller values, which indicates that this controller is able to stabilize the network and improve the convergence properties of the system near the equilibrium point. \textbf{B.} Time domain evolution of the concentrations of $\X{3}$ for various values for the set-point $\mu/\theta$ and controller gains $k_p$. The left column depicts simulation results for zero initial conditions (convergence properties) whereas the right column depicts the response of the closed-loop network when the parameter $k_{32}$ changes from 2 to 3 at $t=5$ (perfect adaptation property).}\label{fig:persistent:2}
\end{figure}

\subsection*{How restrictive are the stability conditions on the unimolecular networks?}
%
%\blue{A natural question to ask how strong are the assumptions on the considered class of unimolecular networks. As shown in the above theoretical results and simulations, imposing $A$ to be Hurwitz stable is clearly a restrictive assumption, which was later relaxed to considering matrices $A$ which are output-unstable. The next question is, therefore, how restrictive is this assumption?\\
%
%A requirement for the structural stability with respect to the controller gain $k_p$ is that all the non-actuated species $\X{2},\ldots,\X{n})$ be stable. This may also be a requirement for stability and even for the mere existence of a nonnegative equilibrium point for the controlled network. The requirement of the natural stability of the non-actuated species can be easily explained by first noting that the actuation can only degrade the controlled species while, at the same time, the controlled species can only be an activator for the other species. This means that if one of the actuated species is not stable, then there is no possible way it could be stabilized though the degradation of the controlled species alone.}\\

\blue{A natural question is how strong the assumptions are on the considered class of unimolecular networks. As shown in the theoretical results and simulations above, imposing that \( A \) is Hurwitz stable is clearly a restrictive assumption, which was later relaxed to considering matrices \( A \) that are output-unstable. Therefore, the next question is, how restrictive is this assumption?}\\

\blue{As proved in Section \ref{sec:uni:niAIC:stable} and Section \ref{sec:uni:niAIC:unstable}, structural stability with respect to the controller gain \( k_p \) requires that all non-actuated species \( X_2, \ldots, X_n \) be stable. Nonnegative equilibrium points as well as an unstable behavior for the controlled network may also result from this assumption being not satisfied. The necessity of the natural stability of the non-actuated species can be easily explained by noting that the actuation affects only the degradation of the controlled species, while the controlled species can only act as an activator for the other species. This means that if any of the non-actuated species is unstable, there is no way to stabilize it through the degradation of the controlled species alone.}\\

%
%the submatrix of $A$ associated with all the species of the network except the controlled species
%
%
%In the first case, it is assumed that $A$ is Hurwitz stable while this is relaxed to the case of output-unstable matrices in the second case. We illustrate here that when those conditions are violated, either nonnegativity of the equilibrium point or, when it is indeed nonnegative, then its stability is lost.
%
%The stability of the species except the last one $\X{n}$ is sort of natural to require. Indeed, as the controller can only degrade the controlled species, and that the controlled species only act catalytically on the other species (that is, activation), then the only way for the controlled network to be stable is that those species exhibit a natural stability behavior.

\blue{We consider now two simple case scenarios that illustrate the previously evoked technical problems. The first one is that of a unimolecular network described by the matrices
\begin{equation}
  A=\begin{bmatrix}
    1 & 0\\1 & -1
  \end{bmatrix},\ b_0=\begin{bmatrix}
    1\\0
  \end{bmatrix}.
    \end{equation}
This matrix is neither Hurwitz nor output-unstable as the first molecular species is unstable. The equilibrium concentrations for the species of the network controlled using an niAIC with output-inhibition is given by
\begin{equation}
 x^*=-(A-k_pz_2^*e_ne_n^T)^{-1}b_0=\begin{bmatrix}
    -1\\
-\dfrac{1}{k_pz_2^* + 1}
  \end{bmatrix}
\end{equation}
and unavoidably contains a negative equilibrium for the species $\X{1}$, which is not admissible. This means that the controlled network will necessarily be unstable.}\\

\blue{For the second scenario, we slightly change $b_0$ to $b_0=\begin{bmatrix}
  0\\1
\end{bmatrix}$, which yields the (now admissible) equilibrium point
\begin{equation}
    x^*=\begin{bmatrix}
    0\\
    \dfrac{1}{k_pz_2^* + 1}
  \end{bmatrix}.
\end{equation}
Setting then $z_2^*:=\dfrac{1}{k_p(1/r-1)}$ yields $x_2^*=r$, $z_1^*=\mu/(\eta k_pz_2^*)$, and the Jacobian matrix
\begin{equation}
    J=\begin{bmatrix}
  1 & 0 & 0 & 0\\
  1 & -\dfrac{1}{r} & 0 & -k_pr\\
  0 & 0 & -\eta\left(\dfrac{1}{r}-1\right) & \dfrac{k_p\mu r}{r-1}\\
  0 & \theta  & -\eta\left(\dfrac{1}{r}-1\right) & \dfrac{k_p\mu r}{r-1}
  \end{bmatrix},
\end{equation}
which is not Hurwitz stable because of the presence of an eigenvalue equal to one. This proves that the unique equilibrium point of the controlled network is unstable in this situation. }\\

\blue{Those two examples illustrate the two problems that may arise in whenever the matrix $A$ is neither Hurwitz nor output-unstable: absence of a nonnegative equilibrium point or instability of the nonnegative equilibrium point.}

\subsection*{Extension to the  aiAIRC with output inhibition.}

We may now ask the question of how this translates to the aiAIRC controller with output inhibition, as depicted in Figure \ref{main:AIRC:AIRC}. This is, in fact, true and stated in the following result:
\begin{theorem}
  Assume that $A$ is Metzler, output unstable, and nonsingular and that $g_0\ne 0$. Then, the unique equilibrium point of the system \eqref{main:eq:mainsystCL} is locally exponentially stable for all $\eta,k_p,\mu,\theta>0$ and all $k_i\in[0,\bar k_i)$ for some $\bar k_i>0$.
\end{theorem}

\blue{This result indicates that, within the context of the aiAIRC, the structural stability of the controlled network is preserved with respect to the parameters \( \eta, k_p, \mu, \theta > 0 \). However, the result does not address structural stability concerning the gain \( k_i \) of the naAIC arm of the aiAIRC. At first glance, this may seem conservative, as it is possible that the network is also structurally stable with respect to \( k_i > 0 \) in certain cases. This issue lies beyond the scope of this paper because the mathematical tools developed in the SI do not apply here due to technical reasons that we will not discuss here. This exciting open problem is left for future research.}\\

\blue{We can, however, illustrate the above result and our claim about the structural stability of the network with respect to \( k_i \) through numerical simulations depicted in Figure~\ref{fig:AIRCstructural}. Consistent with the developed theory, the first panel shows that for the controlled network to be stable when \( r > g_0 \), it is necessary for \( k_i \) to be positive. This implies that the naAIC component of the aiAIRC is essential for the stability of the controlled network, aligning with previous works on the topic \cite{Briat:15e,Briat:19:Logistic}. Similarly, the third panel demonstrates that when \( r < g_0 \), a positive \( k_p \) is required for stability, indicating that the niAIC component is crucial. This finding is consistent with the results previously discussed in this paper.}\\

\blue{In the special case when \( r = g_0 \), both arms of the aiAIRC are essential to ensure the stability of the controlled network. This scenario is particularly interesting because this equilibrium point is unstable for both the naAIC and niAIC controller networks individually. Stability can only be achieved by combining both controllers into the unified aiAIRC structure.}\\

\blue{In all these scenarios, we observe that the controlled network remains stable for values of both \( k_p \) and \( k_i \) that are significantly greater than zero. Although not displayed in Figure~\ref{fig:AIRCstructural}, this stability persists even for much larger parameter values. This suggests that the controlled network is structurally stable with respect to the controller parameters \( k_p \) and \( k_i \).}

\begin{figure}[H]
  \centering
  \includegraphics[width=0.45\textwidth]{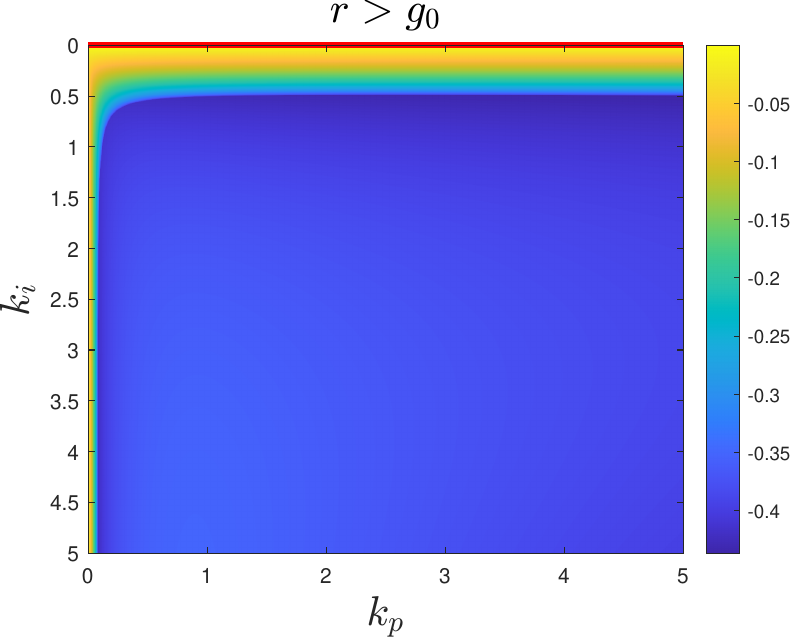}\includegraphics[width=0.45\textwidth]{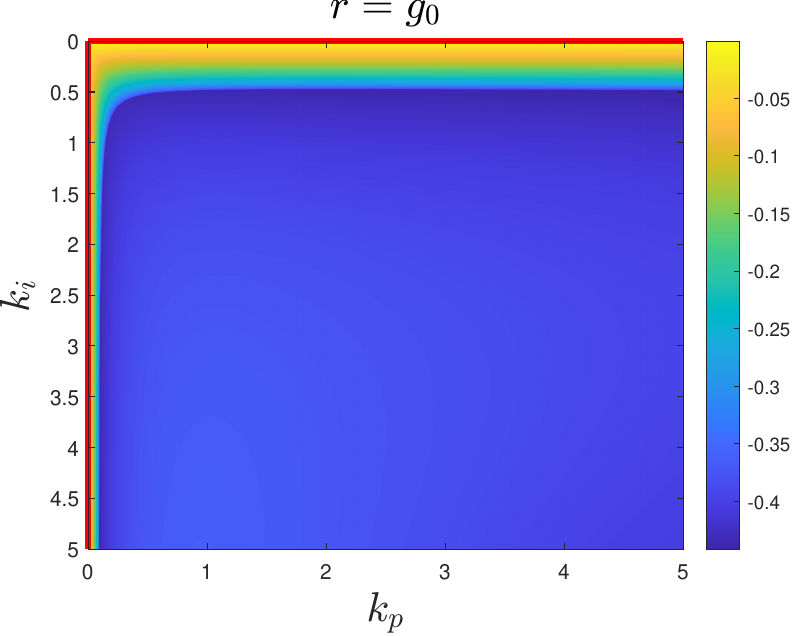}\\
  \includegraphics[width=0.45\textwidth]{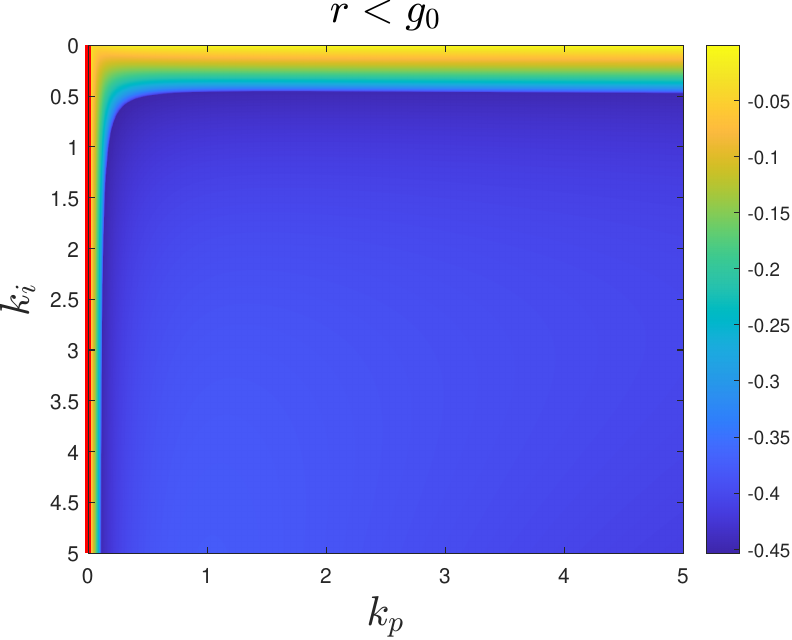}
%  \caption{\red{Stability regions in the $(k_p,k_i)$-plane of the (unique) equilibrium point of the system \eqref{main:eq:mainsystCL} associated with the gene expression network with protein maturation \eqref{main:eq:RN:maturation} controlled with the aiAIRC with output-inhibition \eqref{main:eq:AIC:aiAIRCoi}. The parameters common to all simulation panels are $\gamma_1=1$, $\gamma_2=\gamma_3=2$, $k_{21}=1$, $k_{32}=1$, $k_{13}=0$, $\nu=0$, $\eta=100$, $\theta=1$ and $\mu=2.5$. We have $k_0=5$, which corresponds to $g_0=1.25$, so $r>g_0$ in the first panel, $k_0=10$, which corresponds to $g_0=2.5$ and, thus $r=g_0$, in the second panel, and $k_0=15$, which corresponds to $g_0=3.75$ and, thus $r<g_0$, in the third panel. The red region exclusively lying on the boundary of the nonnegative orthant represents the region in which the system does not have any nonnegative equilibrium point (i.e. lack of admissibility for the set-point) or when the unique nonnegative equilibrium point is unstable when it exists.}}\label{fig:AIRCstructural}
\caption{\blue{Stability regions in the \((k_p, k_i)\)-plane for the unique equilibrium point of system \eqref{main:eq:mainsystCL}, associated with the gene expression network with protein maturation  \eqref{main:eq:RN:maturation} controlled by the aiAIRC with output inhibition \eqref{main:eq:AIC:aiAIRCoi}. Common parameters for all simulation panels are: \(\gamma_1 = 1\); \(\gamma_2 = \gamma_3 = 2\); \(k_{21} = 1\); \(k_{32} = 1\); \(k_{13} = 0\); \(\nu = 0\); \(\eta = 100\); \(\theta = 1\); and \(\mu = 2.5\). In the first panel, \(k_0 = 5\) (which gives \(g_0 = 1.25\)), so \(r > g_0\); in the second panel, \(k_0 = 10\) (yielding \(g_0 = 2.5\)), so \(r = g_0\); and in the third panel, \(k_0 = 15\) (resulting in \(g_0 = 3.75\)), so \(r < g_0\). The red regions along the boundary of the nonnegative orthant represent areas where the system either lacks a nonnegative equilibrium point (i.e., the set-point is not admissible) or where the unique nonnegative equilibrium point exists but is unstable.}}\label{fig:AIRCstructural}
\end{figure}

\subsection*{General networks}

We consider now the general case consisting of the interconnection of the network \eqref{main:eq:RN} and the niAIC with output inhibition \eqref{main:eq:AIC:niAICoi}, which can be described by the dynamical model
\begin{equation}\label{main:eq:NLCL}
\begin{array}{rcl}
  \dot{x}(t)&=&f(x(t))-e_nk_pz_2(t)x_n(t)+b_0\\%,\ x(0)=x_0\\
   \dot{z}_1(t)&=&\mu-\eta k_pz_1(t)z_2(t)\\%,\ z_1(0)=z_{0,1}\\
   \dot{z}_2(t)&=&\theta e_n^Tx-\eta k_pz_1(t)z_2(t),%,\ z_2(0)=z_{0,2}
\end{array}
\end{equation}
where the function $f(x)$ is sufficiently regular so that solutions to this model exist. In the current setting, this function will contain mass-action, Hill, or Michaelis-Menten terms.\\

Due to the nonlinear nature of the network, there are a number of additional difficulties over the unimolecular/linear case that need to be addressed. We briefly explain them here and refer the readers to the SI for more details. The first difficulty concerns the analytical calculation of the admissibility set and, while it may still be possible to characterize it analytically, numerical methods will be needed to compute the equilibrium solutions to the above system in most cases. So, one cannot expect to have a general, clean expression for the admissibility set in this case. The second difficulty is that there may be multiple equilibrium points $(x^*(r),z_1^*(r),z_2^*(r))$ associated with a given admissible set-point value $r$. To avoid complications, we assume here that this is not the case and that the set-point $r$ uniquely defines the equilibrium point $(x^*(r),z_1^*(r),z_2^*(r))$. The third difficulty lies at the level of the static input-output function  $F:u\mapsto y$, which maps input values to equilibrium output values. This function is much more difficult to explicitly characterize than in the unimolecular/linear case and one has to rely on its implicit definition given by the set of algebraic equations $f(x)-e_nuy+b_0=0$, $y=x_n$, and $y=F(u)$. In fact, a solution to this system is not even guaranteed to exist and additional conditions on the function $f$ are required to ensure that this is the case. Finally, the last difficulty is that the function $F$ needs to be locally decreasing at the equilibrium points of interests for them to be potentially stable. This corresponds to having a "negative gain", meaning that increasing the input decreases the output of the system.\\

Assuming that the above considerations have been adequately addressed and that the corresponding conditions/assumptions are satisfied, the stability of the equilibrium point associated with the set-point $r$  can then be established through the analysis of the linearized dynamics
\begin{equation}\label{main:eq:mainsystCL:nonlinear}
  \begin{bmatrix}
    \dot{\tilde{x}}(t)\\
    \dot{\tilde{z}}_1(t)\\
    \dot{\tilde{z}}_2(t)
  \end{bmatrix}=\begin{bmatrix}
    J^*(r)-e_n  e_n ^Tu_*(r) & 0& -e_n  k_pr\\
    0 & -\eta u_*(r) & -k_p\mu/u_*(r)\\
    \theta e_n ^T & -\eta u_*(r) & -k_p\mu/u_*(r)
  \end{bmatrix}\begin{bmatrix}
    \tilde{x}(t)\\
    \tilde{z}_1(t)\\
    \tilde{z}_2(t)
  \end{bmatrix}
\end{equation}
where $J^*(r)$  is the Jacobian matrix of the system, defined as $\textstyle J^*(r):=\left.\frac{\partial f(x)}{\partial x}\right|_{x=x^*(r)}$, and $u_*(r)$ is such that $r=F(u_*(r))$. The matrix $J^*(r)$ is the nonlinear analogue of the $A$ matrix in \eqref{main:eq:mainsystCL} and \eqref{main:eq:mainsystCL2:linear}, with the striking difference that the local dynamics is now influenced by the set-point. In this regard, different properties for the system are expected depending on our choice for the set-point value $r$. From the linear dynamics \eqref{main:eq:mainsystCL:nonlinear}, we can define the local transfer function for the nonlinear network
  \begin{equation}
  H_n(s,r)=e_n^T(sI-(J^*(r)-u_*(r)e_ne_n^T))^{-1}e_n,
\end{equation}
which plays an analogous role as in the linear/unimolecular case, with the difference now that the transfer function is set-point-dependent. This transfer function describes the behavior of the system about the corresponding equilibrium point. Interestingly, the DC-gain $H(0,r)$ is related to the static-output map $F$ as $\textstyle H(0,r)=-\left.\frac{dF(u)}{du}\right|_{u=u_*(r)}$ meaning that the DC-gain is positive if and only if the input-output map $F$ is decreasing at $u_*(r)$.\\

With the above discussion in mind, we are now in position to state the main result of this section on general nonlinear networks:
\begin{theorem}\label{main:th:generalNL}
Suppose that the assumptions and conditions discussed above and in the SI are satisfied, and assume further that
\begin{enumerate}[(a)]
  \item the set-point $r$ is admissible,
  \item the equilibrium $(x^*(r),u_*(r))$ is uniquely defined by the set-point $r$,
  %\item $F$ is monotonically decreasing,
  \item $H_n(0,r)>0$, and
  \item one of the following equivalent statements holds:
  \begin{enumerate}[({d}1)]
    \item There exist a symmetric positive definite matrix $P_1(r)$ and a scalar $\eps>0$ such that
    \begin{equation}\label{main:eq:LMI:NL}
        (J^*(r)-F^{-1}(r)e_ne_n^T)^TP(r)+P(r)(J^*(r)-F^{-1}(r)e_ne_n^T)+2\eps e_ne_n^T
    \end{equation}
    is negative definite where
    \begin{equation}\label{eq:Pr}
        P(r) = \begin{bmatrix}
          P_1(r) & 0\\
          0 & 1
        \end{bmatrix}.
    \end{equation}
    \item There exists a scalar$\eps>0$ such that the system $(J_{11}^*(r), J_{12}^*(r), -J_{21}^*(r), u_*(r)-J_{22}^*(r)-\eps)$ is strictly positive real.
  \end{enumerate}
\end{enumerate}
  Then, the unique equilibrium point $(x^*(r),z_1^*(r),z_2^*(r))$ of the system \eqref{main:eq:NLCL} is locally exponentially stable for all $\eta,k_p>0$.
\end{theorem}

Due to the generality of the problem, it is difficult to provide more explicit conditions than those ones and we provide below a computational procedure to verify them. The admissibility of a set-point (i.e. condition (a)) as well as the uniqueness of the associated equilibrium point (i.e. condition (b)) can be both established by numerically solving the expressions $f(x)-e_nk_prz_2+b_0=0$ and $x_n=r$ for $(x,z_2)$. If there is no nonnegative solution, then the set-point is not admissible, and if there are multiple nonnegative solutions, then the uniqueness of the equilibrium point does not hold. This provides a way to check the two first conditions. The Jacobian matrix $J(x)$ can be computed using symbolic calculations, from which a simple evaluation at $x=x^*(r)$ would yield $J^*(r)$. The transfer function $H(s,r)$ can be easily evaluated once $J^*(r)$ and $u_*(r)$ are known in order to check condition (c). Finally, condition (d1) is a so-called Linear Matrix Inequality (LMI) problem, a convex decision problem, which can be solved using freely available solvers \cite{Sturm:01a,Tutuncu:03}. The condition (d2) can be numerically verified by computing the poles and the zeros of the associated transfer function and checking whether the Nyquist diagram remains in the open right half-plane according to the definitions and results stated in Box 2.\\

%Condition (c) can be checked by numerically evaluating the function $F$ at the points $\{u^1,\ldots,u^N\}\subset\mathbb{R}_{\ge0}$ by solving for $(x^i,y^i)$ in $f(x^i)-e_nu^iy+b_0=0$ and $y^i=x_n^i$. This gives an evaluation of the function $F$ at those points, i.e. $y^i=F(u^i)$, from which we can verify its monotonic behavior, at least in the range of values we are interested in. Although not exact, this method is simple and can be made as accurate as needed by increasing the number of evaluation points, at the expense of a higher computational cost.

We illustrate the above result by considering the network
  \begin{equation}\label{main:eq:RN:NL}
  \begin{matrix}
    \phib&\rarrow{k_0}&\X{1} & \X{1}&\rarrow{k_{21}}&\X{1}+\X{2} & \X{i}&\rarrow{\gamma}\phib,i=1,2\\
    \X{2}&\rarrow{h(x)}&\X{2}+\X{1}& \X{2}&\rarrow{u}&\phib
  \end{matrix}
  \end{equation}
  where the last reaction is here to represent the actuation reaction with input $u$. All reactions are mass-action with positive rates except for the catalytic reaction $\X{2}\rarrow{h(x)}\X{2}+\X{1}$ which has a propensity function given by $h(x)=k_{12}/(1+x_2)$. Those parameters are all in the set of fixed parameters $\mathcal{P}_f$. The goal here is to control the second species, that is, $\Yz=\X{2}$. As a result, the model of this system is given by
  \begin{equation}\label{main:eq:ex:feedback}
    \begin{array}{rcl}
              \dot{x}_1&=&-\gamma x_1+\dfrac{k_{12}}{1+x_2}+k_0,\\%\qquad       \dot{x}_2=kx_1-\gamma x_2-u
              \dot{x}_2&=&k_{21}x_1-\gamma x_2-ux_2
    \end{array}
  \end{equation}
  where $x_1$ and $x_2$ denote the mRNA and protein concentrations, respectively. It turns out that the admissibility set can be explicitly computed and is given by $(0,r_\mn)$ where $r_{\mn}$ is the unique positive root of the polynomial
\begin{equation}
  \gamma^2r^2+r(\gamma^2-k_{21}k_0)-k_{21}(k_{12}+k_0).
\end{equation}
This also shows that the equilibrium point is unique. The Jacobian of this system is given by
  \begin{equation}
    J(x)=\begin{bmatrix}
        -\gamma & -\dfrac{k_{12}}{(1+x_2)^2}\\
      k_{21} & -\gamma
    \end{bmatrix}
  \end{equation}
  and is invertible for all $x_2\ge0$, and so is $J(x)-ue_ne_n^T$ for all $x\ge0$ and $u\ge0$. It can be shown that the function $F$ is monotonically decreasing, hence invertible.  Therefore, we have that $F(0)=r_{\mn}$ and $\lim_{u\to\infty}F(u)=0$. More details can be found in the SI.\\

The Jacobian evaluated at the equilibrium point corresponding to $x_2^*=r$ is given by
  \begin{equation*}
    J^*(r)=\begin{bmatrix}
      -\gamma & -\dfrac{k_{12}}{(1+r)^2}\\
      k_{21} & -\gamma
    \end{bmatrix}
  \end{equation*}
  and the associated transfer function is $$H_n(s,r)=\dfrac{s+\gamma}{s^2+(2\gamma+u_*(r))s+\gamma(\gamma+u_*(r))+k_{21}k_{12}/(1+r)^2}.$$ From the Routh-Hurwitz criterion, the transfer function $H_n(s,r)$ is stable, has stable zeros and verifies $H(0,r)>0$. Therefore, we just need to show the existence of a matrix $P(r)$ as defined in Theorem \ref{main:th:generalNL} such that
  \begin{equation*}
    \begin{bmatrix}
    -2\gamma P(r) & -\dfrac{k_{12}}{(1+r)^2}P(r)+k_{21}\\
    -\dfrac{k_{12}}{(1+r)^2}P(r)+k_{21} & -2(\gamma+u_*(r))
  \end{bmatrix}
  \end{equation*}
  is negative definite. Alternatively, we may check whether the transfer function associated with the system
\begin{equation}
  \begin{array}{rcl}
    \dot{v}&=&-\gamma v-\dfrac{k_{12}}{(1+r)^2}w\\
    z&=&-k_{21}v+(\gamma+u_*(r))w
  \end{array}
\end{equation}
is strictly positive real. It can be seen that this is the case since it is a first order stable transfer function of relative degree 0 with stable zeros, with positive gain and positive feedthrough. Therefore, the closed-loop system consisting of \eqref{main:eq:RN:NL} and the niAIC with output inhibition \eqref{main:eq:AIC:niAICoi} is stable for all $k_p,\eta>0$ and $r>r_{\mn}$.\\

\begin{figure}[H]
  \centering
%  \begin{minipage}[h]{0.45\linewidth}
%    \begin{center}
%    \includegraphics[width=\textwidth]{./Matlab/Nonlinear/niAIC_nonlinear_heatmap_rotated.pdf}
%  \end{center}
%\end{minipage}
%\begin{minipage}[h]{0.45\linewidth}
%    \begin{center}
%  \includegraphics[width=\textwidth]{./Matlab/Nonlinear/niAIC_nonlinear_trajectories.pdf}
%  \end{center}
\includegraphics[width=\textwidth]{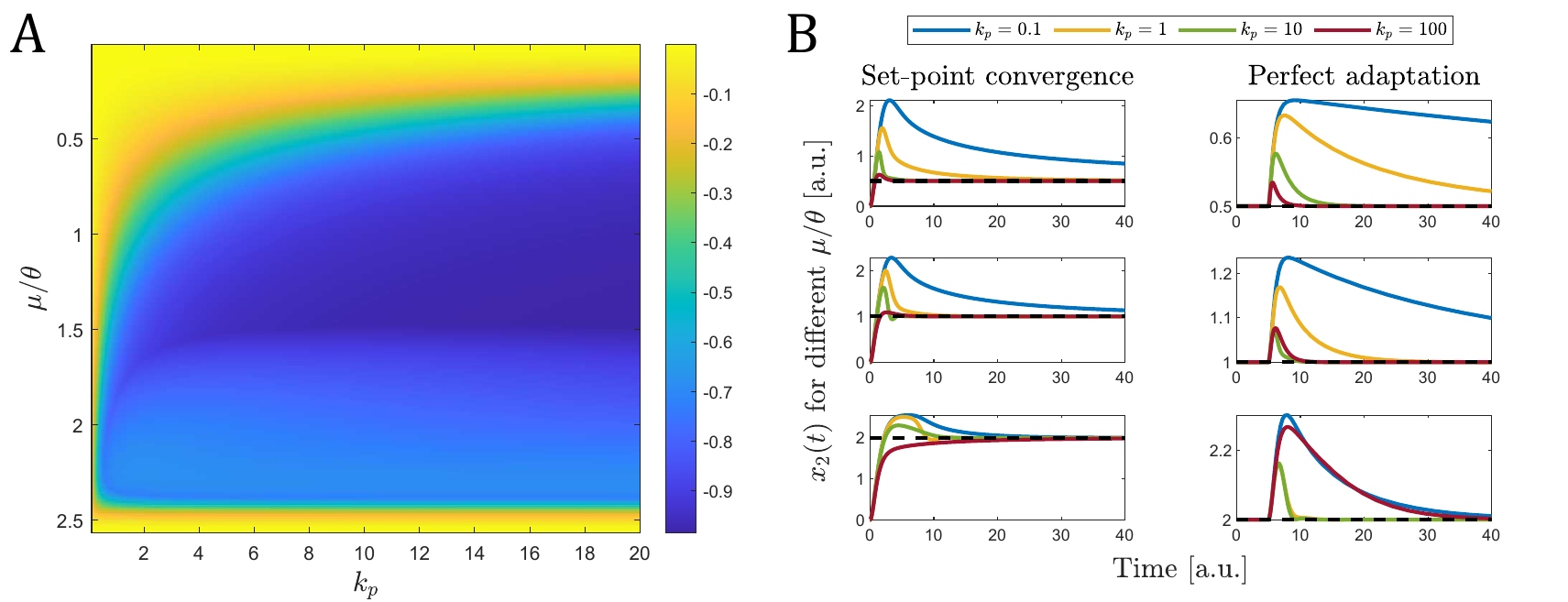}
\caption{\textbf{A.} Spectral abscissa (i.e. real-part of the rightmost eigenvalue) of the system associated with the network \eqref{main:eq:RN:maturation}-\eqref{main:eq:AIC:niAICoi} with the parameters $\gamma_1=1$, $k_{12}=1$, $k_0=1$, $k_{21}=2$, $\eta=100/k_p$, and $\theta=1$ and for various values for $\mu$ and $k_p$. We have that $r_{\mn}=2.5616$. \textbf{B.} Time domain evolution of the concentrations of $\X{2}$ for various values for the set-point $\mu/\theta$ and controller gains $k_p$. The left column depicts simulation results for zero initial conditions (convergence properties) whereas the right column depicts the response of the closed-loop network when the parameter $k_{12}$ changes from 1 to 2 at $t=5$ (perfect adaptation property).}\label{fig:nonlinear:1}
\end{figure}

This simple example shows how the conditions of Theorem \ref{main:th:generalNL} can be checked and emphasizes how difficult this verification could get when the dimension of the problem grows. However, some of those conditions dramatically simplify when considering specific subclasses of networks. This is discussed in the following sections.

\subsection*{The niAIC with output inhibition structurally stabilizes stable cooperative networks under some set-point admissibility condition}

The class of cooperative networks is a class of networks for which the Jacobian matrix $J^*(r)$ is Metzler for some admissible set-points $r$, which means that the network locally behaves like a unimolecular network at that point. For all those points $r$ where this is the case, we say that the system is cooperative at $x^*(r)$. This leads us to the following result
\begin{proposition}\label{prop:cooperative}
  Let $r$ be admissible and assume that the system \eqref{main:eq:NLCL} is cooperative at the equilibrium value $x^*(r)$ , assumed to be unique, and that the matrix $J^*(r)$ is Hurwitz stable. Then, the unique equilibrium point $(x^*(r),z_1^*(r),z_2^*(r))$ of the cooperative system \eqref{main:eq:NLCL} is locally exponentially stable for all $\eta,k_p>0$.
\end{proposition}

As the class of cooperative networks is very similar to unimolecular ones, which are cooperative by construction, extensions to the case of output-unstable systems are rather immediate at the expense of additional notational burden stemming from the nonlinear nature of the network. For this reason, those results are omitted here. Another possible extension of the above result relies in the fact that it can be adapted to systems which are cooperative with respect to a different cone than the nonnegative orthant. Indeed, if one can find a diagonal matrix $S(r)$ with diagonal entries in $\{-1,1\}$ such that $S(r)J^*(r)S(r)$ is Metzler, then the above result also applies to that system.

\subsection*{The niAIC with output inhibition structurally stabilizes Michaelis-Menten networks under some set-point admissibility condition}

Another important class of nonlinear networks benefiting from a complexity reduction is the class of linear mass-action and Michaelis-Menten networks \cite{Alon:07}. Such networks  \eqref{main:eq:NLCL} are described by a function $f$ of the form
\begin{equation}\label{main:eq:MM:f}
  f(x)=Ax+N(x)
\end{equation}
where $N_i(x)=\sum_{j}\alpha_{ij}\dfrac{1}{1+x_j/\beta_{ij}}+\sum_{j}\gamma_{ij}\dfrac{x_j/\delta_{ij}}{1+x_j/\delta_{ij}}$, where $\alpha_{ij},\beta_{ij},\delta_{ij},\gamma_{ij},>0$ and $i=1,\ldots,n$ are all fixed parameters; i.e. they are all in $\mathcal{P}_f$.

This leads us to the following result:
\begin{theorem}\label{main:th:stability:MM}
 Let $r$ be admissible and assume that the controlled Michaelis-Menten network \eqref{main:eq:NLCL}-\eqref{main:eq:MM:f} has a unique equilibrium point $(x^*(r),z_1^*(r),z_2^*(r))$. Assume further that the graph of the network is strongly connected, that $b_0\ne0$, and that $H_n(0,r)>0$. Then, the unique equilibrium point $(x^*(r),z_1^*(r),z_2^*(r))$ of the Michaelis-Menten system  \eqref{main:eq:NLCL}-\eqref{main:eq:MM:f} is locally exponentially stable for all $\eta,k_p>0$.
\end{theorem}

Getting back to the network \eqref{main:eq:ex:feedback} considered previously,  we can observe that it falls into the category considered here. As a result, we can quite readily conclude on its stability properties right after having checked the set-point admissibility and the condition that $H(0,r)>0$ without the need for all the subsequent calculations.

\subsection*{Intein-based implementation}

We provide here a possible implementation of the niAIRC with output inhibition based on inteins \cite{Nanda:20,Wang:22}. Such molecules have been to be an effective tool in the design of synthetic circuits by allowing to implement synthetic networks capable of performing computations. Even though various AIC structures have already been discussed in \cite{Anastassov:23}, it seems that an intein-based implementation of the niAIC with output inhibition has not been addressed so far. The overall implementation, relying on a split protease connected with an linker containing an N-Intein, is depicted in Figure~\ref{fig:inteins}. \blue{Interestingly, our results demonstrate that the controller structure is intrinsically very robust and will function over a broad range of its implementation parameters -- if not all values. This makes its implementation much easier and reduces the importance of precisely identifying its parameters, unlike other types of controllers that lack structural stability properties. However, fine-tuning may still be necessary to adjust certain quantitative properties of the controlled network, such as its settling time or other time-domain features. This tuning must be done on a case-by-case basis, as these properties depend heavily on the characteristics of the network being controlled. In this regard, formulating general guidelines is difficult and are not considered here.}

%\red{Interestingly, our results all show that the controller structure is intrinsically very robust and will function for a broad range of values, if not all values, of its implementation parameters. This renders its implementation much easier and the identification of its pareameters a less crucial part as for other types of controllers failing to exhibit structural stability properties. However, (fine) tuning may still be necessary to adjust certain quantitative properties for the controlled network such as its settling-time, or any other time-domain feature. This procedure will need to be done on a case by case basis as those properties will also depend very much on the properties of the network to be controlled. In this regard, general guidelines are difficult to formulate.}

\begin{figure}[H]
\centering
\includegraphics[width=0.5\textwidth]{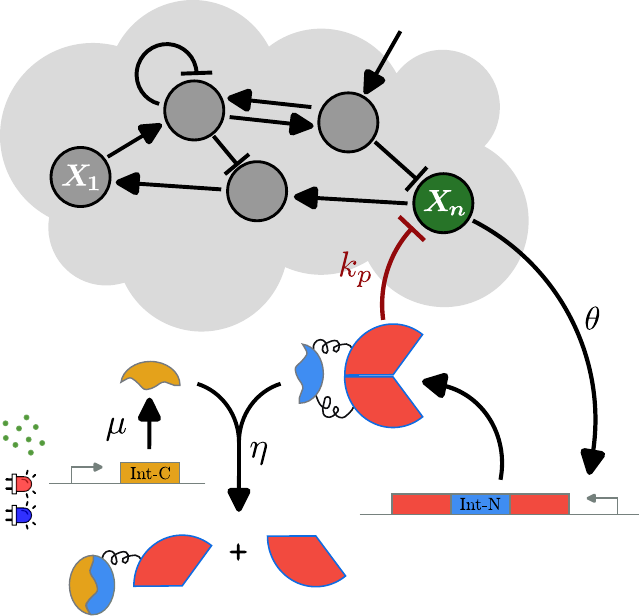}
\caption{Intein-based implementation of the niAIC with output inhibition: the species $\Z{2}$ is a split protease connected with a linker containing an N-Intein whereas the species $\Z{1}$ is a C-Intein, which is constitutively produced. When the two inteins bind, the protease is split in two inactive part. The output species $\X{n}$ acts here as a transcription factor for the split protease-intein compound. }\label{fig:inteins}
\end{figure}

\section{Discussion}

%\red{
%\begin{itemize}
%  \item Limitations of the controller when paths have not fixed sign
%  \item conservativeness of the result on the AIRC
%  \item Take-away message
%\end{itemize}
%}

This paper reports practical results regarding the use of niAICs with output inhibition and the consideration of networks for which the output can be directly actuated. Those practical results are implications of novel, more fundamental, and more general results which are all proven and discussed in the SI. Due to their high level of  technicality, it seemed impractical to present them in the main text. Those results, which are based on ideas from dynamical systems theory, and systems and control theory, have intrinsic value on their own and are sufficiently general to be applicable to other, similar problems. For instance, they directly apply to other types of integral controllers such that as exponential integral controllers \cite{Shoval:11,Briat:16a,Briat:19:Logistic} and logistic integral controllers  \cite{Briat:19:Logistic}. Those results are all reported in the SI where very similar conclusions as for the niAIC are also drawn.\\

The entire set-up is based on the assumption that the output can be actively degraded. It is, therefore, natural to ask the question whether the same results hold true whenever an intermediary molecule is degraded instead. The general answer is "no": one cannot guarantee that the closed-loop network is stable independently of the gains of the controller since the associated linearized system will not satisfy the required positive realness properties. This fact proves that it is essential that the output species be directly inhibited for the structural stability results to hold. \blue{The paper assumes that the control paths from $\X{1}$ to $\X{n}$ and $\X{m}$ to $\X{n}$ have fixed activating or inhibiting roles that do not change over time. While this assumption is reasonable in the context of unimolecular mass-action systems, it may become restrictive in more complex, nonlinear networks, where a path might be activating in one region of the state space and inhibiting in another. This scenario is not addressed in the current work for several reasons. First, the focus of this paper is on output inhibition, which is inherently inhibitory by design, precluding any activation through this actuation path. Such behavior could only arise if the actuated species were different from the controlled species $\X{n}$. Second, the theoretical tools developed in the SI are not applicable to this case, except under very restrictive conditions. This limitation arises for two main reasons. From a technical perspective, the tools fail to provide meaningful predictions about the structural stability of the controlled network in some instances of this scenario. Furthermore, the current controller topology is not really suited to address such complexities, and alternative controller structures may be necessary. Identifying an appropriate topology for this situation is an open and intriguing problem, which we leave for future research.}

%\red{Another assumption made in the paper is that the control paths $\X{1}$ to $\X{n}$ and $\X{m}$ to $\X{n}$ have a fixed activating or inhibiting role that does not change over time. While this assumption is not restrictive in the unimolecular mass-action case, it may become so in the context of more complex, nonlinear networks where a path may be activating in a region of the state-space and inhibiting in another region. This case has not been considered in the present paper for multiple reasons. The first one is that since we are focusing on output inhibition, which is inhibiting by construction, ruling out any possible case of activation using this actuating path. This phenomenon can only occur when the actuated species is another species than the controlled species $\X{n}$. The second reason is that the tools developed in the SI do not really apply to this case, unless in some very restrictive cases. The reasons are twofold: the first one is purely technical in the sense that mathematically speaking the tools fail to produce any interesting prediction about the structural stability of the controlled network. The second one is simply that the structure of the controller is not adequate for addressing this particular situation and that a different controller topology may be required. It is unclear at this time what this structure should be and this introiguing problem is left for future research.}\\

While the results obtained in this paper shed some light on the structural stability properties of networks controlled with an niAIC with output inhibition, those results do not provide any insights regarding how to properly choose the controller parameters. Such parameters could be numerically optimized so that certain performance properties are satisfied such that having a small spectral abscissa (which governs the rate of convergence near the equilibrium point) or reducing oscillations/overshoot, etc. Those performance criteria are typically difficult to optimize manually but could be easily done computationally using iterative methods.\\

The obtained results rely on linearization, a procedure that makes the analysis simpler but with the additional caveat that the results will only hold locally: the convergence of the state trajectories to the equilibrium point is only guaranteed provided that the initial conditions are "close enough" to the equilibrium point. Analytically quantifying how close they need to be is a difficult task in general and this analysis should be performed on a case-by-case basis. An important question is whether the state trajectories will converge to the equilibrium point regardless the initial condition, a property called global asymptotic stability. This is a difficult problem and the obtained results can neither be used nor easily extended to assess this. A possible workaround consists of constructing a Lyapunov function which would prove the global asymptotic stability of the equilibrium point of the closed-loop network. Moreover, proving the structural stability of the closed-loop network with respect to (some of) the controller parameters requires an alternative route than through the consideration of transfer functions and their positive realness properties. Certainly the most natural way to resolve this problem would be to consider the time-domain analogue of positive realness properties called passivity properties. In this case, the main tools at our disposal are storage functions and supply-rates, that may allow us to prove the passivity properties of the controller and the network, and the structural stability properties for the closed-loop network \cite{Brogliato:07,vanderSchaft:00}.\\

Some of the results obtained here for the niAIC with output inhibition have been also extended to the aiAIRC with output inhibition using a perturbation argument. This argument allows us to infer that the closed-loop network will maintain its structural stability property provided that the gain $k_i$ is "small enough". However, the current methodology is not only unable to quantify how small this gain should be but also fails to prove or disprove the fact that the closed-loop network remains stable for all set-points and all controller parameters, as extensive simulations tend to suggest. The main bottleneck is that the current approach heavily relies on the fact that the linearized dynamics depends linearly on the gain of the controller and that the equilibrium point depends continuously on it. Those properties can be ensured by suitably modifying the initial problem, a procedure that was introduced in \cite{Briat:15e,Briat:19:Logistic} in the context of the naAIC. Unfortunately, this procedure does not work anymore for the aiAIRC for the simple reason that this controller involves two gains, $k_p$ and $k_i$, which forbids the existence of a similar problem reformulation. Calculations notably show that the expressions for the  equilibrium point and the linearized dynamics are discontinuous functions on the gains, which invalidates the entire approach and demonstrates the need for a radically different approach.\\

%Finally, it seems necessary to discuss the stochastic case and unfortunately the tools developed here cannot be used in this context. Again, new tools and ideas are necessary here to cover this case in more details and it seems that \\

An important limitation of the niAIC with output inhibition that needs to be discussed lies at the level of the set-point admissibility set. It was shown that in the case of unimolecular networks, all the admissible set-points lie below the basal expression level. This can be seen as a serious issue whenever the basal expression levels are low. A first possible workaround would be to increase the basal expression levels through overexpression of the corresponding genes. This would allow for a wider range of values for the set-point but at the same time would create a futile cycle having a high metabolic burden, especially when the set-point is much lower than the basal level. A second, perhaps riskier, strategy consists of destabilizing  the network, through the addition of a positive feedback loop or the addition of an autocatalytic reaction, which would place the network in a regime where all the set-points are admissible.\\

\begin{tcolorbox}[breakable,break at=.93\textheight]{Box 1. The Perfect Adaptation Problem and (Antithetic) Integral Control}\label{box:1}
\footnotesize

The objective of this separate section is to illustrate through simple examples the possibilities and limitations of two types of AICs, the naAIC and the niAIC with output inhibition, on the following simple gene expression network
\begin{equation}\label{main:eq:AIC:gene_exp}
\phib\rarrow{k_0}\X{1},\ \X{1}\rarrow{k_2}\X{1}+\X{2},\ \X{1}\rarrow{\gamma_1}\phib,\ \X{2}\rarrow{\gamma_2}\phib
  \end{equation}
  where $\X{1}$ and $\X{2}$ denote the mRNA and protein species, respectively. Above, the positive rate parameters $k_0,k_2,\gamma_1$, and $\gamma_2$ of the mass-action reactions correspond to the basal transcription rate, the translation rate, the mRNA degradation rate, and the protein degradation rate, respectively. We consider here the problem of ensuring perfect adaptation for the protein concentration. Therefore, $\X{2}$ is here both the controlled and the measured species. From the dynamical model associated with this network, the equilibrium basal level for the concentration of $\X{2}$ is computed to be  $g_0:=k_0k_2/(\gamma_1\gamma_2)$.\\

%that this controller indeed solves the perfect adaptation problem stated in Definition~\ref{main:def:RPAP} for the gene expression network given by

The original version of the AIC, as introduced in \cite{Briat:15e}, corresponds to the naAIC setup depicted in Figure~\ref{main:fig:AICtable}. A possible reaction network representation for this controller is given by
\begin{equation}\label{main:eq:AIC:naAIC}
  %\phib\rarrow{\mu}\Z{1},\ \Yz\rarrow{\theta}\Yz+\Z{2},\ \Z{1}+\Z{2}\rarrow{\phi\eta}\phib,\  \Z{1}\rarrow{k}\Z{1}+\X{1},
  \phib\rarrow{\mu}\Z{1},\ \Yz\rarrow{\theta}\X{2}+\Z{2},\ \Z{1}+\Z{2}\rarrow{\eta}\phib,\  \Z{1}\rarrow{k}\Z{1}+\X{1},
\end{equation}
where $\Z{1}$ is both the actuating and set-point species, $\Z{2}$ is the sensing species, and $\X{1}$ is the actuated species. The parameters $\mu,\theta,\eta,k$ are all positive rate parameters of the reactions, which are all assumed to be mass-action. We note that this controller activates the production of $\X{1}$, which means that it can only increase the level of the controlled species $\X{2}$ over the basal level $g_0$. Calculations accordingly show that the set-point $r=\mu/\theta$ must be such that $r>g_0$ for the controlled network to be satisfy the perfect adaptation property stated in Definition \ref{main:def:RPA}. When this is not satisfied, one component of the state of the network grows without bound as depicted in Figure~\eqref{main:fig:nAIC}.B. However, when the stability conditions for the controlled network are met, we automatically have that $x_2(t)\to r$  as $t\to\infty$ demonstrating the perfect adaptation property for the controlled network; see Figure~\eqref{main:fig:nAIC}.A.\\

Consider now for the sake of comparison another type of AIC controller, the niAIC with output inhibition. A possible reaction network representation for this controller is given by
\begin{equation}\label{main:eq:AIC:niAIC}
  \phib\rarrow{\mu}\Z{1},\ \X{2}\rarrow{\theta}\X{2}+\Z{2},\ \Z{1}+\Z{2}\rarrow{\eta}\phib,\  \X{2}+\Z{2}\rarrow{k}\Z{2}
\end{equation}
where now $\Z{1}$ is the set-point species, $\Z{2}$ is both the sensing and actuating species, and $\X{2}$ is the measured/controlled/actuated species.  We can observe in this case that the actuating species $\Z{2}$ actively degrade the controlled species $\X{2}$, which indicates that this controller can only decrease the output with respect to the basal level $g_0$. In fact, calculation show exactly that it is necessary that the set-point $r$ be such that $r<g_0$ for the network to have stability and perfect adaptation properties. As for the naAIC, if this condition is not met, one component of the state of the network grows without bound as depicted in Figure~\eqref{main:fig:nAIC}.D. However, when the stability conditions for the controlled network are met, we automatically get that $x_2(t)\to r$  as $t\to\infty$ demonstrating the perfect adaptation property for the controlled network; see Figure~\eqref{main:fig:nAIC}.C.

\begin{center}
  \includegraphics[width=0.55\textwidth]{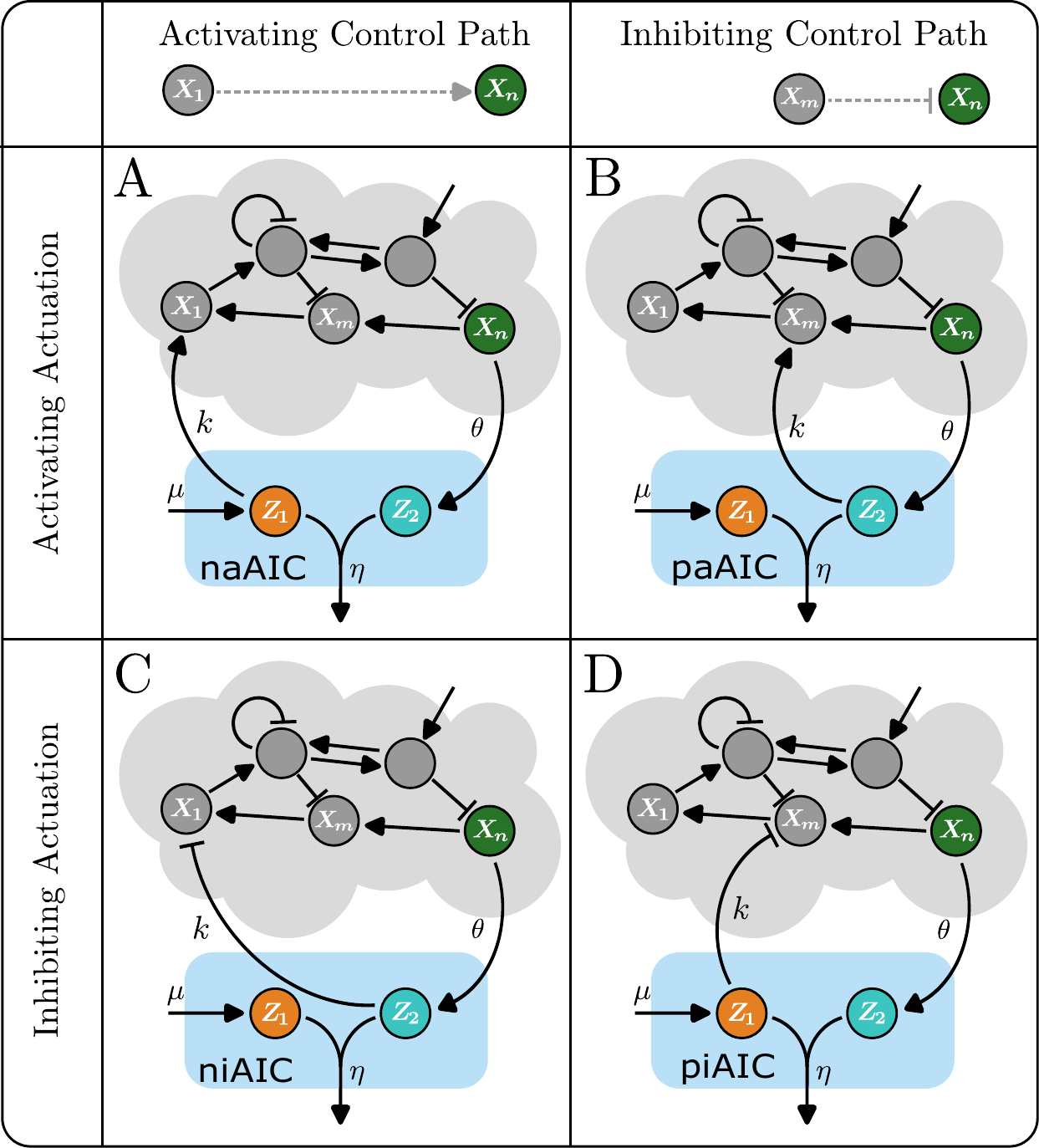}
  \captionof{figure}{\footnotesize
    The main objective of Antithetic Integral Controllers is to interact with a network in a way that makes one of its species, here $\X{n}$, exhibit the robust perfect adaptation property. The core mechanism of the AIC is the sequestration reaction that allows to perform a comparison operation between the levels of the two controller species $\Z{1}$ and $\Z{2}$. AICs admit different possible architectures depending on the role of its molecular species, which may act as activators or inhibitors. \blue{The naming convention for these controllers is as follows: "n" indicates that the loop is closed negatively by the controller (i.e., $\X{n}$ inhibits $\X{1}$ or $\X{m}$), "p" signifies that the loop is closed positively (i.e., $\X{n}$ activates $\X{1}$ or $\X{m}$), and "a" or "i" denotes whether the actuating species is an activator or a repressor. It is assumed that $\X{1}$ activates $\X{n}$ in the networks of the first column, whereas $\X{m}$ inhibits $\X{n}$ in those in the second column. Based on this, the naAIC bears its name because it activates $\X{1}$ using $\Z{1}$ and closes the loop negatively, since $\X{n}\regulationarrow[act]\Z{2}\regulationarrow[rep]\Z{1}\regulationarrow[act]\X{1}$, which can be reduced to $\X{n}\regulationarrow[rep]\X{1}$. Similarly, the piAIC bears its name because it inhibits $\X{m}$ using $\Z{1}$ and closes the loop positively, since $\X{n}\regulationarrow[act]\Z{2}\regulationarrow[rep]\Z{1}\regulationarrow[rep]\X{m}$, which can be reduced to $\X{n}\regulationarrow[act]\X{m}$.}
%    \red{The naming convention for those controllers are as follows: "n" stands for the fact that the loop is closed negatively by the controller (i.e. $\X{n}$ inhibits $\X{1}$ or $\X{m}$), "p" indicates that the loop is closed positively (i.e. $\X{n}$ activates $\X{1}$ or $\X{m}$), and  "a" or "i" indicates whether the actuating species is an activator or a repressor. It is assumed that $\X{1}$ activates $\X{n}$ in the networks of the first column whereas $\X{m}$ inhibits $\X{n}$ in those in the second column. Based on this, the naAIC bears its name because it activates $\X{1}$ using $\Z{1}$ and it closes the loop negatively since  $\X{n}$ inhibits itself through the chain $\X{n}\regulationarrow[act]\Z{2}\regulationarrow[rep]\Z{1}\regulationarrow[act]\X{1}$, which can be reduced to $\X{n}\regulationarrow[rep]\X{1}$ Similarly, the piAIC bears its name because it inhibits $\X{m}$ using $\Z{1}$ and it closes the loop positively since  $\X{n}$ inhibits itself through the chain $\X{n}\regulationarrow[act]\Z{2}\regulationarrow[rep]\Z{1}\regulationarrow[rep]\X{m}$, which be reduced to $\X{n}\regulationarrow[act]\X{m}$.}
    }\label{main:fig:AICtable}
\end{center}

\begin{center}
\includegraphics[width=0.65\textwidth]{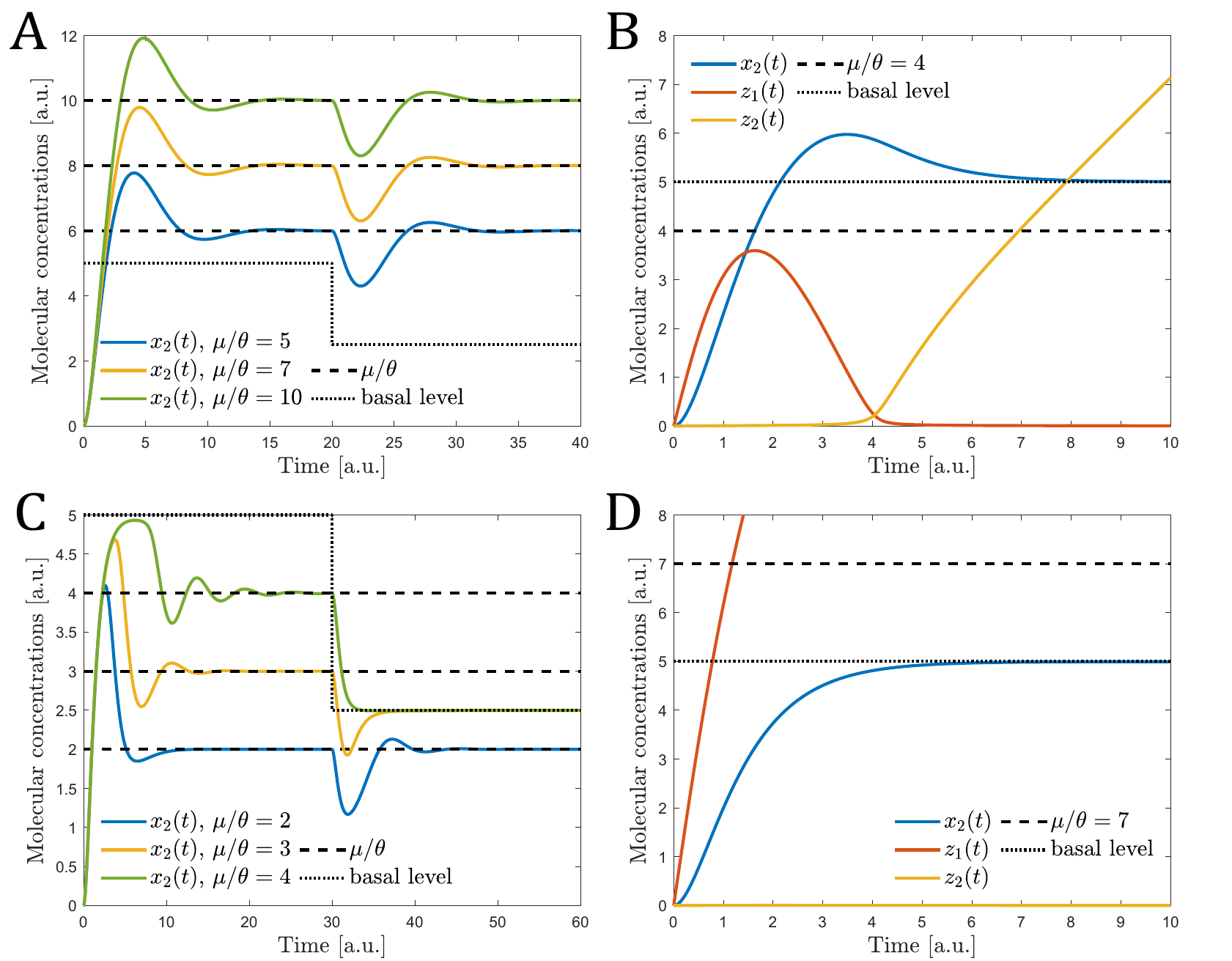}
\captionof{figure}{\footnotesize \textbf{A.} Simulation of the reaction network \eqref{main:eq:AIC:gene_exp}-\eqref{main:eq:AIC:naAIC} with the parameters $\gamma_1=1$, $\gamma_2=2$, $k_2=1$, $k_0=10$, $k=1$, $\eta=100$, $\theta=1$ and for various values for $\mu$. At time $t=20$, the value of $k_0$ is divided by 2. When the set point $\mu/\theta$ is set above the basal level, the controlled species $\X{2}$ exhibits the robust perfect adaptation property. \textbf{B}. When the set-point is set to a value lower than the basal level, the concentration of controller species $\Z{2}$ grows without bound and perfect adaptation is not achieved for $\X{2}$. \textbf{C.}  Simulation of the reaction network \eqref{main:eq:AIC:gene_exp}-\eqref{main:eq:AIC:niAIC} with the parameters $\gamma_1=1$, $\gamma_2=2$, $k_2=1$, $k_0=10$, $k=0.5$, $\eta=100$, $\theta=1$ and for various values for $\mu$. At time $t=30$, the value of $k_0$ is divided by 2. When the set point $\mu/\theta$ is set below the basal level, the controlled species $\X{2}$ exhibits the robust perfect adaptation property. \textbf{D.}  When the set-point is set to a value higher than the basal expression level, the concentration of the controller species $\Z{1}$ grows without bound and perfect adaptation is not achieved for $\X{2}$, which can only converge to its basal level.}\label{main:fig:nAIC}
\end{center}

\end{tcolorbox}

\begin{tcolorbox}[breakable,
break at=.93\textheight]{Box 2. Passivity analysis and the analysis of interconnections}\label{box:2}
\footnotesize
\ \\ \ \\
A very powerful tool for the analysis of feedback linear systems is called the Nyquist criterion, which is a frequency domain graphical criterion. Consider a linear system with transfer function $H_1(s)$ with poles with negative real part and such that $H_1(0)>0$. Consider further the transfer function $H_2(s)=k/s$ where $k>0$. We are now interested in studying the stability of the interconnection depicted in Figure \ref{main:fig:NyqInt}.B. The Nyquist criterion for stable systems provide a way to answer to this question:
\begin{theorem}[Nyquist stability criterion for stable systems]
The feedback interconnection depicted in Figure \ref{main:fig:NyqInt}.B with the previously defined transfer functions is stable if and only if $kH(j\omega)/j\omega$ does not encircle the point -1 in the complex plane as $\omega$ sweeps from $-\infty$ to $\infty$.
\end{theorem}

This criterion is very simple to use as one just needs to plot the Nyquist diagram using off-the-shelf computational tools and verify whether it encircles the critical point -1 or not. This can also be checked analytically by verifying that for all critical frequencies $\omega_c$ defined as $\arg(kH(j\omega_c)/j\omega_c)=-\pi$, the condition $|kH(j\omega_c)/j\omega_c|<1$ holds.\\

We now consider the question of establishing whether the interconnection is stable for all positive $k$'s. To this aim, first observe that if there exists a critical frequency $\omega_c$ as defined above, then one can select $k>k_c:=-(H(j\omega_c)/j\omega_c)^{-1}>0$ such that $kH(j\omega)/j\omega$ will encircle the point -1. Therefore, one can conclude that the stability of the interconnection holds for all $k>0$ if and only if there is no such critical frequency $\omega_c$. Noting that
\begin{equation}
  \arg(kH(j\omega)/j\omega)=\arg(H(j\omega))-\arg(j\omega)=\arg(H(j\omega))-\pi/2,
\end{equation}
we can conclude that there are no critical frequencies if and only if
\begin{equation}
  \arg(H(j\omega))\in(-\pi/2,\pi/2)
\end{equation}
for all $\omega\in\mathbb{R}$. This is equivalent to saying that $\Re[H(j\omega)]>0$ for all $\omega\in\mathbb{R}$ or that the Nyquist diagram is strictly included in the open right half-plane as illustrated in Figure~\ref{main:fig:NyqInt}.A. Interestingly, we can connect this property to the positive realness properties of transfer functions given below:
\begin{define}\label{main:def:PRSPR}
  The rational, proper function $H:\mathbb{C}\mapsto\mathbb{C}$ is said to be
  \begin{enumerate}[(1)]
  \item Weakly Strictly Positive Real (WSPR) if
    \begin{enumerate}[({\theenumi}a)]
      \item\label{main:st:SPR:1} all the poles of $H(s)$ are in the open left half-plane, and
      \item\label{main:st:SPR:2} $\Re[H(j\omega)]>0$ for all $\omega\in[0,\infty)$.
      %\item\label{main:st:SPR:3} $H(\infty)>0$ or $\lim_{\omega\to \infty}\omega^2\Re[H(j\omega)]>0$.
    \end{enumerate}
    \item Strictly Positive Real (SPR) if
    \begin{enumerate}[({\theenumi}a)]
      \item it is weakly strictly positive real, and
      \item\label{main:st:SPR:3} $H(\infty)>0$ or $\lim_{\omega\to \infty}\omega^2\Re[H(j\omega)]>0$.
    \end{enumerate}
\end{enumerate}
\end{define}

As previously said, checking those properties can be done computationally and graphically. Computing the poles of the transfer function as well as plotting the Nyquist diagram can be easily performed using off-the-shelf software. A difficulty is the possible existence of values for $\omega$ for which the diagram come very close to the imaginary axis, making it difficult to establish whether they intersect. To palliate this, the existence of intersections can be assessed by finding real solutions $\omega$ to the equation $\Re[H(j\omega)=0$, which amounts to numerically solving a polynomial for real solutions. Finally, the additional condition in the SPR property can be easily checked using the expression
\begin{equation}
\lim_{\omega\to \infty}\omega^2\Re[H(j\omega)]=\sum_{i=1}^{n_z} z_i-\sum_{i=1}^{n_p} p_j,\ \textnormal{whenever }n_p=n_z+1,
\end{equation}
and where the $z_i$'s and $p_j$'s are the zeros and the poles of the transfer function $H(s)$, respectively.

\begin{figure}[H]
  \centering
  \includegraphics[width=.4\textwidth]{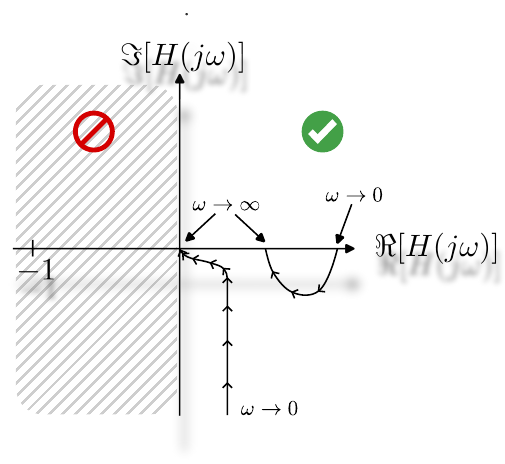}\includegraphics[width=.5\textwidth]{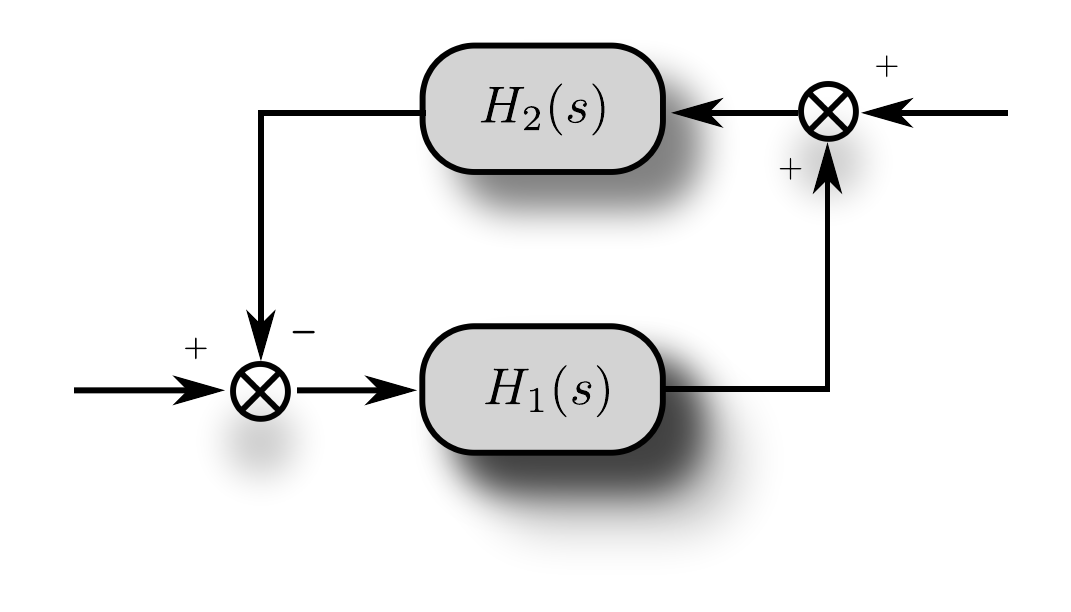}
  \caption{A. Nyquist Diagram. The Nyquist diagram of a transfer function $H(s)$ consists of representing the locus of $H(j\omega)$ in the complex plane as $\omega$ goes to zero to infinity, a direction indicated by the arrows on the diagram. For a transfer function to be positive real, it is necessary that this locus be included in the right half-plane. The Nyquist diagram of the transfer functions $(s+1)/(s(s+2))$ and $(s+1)/(s+2)$ is represented here. The first one is positive real whereas the second is strictly positive real. B. A typical interconnection of systems as considered in systems and control theory for the (structural) stability analysis of control systems.}\label{main:fig:NyqInt}
\end{figure}

\end{tcolorbox}

\section*{Author contributions}

M.K. and C.B. conceived the study, C.B. performed the research, C.B. performed the numerical simulations, M.K. supervised the research and secured the funding, C.B. and M.K. wrote the paper.

%\section*{Funding statement}
%
%\red{This project has received funding from the European Research Council (ERC) under the European Union's Horizon 2020 research and innovation programme (grant agreement 743269).}

\section*{Declaration of interests}

The authors declare no competing interests.

\section*{Declaration of generative AI and AI-assisted technologies in the writing process}

%The authors declare not having used any generative AI tool for their research and the writing of this manuscript.

%Generative AI and AI-assisted technologies have only been used in the writing process to improve the readability and language of the manuscript.
During the preparation of this work the authors used ChatGPT in order to improve the readability and language of the manuscript. After using this tool/service, the authors reviewed and edited the content as needed and take full responsibility for the content of the published article.

\section*{Acknowledgments}

The authors gratefully thank Christian Cuba Samaniego, Maurice Filo, Jean-Baptiste Lugagne, Noah Olsman, and Yili Qian for fruitful discussions.

%\bibliographystyle{plain}
%\bibliography{../../../Lastbib/global, ../../../Lastbib/briat}

\setcounter{section}{0}
\renewcommand{\thesection}{S\arabic{section}}
\renewcommand{\thetheorem}{\thesection.\arabic{theorem}}
\renewcommand{\thefigure}{S\arabic{figure}}

\newpage

\begin{center}
{\Huge\textbf\newline{Supplementary Information}} 
\end{center}

\tableofcontents

\ \\
\noindent\textbf{Notations.} The cone of symmetric positive (semi)definite matrices of dimension $n$ is denoted by ($\mathbb{S}^n_{\succeq0}$) $\mathbb{S}^n_{\succ0}$. For two symmetric matrices $A,B$ with same dimension, $A\succ B$ ($A\succeq B$) means that $A-B\succ0$ ($A-B\succeq0$); i.e. $A-B$ is positive (semi)definite. The natural basis for $\mathbb{R}^n$ is denoted by $\{e_1,\ldots,e_n\}$.

\resumetoc
\section{Preliminary definitions and results}\label{sec:prel}

This section is devoted to introducing the main definitions and results that will prove instrumental in proving the main results of the paper.

\subsection{Results on (positive) linear systems}

Let us introduce the following Linear Time-Invariant (LTI) system
\begin{equation}\label{eq:LTI:positive}
  \begin{array}{rcl}
    \dot{x}(t)&=&Ax(t)+Bu(t), t\ge0\\
    y(t)&=&Cx(t)+Du(t)\\
    x(0)&=&x_0
  \end{array}
  \end{equation}
 where $A\in\mathbb{R}^{n\times n}$, $B\in\mathbb{R}^{n\times m}$, $C\in\mathbb{R}^{p\times n}$, and $D\in\mathbb{R}^{p\times m}$. The following definition introduces the concept of spectral abscissa:
\begin{define}
 The spectral abscissa $\alpha(A)$ of a matrix $A\in\mathbb{R}^{n\times n}$ is defined as
 \begin{equation}
   \alpha(A):=\max\left\{\Re(\lambda):\det(\lambda I-A)=0\right\}.
 \end{equation}
\end{define}

This concept is essential in the analysis of linear dynamical systems of the form \eqref{eq:LTI:positive} since it is immediate to see that this the 0-equilibrium point of that system is asymptotically stable if and only $\alpha(A)<0$, which is equivalent to saying that the spectrum of $A$ is included in the open left half-plane of the complex plane. The next result states conditions under which the system \eqref{eq:LTI:positive} is internally positive:
\begin{proposition}[\cite{Farina:00}]
The following statements are equivalent:
\begin{enumerate}[(a)]
\item The linear dynamical system \eqref{eq:LTI:positive} is internally positive; i.e. for all $x_0\ge0$ and all $u(t)\ge0$, $t\ge0$, we have that $x(t)\ge0,\ y(t)\ge0$, for all $t\ge0$.
\item The matrix $A$ is Metzler (i.e. all the off-diagonal entries are nonnegative) and the matrices $B,C,D$ are nonnegative (i.e. all the entries are nonnegative).
\end{enumerate}
\end{proposition}

\begin{proposition}[\cite{Berman:94,Farina:00,Kaszkurewicz:00,Kushel:19}]\label{prop:MHS}
  Let $A\in\mathbb{R}^{n\times n}$ be Metzler. Then, the following statements are equivalent:
  \begin{enumerate}[(a)]
    \item $A$ is Hurwitz stable (i.e. all its eigenvalues have negative real part or, equivalently, $\alpha(A)<0$),
    \item For all $q\in\mathbb{R}^n_{>0}$, there exists a vector $v\in\mathbb{R}^n_{>0}$ such that $v^TA=-q^T$, and
    \item There exists a diagonal matrix $D$ with positive diagonal such that $A^TD+DA\prec0$ (i.e. $A^TD+DA$ is negative definite).
    \item $A$ is invertible and $-A^{-1}$ is nonnegative.
  \end{enumerate}
\end{proposition}

\black{\begin{lemma}[\cite{Briat:15e}]\label{lemma:gains}
  Assume that $A$ is Metzler and Hurwitz stable, then the system $(A,e_i,e_j^T)$, $i,j\in\{1,\ldots,n\}$, with transfer function $H_{ij}(s)=e_j^T(sI-A)^{-1}e_i$, is output controllable if and only if one of the following statements holds:
  \begin{enumerate}
    \item The impulse response $h_{ij}(t)=e_j^Te^{At}e_j$ is not identically zero.
    \item The DC-gain $H_{ij}(0)=-e_j^TA^{-1}e_i$ is positive.
    \item The rank condition $\rank\begin{bmatrix}
      e_j^Te_i & e_j^TAe_i & e_j^TA^{n-1}e_i
    \end{bmatrix}=1$ is satisfied.
  \end{enumerate}
\end{lemma}
%\begin{proof}
%  dsdsdsdsd
%\end{proof}
}

%\begin{proof}
%  The proof is given for completeness. Assume that $A$ is Hurwitz stable, then $A$ is invertible. Let $M(t):=\int_0^te{As}\ds$ and clearly $M(t)\ge0$ for all $t\ge0$ since $e^{As}\ge0$ for all $s\ge0$. Since $A$ is invertible, then we have that $M(t)=A^{-1}(e^{At}-I)\ge0$ and since $A$ is Hurwitz stable, this implies that $M(t)\to -A^{-1}$ as $t\to+\infty$. Since the limit is nonnegative then the conclusion follows. Conversely, assume that $A$ is invertible and $-A^{-1}$ is nonnegative. Then, $-\mathds{1}^TA^{-1}=:v^T>0$ since there cannot be any zero column. Multiplying on the right by $A$ yields $v^TA=-\mathds{1}^T$, which implies that $A$ is Hurwitz stable.
%\end{proof}

\begin{proposition}[\cite{Berman:94}]\label{prop:PF}
Assume that the matrix $M\in\mathbb{R}^{n\times n}$ is Metzler. Then, the following statements hold:
\begin{enumerate}[(a)]
  \item There exists an eigenvalue $\lambda_{\mathrm{PF}}(A)$ of $A$ that is equal to $\alpha(A)$, which is referred to as the Perron-Frobenius eigenvalue. Moreover, if the matrix $A$ is irreducible, then this eigenvalue is unique.
  \item Let $\tilde M$ be a Metzler matrix such that $M\le \tilde M$ (where the inequality is component-wise), then $\alpha(M)\le \alpha(\tilde M)$.
\end{enumerate}
\end{proposition}

Finally, the following result will be crucial for proving the main results of the paper:
\black{\begin{lemma}\label{lemma:zeros}
  Consider the linear system
  \begin{equation}\label{eq:SSHn}
  \begin{array}{rcl}
    \dot{x}&=&Mx+e_nw\\
    z&=&e_n^Tx
  \end{array}
  \end{equation}
  where $M\in\mathbb{R}^{n\times n}$. Then, we have that
  \begin{equation}
    e_n^T\adj(sI-M)e_n=\det(sI-M_{11}),
  \end{equation}
 where $M_{11}$ is $(n-1)\times(n-1)$ upper left block of $M$ and $\adj(\cdot)$ denotes the adjugate matrix. Moreover,
  \begin{enumerate}[(a)]
  \item if the pairs $(M,e_n)$ and $(M,e_n^T)$ are controllable and observable, respectively, then
\begin{equation}
  H(s):=e_n^T(sI-M)^{-1}e_n=\dfrac{\det(sI-M_{11})}{\det(sI-M)}.
\end{equation}
  \item if $M$ is a Metzler, Hurwitz stable matrix, then the solutions of $\det(sI-M_{11})=0$ are all located in the open left half-plane of the complex plane.
  \end{enumerate}
\end{lemma}
\begin{proof}
%The transmission zeros of the system \eqref{eq:SSHn} are given by the values $s\in\mathbb{C}$ such that
First note that
\begin{equation}\label{eq:detzeros}
  e_n^T\adj(sI-M)e_n=\det\begin{bmatrix}
    sI-M_n & -e_n\\
    e_n^T & 0
  \end{bmatrix}.
\end{equation}
If we decompose $M$ as
\begin{equation}
 M=:\begin{bmatrix}
    M_{11} & M_{12}\\
    M_{21} & M_{22}
  \end{bmatrix},
  %=\begin{bmatrix}
%    S^T\\
%    e_n^T
%  \end{bmatrix}M\begin{bmatrix}
%    S & e_n
%  \end{bmatrix},
\end{equation}
where $M_{11}\in\mathbb{R}^{(n-1)\times(n-1)}$, $M_{12}\in\mathbb{R}^{(n-1)\times 1}$, $M_{21}\in\mathbb{R}^{1\times (n-1)}$, and $M_{22}\in\mathbb{R}$, then the expression \eqref{eq:detzeros} becomes
\begin{equation}
  \begin{bmatrix}
     sI-M_n & -e_n\\
    e_n^T & 0
  \end{bmatrix}=\begin{bmatrix}
    sI-M_{11} & -M_{12} & 0\\
    -M_{21} & s-M_{22} & -1\\
    0 & 1 & 0
  \end{bmatrix}.
\end{equation}
Using now the properties of the determinant on the last row, we get that
\begin{equation}
  \det\begin{bmatrix}
   sI-M_{11} & -M_{12} & 0\\
    -M_{21} & s-M_{22} & -1\\
    0 & 1 & 0
  \end{bmatrix}=-\det\begin{bmatrix}
    M_{11}-sI_n &  0\\
    M_{21} & -1
  \end{bmatrix}=\det(sI-M_{11}),
\end{equation}
which proves the first statement of the result.\\

\noindent Assuming now that the pairs $(M,e_n)$ and $(M,e_n^T)$ are controllable and observable, respectively, then there is no pole/zero cancellations in the transfer function $H(s)=\dfrac{e_n^T\adj(sI-M)e_n}{\det(sI-M)}$ and the result follows.\\

%
%
%  The zeros are given by the dynamics of the system when $z(t)=0$ for all $t\ge0$ for a chosen input $w(t)$ and initial condition. If we decompose $M$ as
%\begin{equation}
%  M=\begin{bmatrix}
%    M_{11} & M_{12}\\
%    M_{21} & M_{22}
%  \end{bmatrix},\ M_{11}\in\mathbb{R}^{(n-1)\times(n-1)},\ M_{22}\in\mathbb{R}, M_{12},M_{21}^T\in\mathbb{R}^{(n-1)\times 1}
%\end{equation}
%and we define $w(t)=-e_n ^TMx(t)$, we obtain
%\begin{equation}
%  \dot{x}(t)=\begin{bmatrix}
%    M_{11} & 0\\
%    0 & 0
%  \end{bmatrix}x(t)
%\end{equation}
%Therefore, one can observe that the zeros coincide with the eigenvalues of the matrix $M_{11}=S^TMS$. This proves the first statement.

\noindent Finally, assuming that the matrix $M$ is Metzler and Hurwitz stable, then it is immediate to see that
\begin{equation}
  M_d:=\begin{bmatrix}
    M_{11} & 0\\
    0 & M_{22}
  \end{bmatrix}\le M
\end{equation}
where the inequality is entry-wise. From Perron-Frobenius theory (or Proposition \ref{prop:PF}), $M_d\le M$ implies that $\lambda_{\mathrm{PF}}(M_d)\le\lambda_{\mathrm{PF}}(M)$. Since $M$ is Hurwitz stable, then $\lambda_{\mathrm{PF}}(M)<0$, which implies that $M_d$, $M_{11}$ and $M_{22}$ are all Hurwitz stable matrices. This proves the third statement of the result and completes the proof.
\end{proof}}

\subsection{Positive real transfer functions and associated results}

\black{We consider the following definitions which can be found in \cite{Brogliato:07}:
\begin{define}\label{def:PRSPR}
  The function $H:\mathbb{C}\mapsto\mathbb{C}$ is said to be
  \begin{enumerate}[(1)]
    \item Positive Real (PR) if
    \begin{enumerate}[({\theenumi}a)]
      \item\label{st:PR:1} all the poles of $H(s)$ are in the closed left half-plane,
      \item\label{st:PR:2}  $\Re[H(j\omega)]\ge0$ for all $\omega\in[0,\infty)$, and
      \item\label{st:PR:3}  any pure imaginary pole $j\omega$ is simple and $\lim_{s\to j\omega}(s-j\omega)H(s)\ge0$.
    \end{enumerate}
    \item Weakly Strictly Positive Real (WSPR) if
    \begin{enumerate}[({\theenumi}a)]
      \item\label{st:SPR:1} all the poles of $H(s)$ are in the open left half-plane, and
      \item\label{st:SPR:2} $\Re[H(j\omega)]>0$ for all $\omega\in[0,\infty)$.
      %\item\label{st:SPR:3} $H(\infty)>0$ or $\lim_{\omega\to \infty}\omega^2\Re[H(j\omega)]>0$.
    \end{enumerate}
    \item Strictly Positive Real (SPR) if
    \begin{enumerate}[({\theenumi}a)]
      \item it is weakly strictly positive real, and
      \item\label{st:SPR:3} $H(\infty)>0$ or $\lim_{\omega\to \infty}\omega^2\Re[H(j\omega)]>0$.
    \end{enumerate}
    \item Strong Strictly Positive Real (Strong SPR) if
    \begin{enumerate}[({\theenumi}a)]
      \item it is weakly strictly positive real, and
      \item there exists a $\delta>0$ such that $\Re[H(j\omega)]\ge\delta$ for all $\omega\in[0,\infty]$.
    \end{enumerate}
  \end{enumerate}
\end{define}}

\begin{figure}
  \centering
  \includegraphics[width=.5\textwidth]{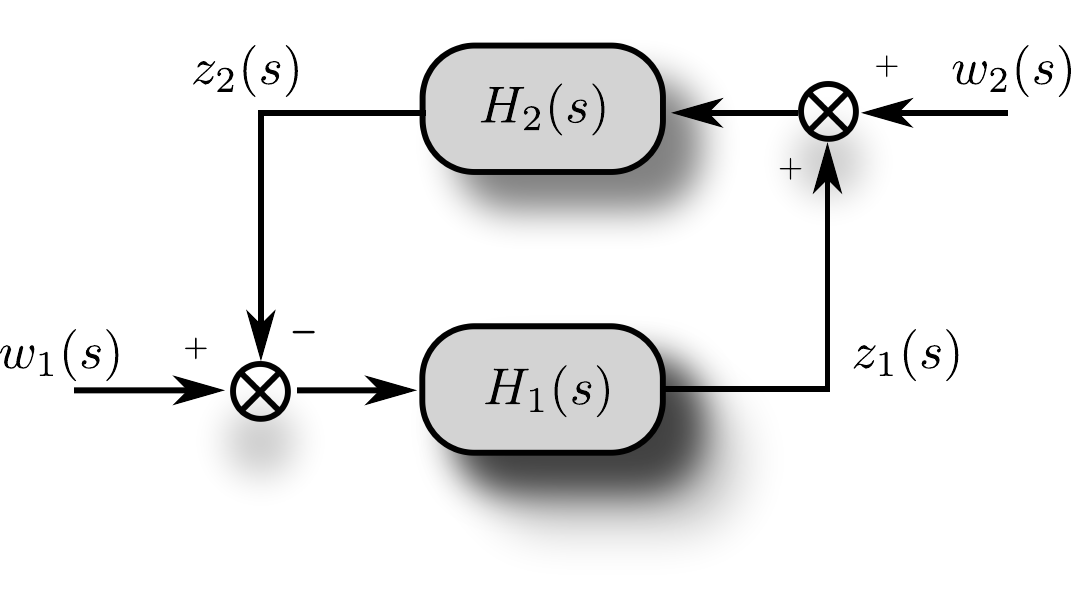}
  \caption{Negative interconnection}\label{fig:interconnectionp_passive}
\end{figure}

\black{The following result pertaining to the stability of the interconnection depicted in Figure~\ref{fig:interconnectionp_passive} will be instrumental in proving the main results of the paper:
\begin{theorem}[\hspace{-1pt}{\cite[Lemma 3.37]{Brogliato:07}}]\label{th:interconnection}
Let $H_1,H_2:\mathbb{C}\mapsto\mathbb{C}$ be two transfer functions. Assume further that
\begin{enumerate}
  \item  $H_1$ is strictly or weakly positive real, and
  \item $H_2$ is positive real.
\end{enumerate}
Then, the negative feedback interconnection defined as $z_1=H_1(w_1-z_2)$ and $z_2=H_2(w_2+z_1)$, depicted in Figure~\ref{fig:interconnectionp_passive}, is internally stable, this is, the transfer $w\mapsto z$ given by
\begin{equation}
    \dfrac{1}{1+H_1H_2}\begin{bmatrix}
      H_1 & -H_1H_2\\
      H_1H_2 & H_2
    \end{bmatrix}
\end{equation}
is stable.
%The negative feedback interconnection of a SISO  function is internally stable. This is equivalent to say that $(1+H_1(s)H_2(s))^{-1}$,  $H_1(s)(1+H_1(s)H_2(s))^{-1}$, and $H_2(s)(1+H_1(s)H_2(s))^{-1}$ are all proper, stable transfer functions.
\end{theorem}}

The following result provides sufficient conditions under which a rational, strictly proper transfer function is WSPR:
\black{\begin{proposition}\label{prop:LMI}
  Let us consider the matrices $A\in\mathbb{R}^{n\times n}$, $B\in\mathbb{R}^{n\times m}$, and $C\in\mathbb{R}^{m\times n}$, and define the function $H(s):=C(sI-A)^{-1}B$. Assume further that
 \begin{enumerate}[(1)]
    \item $H(s)$ has no zero on the imaginary axis, and
    \item there exist a matrix $P\in\mathbb{S}_{\succ0}^n$ and an $\eps>0$ such that $PB-C^T=0$ and $A^TP+PA+2\eps C^TC\prec0$.
 \end{enumerate}
   Then, $H(s)$ is weakly strictly positive real.
 % \begin{enumerate}[(a)]
%    \item the matrix $A$ is Hurwitz stable, and
%    \item $\Re[H(j\omega)]>0$ for all $\omega\in[0,\infty)$ or, equivalently,
%  \end{enumerate}
\end{proposition}
\begin{proof}
 Note that $A^TP+PA+2\eps C^TC\prec0$ for some $P\succ0$ and $\eps>0$ implies that $A$ is Hurwitz stable, meaning that all the poles of $H(s)$ have negative real part. This proves that condition \eqref{st:SPR:1} of Definition \ref{def:PRSPR} holds.\\

 \noindent To show that the condition \eqref{st:SPR:2} of Definition \ref{def:PRSPR} holds, first note that under the conditions of the result, the following matrix
  \begin{equation}
     \begin{bmatrix}
      A^TP+PA+2\eps C^TC & PB-C^T\\
      B^TP-C & 0
    \end{bmatrix}=\begin{bmatrix}
      A^TP+PA & PB\\
    B^TP & 0
    \end{bmatrix}+\begin{bmatrix}
      C & 0\\
      0 & I
    \end{bmatrix}^T\begin{bmatrix}
      2\eps I & -I\\
      -I & 0
    \end{bmatrix}\begin{bmatrix}
      C & 0\\
      0 & I
    \end{bmatrix}
  \end{equation}
  is negative semidefinite. From the Kalman-Yakubovich-Popov Lemma \cite{Willems:71b,Rantzer:96}, this is equivalent to saying that
  \begin{equation}
    \begin{bmatrix}
      H(j\omega)\\
      I
    \end{bmatrix}^*\begin{bmatrix}
      2\eps I & -I\\
      -I & 0
    \end{bmatrix}\begin{bmatrix}
      H(j\omega)\\
      I
    \end{bmatrix}\le0,\ \textnormal{for all }\omega\in[0,\infty).
  \end{equation}
  Note that $H(j\omega)$ is well-defined since $A$ has no eigenvalues on the imaginary axis. Expanding the left-hand side yields
  \begin{equation}
    2\eps |H(j\omega)||^2-2\Re[H(j\omega)]\le0,\ \textnormal{for all }\omega\in[0,\infty),
  \end{equation}
  which implies that $\Re[H(j\omega)]\ge \eps|H(j\omega)|^2$ for $\omega\in[0,\infty)$. Since $H(s)\ne0$ on the imaginary axis, this implies that the statement \eqref{st:SPR:2} in  Definition \ref{def:PRSPR} holds, and proves the desired result.
\end{proof}}

The following technical result will also prove to be very useful for proving the main results of the paper.
\black{\begin{lemma}\label{lemma:infinity}
  %Let us consider a SISO transfer function $H(s)=KN(s)/D(s)$ where $K\in\mathbb{R}$, where $N(s),D(s)$ are monic Hurwitz stable polynomials such that $\deg(D)=n$ and $\deg(N)=n-1$. Then, we have that
  Let $H:\mathbb{C}\mapsto\mathbb{C}$ be a strictly proper, rational transfer function and decompose it $H(s)=:kN(s)/D(s)$ where $k\in\mathbb{R}$, and $N(s),D(s)$ are coprime, monic polynomials such that $\deg(D)=n$ and $\deg(N)=n-1$. Then, we have that
  \begin{equation}
    \lim_{\omega\to\infty}\omega^2k\Re[H(j\omega)]=k\left(\sum_{i=1}^{n-1}z_i-\sum_{i=1}^{n}p_i\right)
  \end{equation}
  where the $p_i$'s are the poles and the $z_i$'s are the zeros of the transfer function $H(s)$; i.e. $N(z_i)=0$, $i=1,\ldots,n-1$, and $D(p_i)=0$, $i=1,\ldots,n$.
\end{lemma}
\begin{proof}
  We have that
  \begin{equation}
      \Re\left[\dfrac{N(j\omega)}{D(j\omega)}\right]=\dfrac{\Re\left[N(j\omega)D(j\omega)^*\right]}{|D(j\omega)|^2}=\dfrac{\Re\left[N(j\omega)D(-j\omega)\right]}{|D(j\omega)|^2}.
  \end{equation}
  Since the degree of $N(j\omega)D(-j\omega)$ in $\omega$ is $2n-1$ and the degree of $|D(j\omega)|^2$ is $2n$, to properly evaluated the limit of $\omega^2\Re[H(j\omega)]$, we need to find the coefficient of the term in $\omega^{2n -2}$ of  $N(j\omega)D(-j\omega)$.\\

  \noindent To figure this out, let $N(s)=\sum_{i=0}^{n-1}N_is^i$ and $D(s)=\sum_{i=0}^{n}D_is^i$ , which yields
\begin{equation}
  \begin{array}{rcl}
    \Re[N(j\omega)D(-j\omega)]&=&\Re\left[\left(\sum_{i=0}^{n -1}N_i(j\omega)^i\right)\left(\sum_{i=0}^{n }D_i(-j\omega)^i\right)\right]\\
                                                &=& N_{n -1}(j\omega)^{n -1}D_{n -1}(-j\omega)^{n -1}+N_{n -2}(j\omega)^{n -2}D_{n }(-j\omega)^{n}+R_N(\omega)\\
                                                &=& \left(N_{n -1}D_{n -1}-N_{n -2}D_{n }\right)\omega^{2n -2}+R_N(\omega),
  \end{array}
\end{equation}
where $R_N(\omega)$ is the reminder containing even powers of $\omega$ less than $2n-2$. Define also $R_D(\omega)$ to be a polynomial of even powers less than $2n$ such that  $|D(j\omega)|^2=\omega^{2n}+R_D(\omega)$. This yields
\begin{equation}
\begin{array}{rcl}
  \lim_{\omega\to\infty}\omega^2k\Re[H(j\omega)]    &=&     \lim_{\omega\to\infty}\omega^2k\dfrac{\left(N_{n -1}D_{n -1}-N_{n -2}D_{n }\right)\omega^{2n-2 }+R_N(\omega)}{w^{2n}+R_D(\omega)}\\
                                                                                                        &=&k\left(N_{n -1}D_{n -1}-N_{n -2}D_{n }\right).
\end{array}
\end{equation}
As $N(s)$ and $D(s)$ are both monic polynomials, we have that $D_n =N_{n -1}=1$, and this, together with the fact that  $-D_{n -1}$ coincides with the sum of the poles and $-N_{n -2}$ coincides with the sum of the zeros, prove the desired result.
\end{proof}}

\section{Reaction networks and robust perfect adaptation}

\subsection{Reaction Networks}

\black{Reaction networks are a very powerful modeling paradigm that can be used to represent any population system, such as those arising in ecology, epidemiology, opinion dynamics, multi-agent systems, etc.\cite{Goutsias:13}. A reaction network $(\Xz,\mathcal{R})$ consists of a set of $n$ molecular species $\Xz=\{\X{1},\ldots,\X{n}\}$ that interacts through $K$ reaction channels $\mathcal{R}=\{\mathcal{R}_1,\ldots,\mathcal{R}_K\}$ denoted as
\begin{equation}\label{eq:RN}
 \mathcal{R}_k:\ \sum_{i=1}^n\zeta_{k,i}^l\X{i}\rarrow{\lambda_k}\sum_{i=1}^n\zeta_{k,i}^r\X{i},\ k=1,\ldots,K
\end{equation}
where $\zeta_{k,i}^\ell,\zeta_{k,i}^r\in\mathbb{Z}^n_{\ge0}$ are the left and right stoichiometric vectors. The stoichiometric vector of reaction  $\mathcal{R}_k$ is given by $\zeta_k:=\zeta_k^r-\zeta_k^\ell\in\mathbb{Z}^n$ where $\zeta_k^r=\col(\zeta_{k,1}^r,\ldots,\zeta_{k,n}^r)$ and $\zeta_k^l=\col(\zeta_{k,1}^l,\ldots,\zeta_{k,n}^l)$. %The stoichiometry matrix $S\in\mathbb{Z}^{n\times K}$ is defined as $S:=\begin{bmatrix}  \zeta_1&\ldots&\zeta_K\end{bmatrix}$.
Each reaction $\mathcal{R}_k$ is also described by its propensity function $\lambda_k$ that describes the conditions under which this reaction occurs. Such functions may take various forms depending on the context and the type of kinetics, such as mass-action, Hill, Michaelis-Menten, etc.; see \cite{Voit:00,Alon:07}. In particular, when the network only has mass-action kinetics, only the reaction rates are indicated in \eqref{eq:RN}. This will be explicitly mentioned when this is the case. In all those cases, we define $\mathcal{P}$ to be the set of the network parameters.\\

In the deterministic setting, reaction networks are quantitatively described in terms of a vector of concentrations, denoted here by $x(t)$, which evolves on the state-space $\mathcal{S}\subseteq\mathbb{R}_{\ge0}^n$. In that case, the propensity functions $\lambda_k:\mathcal{S}\mapsto\mathbb{R}_{\ge0}$ are defined in such a way that $\mathcal{S}$ is forward invariant; i.e. for all $x_0\in\mathcal{S}$, we have that $x(t)\in\mathcal{S}$ for all $t\ge0$. This will be tacitly assumed to be the case in the rest of the paper. The dynamical model representing the deterministic reaction network \eqref{eq:RN} is therefore given by the Reaction Rate Equation (RRE)
\begin{equation}\label{eq:RRE}
\begin{array}{rcl}
  \dot{x}(t)&=&\displaystyle\sum_{k=1}^K\zeta_k\lambda_k(x(t)),\ t\ge0\\ % S\lambda(x(t))=
  x(0)&=&x_0
\end{array}
\end{equation}
which takes the form of a system of differential equations, emphasizing that reactions are considered as continuous processes which all push the state in the direction given by their associated stoichiometric vector.}

\subsection{Perfect Adaptation Problem and Integral Control}

\black{The perfect adaptation property \cite{Alon:07} is the property of a reaction network that certain molecular counts will return to the equilibrium levels they had before some environmental changes and network perturbations. This is one of the many types of homeostatic behaviors that can be found in living organisms \cite{Cannon:29} or that can be theoretically constructed \cite{Alon:07,Ma:09,Briat:15e,Golubitsky:17,Khammash:21,Gupta:22}. The version we consider in this paper is formally defined below:
\begin{define}[Perfect Adaptation]\label{def:RPA}
    Consider a reaction network $(\Xz,\mathcal{R})$ and a species $\Yz\in\Xz$. The species $\Yz$ is said to exhibit the perfect adaptation property in the network $(\Xz,\mathcal{R})$ if
  \begin{enumerate}
  %\item the concentrations $x(t)$ of the molecular species $\Xz$ are bounded at all times,
  %\item The steady-state value of the controlled species is equal to the set-point $y^*$, provided that this set-point is admissible, that is $y(t)\to y^*$ as $t\to\infty$.
  \item the equilibrium point $x^*$ is (locally) asymptotically stable for the dynamics of the network \eqref{eq:RRE}, and
  \item the equilibrium value of the concentrations of the species $\Yz$, denoted by $y^*$, is independent of all the parameters in a subset of $\mathcal{P}$.
  \end{enumerate}
\end{define}}

\black{Perfect adaptation is a very strong property. Indeed, almost all networks will not exhibit it and a natural problem that comes to mind is: given a network that does not show any adaptation property, how can we modify it in an appropriate way so that it exhibits this property? This can either be done by suitably modifying the current network, or through the introduction of new molecular species and reactions. The latter approach is formalized below:}
%
%\begin{define}[Robust Perfect Adaptation Problem]\label{def:RPAP}
%  Consider a reaction network $(\Xz,\mathcal{R})$ as described above and define $\Yz\in\{\X{1},\ldots,\X{n}\}$ to be the so-called \emph{controlled species}. The robust perfect adaptation problem consists of finding an additional set of molecular species and reactions $(\Zz,\mathcal{R}^c)$ such that the controlled species reaction $\Yz$ in the network $(\Xz\cup\Zz,\mathcal{R}\cup\mathcal{R}^c)$ exhibits the following properties:
%  %
%  \begin{enumerate}
%  \item The concentrations $(x(t),z(t))$ of the molecular species $(\Xz,\Zz)$ are bounded and converging, possibly under some conditions.
%  \item The steady-state value of the controlled species $y^*$ is equal to a specified and tunable set-point $r$, possibly under some conditions. %provided that this set-point is admissible, that is $y(t)\to y^*$ as $t\to\infty$.
%  \item The controlled species exhibits robust perfect adaptation properties meaning that $y(t)$ goes back to the set-point value $r$ under the presence of a certain class of stimuli, possibly under some conditions.
%  \end{enumerate}
%\end{define}
\black{\begin{define}[Perfect Adaptation Problem]\label{def:RPAP}
  Consider a reaction network $(\Xz,\mathcal{R})$ and a species $\Yz\in\Xz$. The robust perfect adaptation problem consists of finding a controller network $(\Zz,\mathcal{R}^c)$ such that
  \begin{enumerate}
  \item $\Yz$ exhibits the perfect adaptation property in the network $(\Xz\cup\Zz,\mathcal{R}\cup\mathcal{R}^c)$, and
  \item $y^*$ is equal to a tunable set-point which is a known function of the controller parameters only (i.e. it is independent of the parameters of the network $(\Xz,\mathcal{R})$).
  \end{enumerate}
\end{define}}

\black{This problem is an instance of the regulation problem in control theory -- an ubiquitous problem in science and engineering-- which consists of finding a controller that makes a controlled system not only stable but also such that its output converges to a desired set-point. Many relevant control problems fall into this category such as temperature control, cruise control, auto-pilot design, etc. \cite{Baillieul:15} Different controller structures need to be considered depending on the class of perturbations/stimuli sought to be compensated for. This structure is dictated by the so-called Internal Model Principle \cite{Francis:76,Sontag:03}, which stipulates that controllers need to embed a model of the perturbations they aim at rejecting. This internal model principle is a fundamental result which goes beyond the realm of control theory and is also central in other fields such as neuroscience, cognition, etc. \cite{Bin:22}. When the perturbations/stimuli are assumed to be (locally) constant, this principle states that controllers that solves the regulation problem necessarily incorporate an integral action, whence the name of integral controllers. It has also been recently shown that reaction networks exhibiting robust perfect adaptation properties necessarily involved a hidden integral action \cite{Gupta:22,Araujo:23,Hirono:23} implemented in terms of molecular reactions.}\\

\black{A necessary condition for perfect adaptation is that the set-point be a valid value for the species $\Yz$. This leads us to the concept of set-point admissibility:
\begin{define}[Set-point admissibility]
    Consider a reaction network $(\Xz,\mathcal{R})$ and a species $\Yz\in\Xz$. Moreover, decompose $\mathcal{P}$ as $\mathcal{P}_f\cup\mathcal{P}_t$ where $\mathcal{P}_f$ is a set of fixed network parameters and $\mathcal{P}_t$ is a set of tunable network parameters.

  We say that a set-point $r$ is admissible for the species  $\Yz$ if, given the values for the parameters in $\mathcal{P}_f$, we can select values for the parameters in $\mathcal{P}_t$ such that $x^*\ge0$ and $y^*=r$ where $x^*$ is an equilibrium point for the  dynamics \eqref{eq:RRE}.
\end{define}}

\subsection{Antithetic Integral Controllers}

\black{The Antithetic Integral Controller (AIC), introduced in \cite{Briat:15e} and further discussed in \cite{Briat:19:Logistic,Olsman:19,Olsman:19b}, is one of the few, structurally simple, integral controllers that can be implemented in terms of chemical reactions. Other ones include zeroth-order integral controllers \cite{Ni:09} and autocatalytic integral controllers \cite{Drengstig:12,Briat:16a,Xiao:18,Briat:19:Logistic}. This controller relies on molecular sequestration as core mechanism, which consists of two complementary molecular species strongly binding with each other. Sequestration is a well-known mechanism that has been widely selected to achieve certain pivotal functions within living organisms, such as bacterial stress response through the use of sigma-factors \cite{Trevino:13}. More theoretically, it has also been shown that such a sequestration reaction is essential in the realization problem \cite{Oishi:10b,Fages:17} and that it can be interpreted as a molecular implementation of a subtraction operator \cite{Cuba:23,Filo:23b}, as also emphasized in \cite{Briat:15e,Briat:19:Logistic}. The internal complementarity structure of the AIC makes it a very flexible controller that can easily implement negative and positive feedback loops using both activation or inhibition interactions, as illustrated in Table~\ref{fig:AICs}. More advanced controller structures than those depicted in Table~\ref{fig:AICs} exist and can be cleverly generated through the addition of controller reactions/species to achieve more complex regulatory behaviors \cite{Filo:22}. The common denominator to all those designs is that the set-point should activate the controlled species, which should in turn repress itself, both possibly through a sequence of reactions.}

\begin{table}[H]
  \centering
  %\begin{tabular}{|c||c|c|c|}
  %\hline
  \begin{tblr}{
             colspec = {|Q[c,m]||Q[c,m]|Q[c,m]|},  %{Q[c, wd=3cm] Q[l, wd=4cm] Q[c, wd=3.5cm]}
             %rowspec = {Q[m]||Q[b]|Q[m]},  %{Q[c, wd=3cm] Q[l, wd=4cm] Q[c, wd=3.5cm]}
             %
             row{1}  = {font=\bfseries},
             %row{2-Z} ={rowsep=7pt},
             %
             column{1}  = {font=\bfseries},
             }
  \toprule[1.5pt]
  & Activating System & Inhibiting System\\
%  & n-type AIC & p-type AIC\\
%  %
%  \midrule
  %
    &  \parbox{5cm}{\centering \includegraphics[width=.25\textwidth]{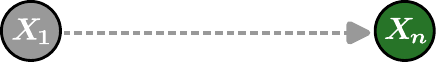}} & \parbox{5cm}{\centering\includegraphics[width=.15\textwidth]{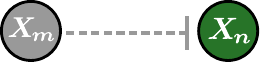}}\\
  \midrule
    {Activating\\AIC}  &  \parbox{5cm}{\includegraphics[width=.3\textwidth]{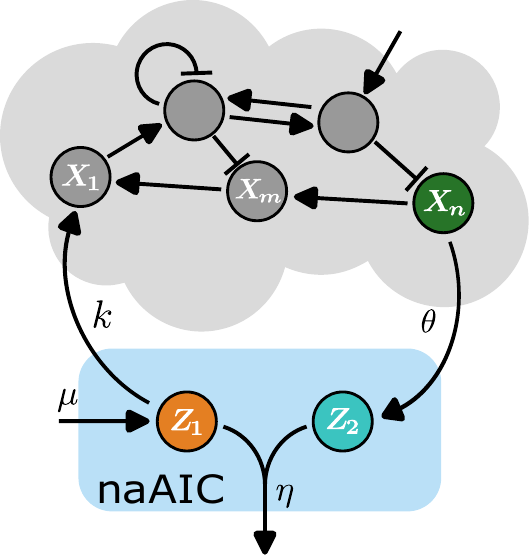}} & \parbox{5cm}{\includegraphics[width=.3\textwidth]{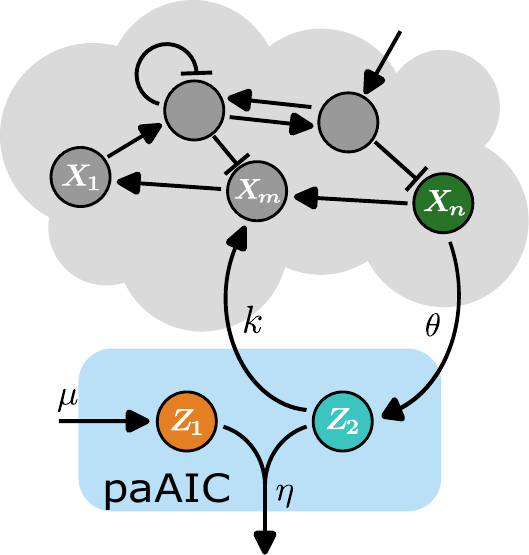}}\\% & \parbox{5cm}{\includegraphics[width=0.3\textwidth]{./Figures/Network_nAIRC.pdf}}\\
    \midrule
    {Inhibiting\\AIC} & \parbox{5cm}{\includegraphics[width=0.3\textwidth]{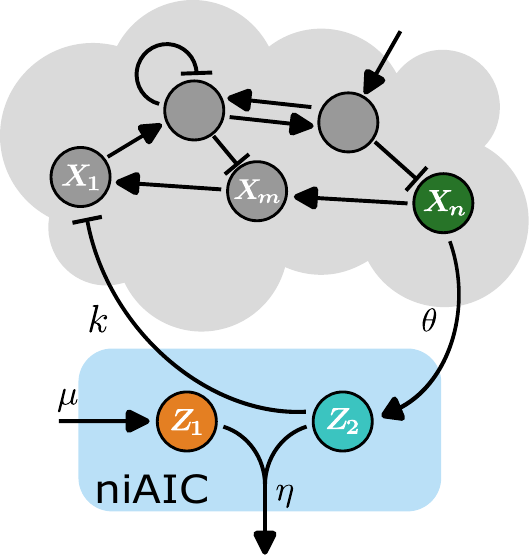}}  & \parbox{5cm}{\includegraphics[width=0.3\textwidth]{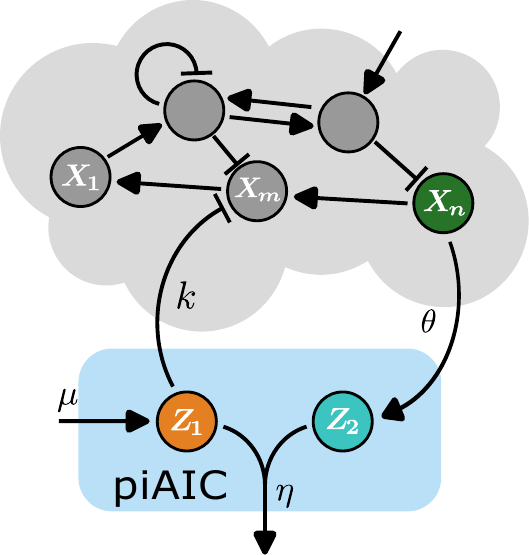}}\\% & \parbox{5cm}{\includegraphics[width=0.3\textwidth]{./Figures/Network_pAIRC.pdf}}\\
  \bottomrule[1.5pt]
  \end{tblr}
  \black{\caption{Four different flavors of the Antithetic Integral Controller depending on whether it activates an activating system (naAIC), inhibits an activating system (niAIC), activates an inhibiting system (paAIC), and inhibits an inhibiting system (piAIC).  This controller involves two additional molecular species, $\Z{1}$ and $\Z{2}$, and four additional reactions. The controllers have information about the level of the controlled species $\Yz=\X{n}$ through the sensing reaction and act back onto the system through the reaction reaction. The set-point is (partially) set by the set-point $\mu$ and, finally, the  essential sequestration reaction $\Z{1}+\Z{2}\rarrow{}\phib$ plays the role of a comparator between the levels of the controlled species.  In all those designs, we can observe that $\mu$ promotes the expression of $\X{n}$, which in turn represses itself through a sequence of reactions. It can be shown that all those controllers solve the perfect adaptation problem described in Definition~\ref{def:RPAP} in the sense that $y=^*x_n^*=\mu/\theta$ independently of the values of the parameters of the networks, provided those verify various technical assumptions.}\label{fig:AICs}}
\end{table}

\subsubsection{naAIC}

\black{The original version of the AIC, as introduced in \cite{Briat:15e}, corresponds to the naAIC setup depicted in Table~\ref{fig:AICs}. A possible reaction network representation for this controller is given by
\begin{equation}\label{eq:AIC:naAIC}
  %\phib\rarrow{\mu}\Z{1},\ \Yz\rarrow{\theta}\Yz+\Z{2},\ \Z{1}+\Z{2}\rarrow{\phi\eta}\phib,\  \Z{1}\rarrow{k}\Z{1}+\X{1},
  \phib\rarrow{\mu}\Z{1},\ \Yz\rarrow{\theta}\Yz+\Z{2},\ \Z{1}+\Z{2}\rarrow{\eta}\phib,\  \Z{1}\rarrow{k}\Z{1}+\X{1},
\end{equation}
where $\Z{1}$ is both the actuating and reference species, $\Z{2}$ is the sensing species, $\Yz$ is the measured/controlled species, and $\X{1}$ is the actuated species. The parameters $\mu,\theta,\eta,k$ are all positive rate parameters of the reactions, which are all assumed to be mass-action. In compliance with the naAIC structure, it is assumed here that $\X{1}$ activates $\Yz$, possibly through a sequence of reactions. The deterministic model of this controller is given by
\begin{equation}\label{eq:AIC:naAIC:det}
    \dot{z}_1(t)=\mu-\eta z_1(t)z_2(t)\quad\textnormal{and}\quad \dot{z}_2(t)=\theta y(t)-\eta z_1(t)z_2(t).
\end{equation}
To avoid entering too technical discussions, we simply illustrate that this controller indeed solves the perfect adaptation problem stated in Definition~\ref{def:RPAP} for the gene expression network
  \begin{equation}\label{eq:AIC:gene_exp}
\phib\rarrow{b_0}\X{1},\ \X{1}\rarrow{k_2}\X{1}+\X{2},\ \X{1}\rarrow{\gamma_1}\phib,\ \X{2}\rarrow{\gamma_2}\phib
  \end{equation}
  where $\X{1}$ and $\X{2}$ denote the mRNA and protein species, respectively. Above, the positive rate parameters $b_0,k_2,\gamma_1$, and $\gamma_2$ of the mass-action reactions correspond to the basal transcription rate, the translation rate, the mRNA degradation rate, and the protein degradation rate, respectively. We consider here the problem of ensuring perfect adaptation for the protein concentration, that is, we set $\Yz=\X{2}$. It can be shown that the basal equilibrium level $$(x_1^*,x_2^*)=\left(\dfrac{b_0}{\gamma_1},\dfrac{b_0k_2}{\gamma_1\gamma_2}\right)$$
  is asymptotically stable for the system
\begin{equation}
\begin{array}{rcl}
      \dot{x}_1(t)&=&-\gamma_1x_1(t)+b_0\\
      \dot{x}_2(t)&=&k_2x_1(t)-\gamma_2x_2(t)
\end{array}
\end{equation}
describing the network \eqref{eq:AIC:gene_exp}. Adding now the input $u$ so as to act on the transcription rate, we obtain the following system describing
\begin{equation}
\begin{array}{rcl}
      \dot{x}_1(t)&=&-\gamma_1x_1(t)+b_0+u(t)\\
      \dot{x}_2(t)&=&k_2x_1(t)-\gamma_2x_2(t)\\
      y(t)&=&x_2(t)
\end{array}
\end{equation}
from which we can observe that $y^*=\frac{(b_0+u^*)k_2}{\gamma_1\gamma_2}>\frac{b_0k_2}{\gamma_1\gamma_2}$. This means that one can only increase the output from its basal expression level and this will be so regardless the type of controllers that is considered.\\

When the set-point $r=\mu/\theta$ is admissible (i.e. $r>\frac{b_0k_2}{\gamma_1\gamma_2}$) and certain conditions are met for the closed-loop network consisting of the interconnection of \eqref{eq:AIC:gene_exp}-\eqref{eq:AIC:naAIC}
%The closed-loop network obtained from the interconnection of the AIC \eqref{eq:AIC:naAIC} and the gene expression network \eqref{eq:AIC:gene_exp} is given by
\begin{equation}
\begin{array}{rcl}
      \dot{x}_1(t)&=&-\gamma_1x_1(t)+b_0+kz_1(t)\\
      \dot{x}_2(t)&=&k_2x_1(t)-\gamma_2x_2(t)\\
      \dot{z}_1(t)&=&\mu-\eta z_1(t)z_2(t)\\
      \dot{z}_2(t)&=&\theta y(t)-\eta z_1(t)z_2(t),
\end{array}
\end{equation}
then this network exhibits perfect adaptation properties and $$\lim_{t\to\infty}x_2(t)=\frac{\mu}{\theta}.$$ When the admissibility condition is not met, however, there is no equilibrium point in the positive orthant, which implies it cannot include any compact forward invariant set, which is equivalent to saying that the solution must grow without bound \cite{Richeson:02,Richeson:04}. Those facts are illustrated in Figure~\ref{fig:naAIC}.}
%
%
%it can be proven that the steady-state value for the $x_2$ is necessarily larger than $\frac{b_0k_2}{\gamma_1\gamma_2}$ and that a set-point $\mu/\theta$ for the controlled species is admissible provided that $\mu/\theta>\frac{b_0k_2}{\gamma_1\gamma_2}$. When this condition is not met, the molecular species $\Z{2}$ grows unbounded simply because we are asking the system to behave in a way that is achievable and this is a property of the system and this constraint is independent of the controller structure. This also has little to do with the negativity of some entries of the equilibrium state-vector.

\begin{figure}
\begin{minipage}[h]{0.45\linewidth}
    \begin{center}
  \includegraphics[width=\textwidth]{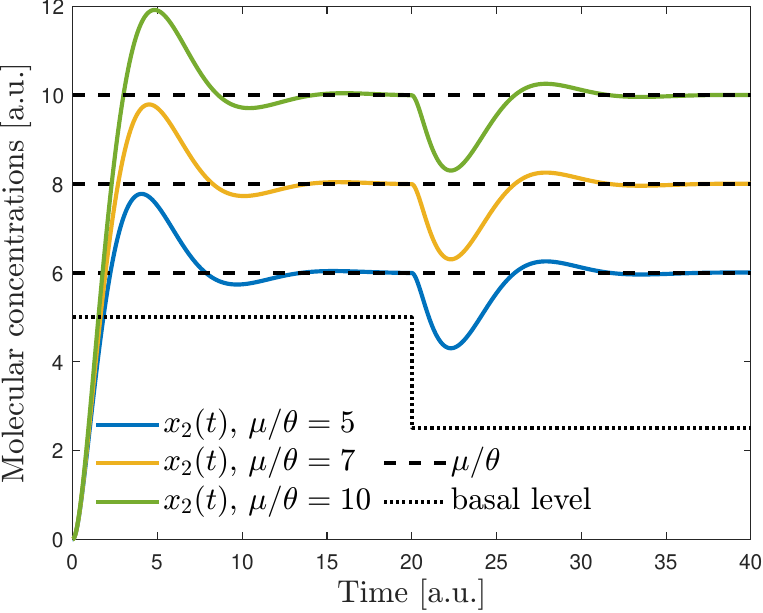}
  \end{center}
\end{minipage}
\begin{minipage}[h]{0.45\linewidth}
    \begin{center}
  \includegraphics[width=\textwidth]{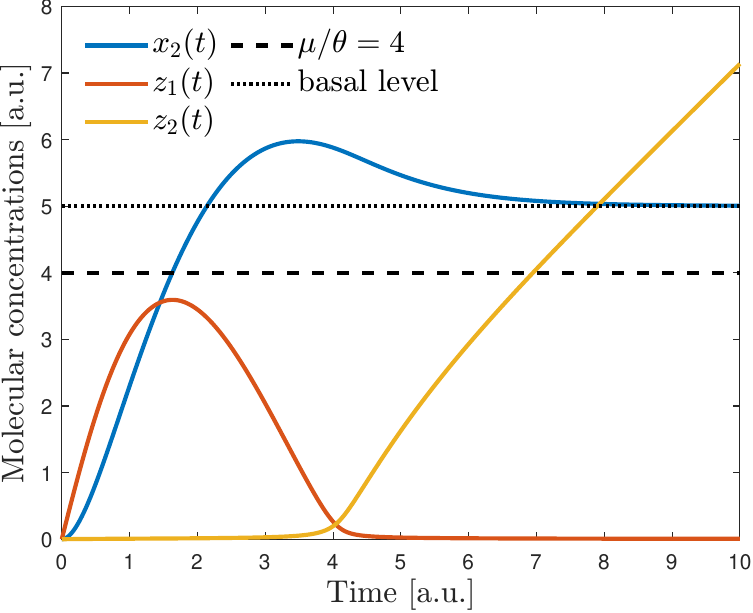}
  \end{center}
\end{minipage}
\caption{Simulation of the reaction network \eqref{eq:AIC:gene_exp}-\eqref{eq:AIC:naAIC} with the parameters $\gamma_1=1$, $\gamma_2=2$, $k_2=1$, $b_0=10$, $k=1$, $\eta=100$, $\theta=1$ and for various values for $\mu$. At time $t=20$, the value of $b_0$ is divided by 2. When the set point $\mu/\theta$ is set above the basal level, the controlled species $\X{2}$ exhibits the robust perfect adaptation property. When the set-point is set to a value lower than the basal level, the concentration of controller species $\Z{2}$ grows without bound and perfect adaptation is not achieved for $\X{2}$.}\label{fig:naAIC}
\end{figure}

\subsubsection{niAIC}

\black{The niAIC shown in Table~\ref{fig:AICs} is another important class of AICs, which will turn out to be an essential part of this paper. This structure was briefly mentioned in \cite{Briat:15e} without entering into technical details. A possible reaction network representation for this controller is given by
\begin{equation}\label{eq:AIC:niAIC}
  \phib\rarrow{\mu}\Z{1},\ \Yz\rarrow{\theta}\Yz+\Z{2},\ \Z{1}+\Z{2}\rarrow{\eta}\phib,\  \X{m}+\Z{2}\rarrow{k}\Z{2}
\end{equation}
where $\Z{1}$ is the reference species, $\Z{2}$ is both the sensing and actuating species, $\Yz$ is the measured/controlled species, and $\X{m}$ is the actuated species. The parameters $\mu,\theta,\eta,k$ are all positive rate parameters of the reactions, which are all assumed to be mass-action. In compliance with the niAIC structure, it is assumed here that $\X{m}$ represses $\Yz$, possibly through a sequence of reactions. The deterministic model of this controller is also given by \eqref{eq:AIC:naAIC:det}.\\

As in the previous section, we illustrate the controller properties on the gene expression network \eqref{eq:AIC:gene_exp} to which we add the control input according to the structure of the controller \eqref{eq:AIC:niAIC} as follows
\begin{equation}
\begin{array}{rcl}
      \dot{x}_1(t)&=&-\gamma_1x_1(t)+b_0\\
      \dot{x}_2(t)&=&k_2x_1(t)-\gamma_2x_2(t)-u(t)x_2(t)\\
      y(t)&=&x_2(t).
\end{array}
\end{equation}
From this expression we can observe that $y^*=\frac{b_0k_2}{\gamma_1(\gamma_2+u^*)}<\frac{b_0k_2}{\gamma_1\gamma_2}$. This means that one can only decrease the output from its basal expression level and this will be so regardless the type of controllers that is considered.\\

When the set-point $r=\mu/\theta$ is admissible (i.e. $r<\frac{b_0k_2}{\gamma_1\gamma_2}$) and certain conditions are met for the closed-loop network consisting of the interconnection of \eqref{eq:AIC:gene_exp}-\eqref{eq:AIC:niAIC}
%The closed-loop network obtained from the interconnection of the AIC \eqref{eq:AIC:naAIC} and the gene expression network \eqref{eq:AIC:gene_exp} is given by
\begin{equation}
\begin{array}{rcl}
      \dot{x}_1(t)&=&-\gamma_1x_1(t)+b_0\\
      \dot{x}_2(t)&=&k_2x_1(t)-\gamma_2x_2(t)-kz_2(t)x_2(t)\\
      \dot{z}_1(t)&=&\mu-\eta z_1(t)z_2(t)\\
      \dot{z}_2(t)&=&\theta y(t)-\eta z_1(t)z_2(t),
\end{array}
\end{equation}
then this network exhibits perfect adaptation properties and $$\lim_{t\to\infty}x_2(t)=\frac{\mu}{\theta}.$$ As for the naAIC, the solution of this system must grow without bound whenever this admissibility condition is not met, as illustrated in Figure~\ref{fig:niAIC}.}

%One can observe that, as opposed to the naAIC case presented before, the controller inhibits $\X{1}$, which means that the actuation will only be one sided and can only decrease with respect to the basel level. This implies that the set-point can only be smaller than that the basal level, which may be seen as a limitation to the current architecture based on the inhibition of $\X{1}$. When the condition $\frac{\mu}{\theta}<\frac{b_0k_2}{\gamma_1\gamma_2}$ and certain conditions on the network are met, the niAIC solves the perfect adaptation problem for the network \eqref{eq:AIC:gene_exp} and we have that $\lim_{t\to\infty}x_2(t)=\frac{\mu}{\theta}$. When this condition is not met, however, there is no equilibrium point in the positive orthant, which implies it cannot include any compact forward invariant set, which is equivalent to saying that the solution must grow without bound \cite{Richeson:02,Richeson:04}. This is illustrated in Figure~\ref{fig:naAIC}.

\begin{figure}
\begin{minipage}[h]{0.45\linewidth}
    \begin{center}
  \includegraphics[width=\textwidth]{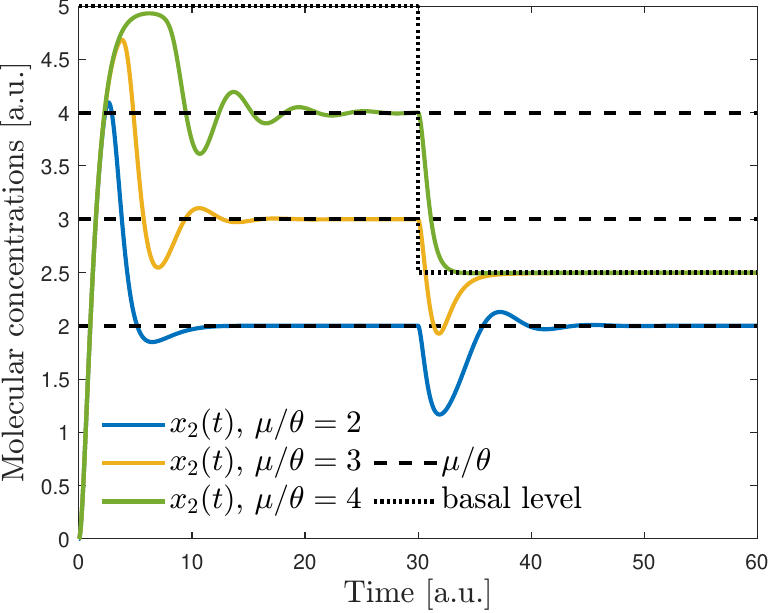}
  %\caption{}
  \end{center}
\end{minipage}
\begin{minipage}[h]{0.45\linewidth}
    \begin{center}
  \includegraphics[width=\textwidth]{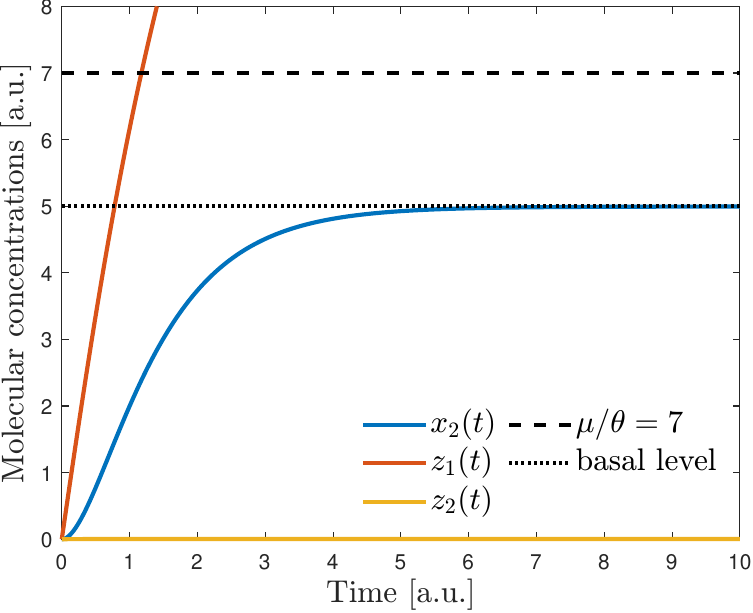}
  %\caption{}
  \end{center}
\end{minipage}
\caption{Simulation of the reaction network \eqref{eq:AIC:gene_exp}-\eqref{eq:AIC:niAIC} with the parameters $\gamma_1=1$, $\gamma_2=2$, $k_2=1$, $b_0=10$, $k=0.5$, $\eta=100$, $\theta=1$ and for various values for $\mu$. At time $t=30$, the value of $b_0$ is divided by 2. When the set point $\mu/\theta$ is set below the basal level, the controlled species $\X{2}$ exhibits the robust perfect adaptation property. When the set-point is set to a value higher than the basal expression level, the concentration of the controller species $\Z{1}$ grows without bound and perfect adaptation is not achieved for $\X{2}$, which can only converge to its basal level.}\label{fig:niAIC}
\end{figure}

\subsection{Antithetic Integral Rein Controllers}

\subsubsection{Generalities}

\black{In light of the previous discussion, it seems interesting to combine both strategies in order to eliminate the limitations on the set of admissible set-points. This may be achieved through the consideration of Antithetic Integral Rein Controllers (AIRCs) that make use of both controller species in order to actuate the system where one would activate the output while the other one would inhibit it. Different versions for the AIRC are depicted in Table~\ref{fig:AIRCs_simple}. The use of both species has been considered in the past in the context of Brink controllers \cite{Cuba:21CellSystems} where both positive and negative actuation were performed to modulate a phosphorylation cycle. It is also expected that a two-way actuation would greatly improve the temporal behavior of the closed-loop network as we would not be relying anymore on the natural convergence properties of the system, which may be slow. Indeed, inhibition is expected to improve decreasing properties of the output while activation should improve the increasing properties.}

%\begin{table}[H]
%  \centering
%  %\begin{tabular}{|c||c|c|c|}
%  %\hline
%  \begin{tblr}{
%             colspec = {|Q[c,m]||Q[c,m]|Q[c,m]|},  %{Q[c, wd=3cm] Q[l, wd=4cm] Q[c, wd=3.5cm]}
%             %rowspec = {Q[m]||Q[b]|Q[m]},  %{Q[c, wd=3cm] Q[l, wd=4cm] Q[c, wd=3.5cm]}
%             %
%             row{1}  = {font=\bfseries},
%             %row{2-Z} ={rowsep=7pt},
%             %
%             column{1}  = {font=\bfseries},
%             }
%  \toprule[1.5pt]
%  $Z_1/Z_2$ & Activation & Inhibition\\
%  %
%  \midrule
%  %
%    Activation  &  \parbox{5cm}{\includegraphics[width=.3\textwidth]{./Figures/Network_aaAIRC.pdf}} & \parbox{5cm}{\includegraphics[width=0.3\textwidth]{./Figures/Network_aiAIRC.pdf}}\\
%    %
%    \midrule
%  %
%    Inhibition   & \parbox{5cm}{\includegraphics[width=.3\textwidth]{./Figures/Network_iaAIRC.pdf}} & \parbox{5cm}{\includegraphics[width=0.3\textwidth]{./Figures/Network_iiAIRC.pdf}}\\
%  \bottomrule[1.5pt]
%  \end{tblr}
%  \caption{Different structrues of the Antithetic Integral Rein Controller depending on whether $\Z{1}$ and $\Z{2}$ activate or repress their respective actuated species. In the network, both $\X{i}$ and $\X{k}$ activate the output whereas both $\X{j}$ and $\X{\ell}$ inhibit it.}\label{fig:AIRCs}
%\end{table}

\begin{table}[H]
  \centering
  %\begin{tabular}{|c||c|c|c|}
  %\hline
  \begin{tblr}{
             colspec = {|Q[c,m]||Q[c,m]|Q[c,m]|},  %{Q[c, wd=3cm] Q[l, wd=4cm] Q[c, wd=3.5cm]}
             %rowspec = {Q[m]||Q[b]|Q[m]},  %{Q[c, wd=3cm] Q[l, wd=4cm] Q[c, wd=3.5cm]}
             %
             row{1}  = {font=\bfseries},
             %row{2-Z} ={rowsep=7pt},
             %
             column{1}  = {font=\bfseries},
             }
  \toprule[1.5pt]
  \diagbox{$Z_1$}{$Z_2$} & {Activation\\ \ \\ \includegraphics[width=.13\textwidth]{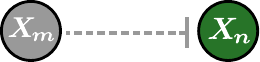}} & {Inhibition\\ \ \\ \includegraphics[width=.13\textwidth]{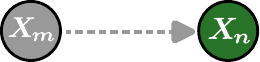}}\\
  %%
%  & \parbox{5cm}{\centering \includegraphics[width=.13\textwidth]{./Figures/header_activation_Xm_small.pdf}} & \parbox{5cm}{\centering\includegraphics[width=.13\textwidth]{./Figures/header_inhibition_Xm_small.pdf}}\\
%  %
  \midrule
    {Activation\\ \ \\ \includegraphics[width=.13\textwidth]{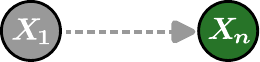}}  &  \parbox{5cm}{\includegraphics[width=.3\textwidth]{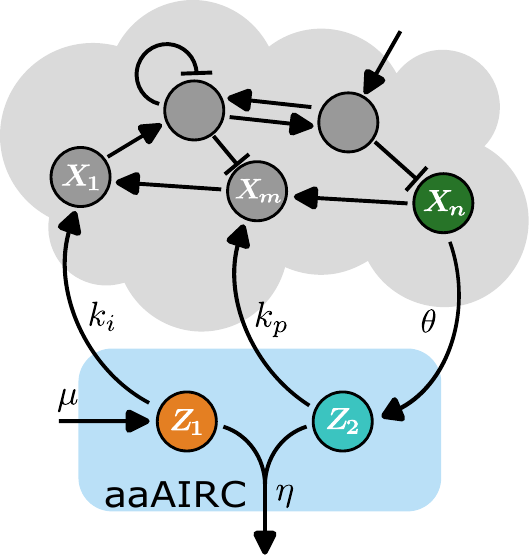}} & \parbox{5cm}{\includegraphics[width=0.3\textwidth]{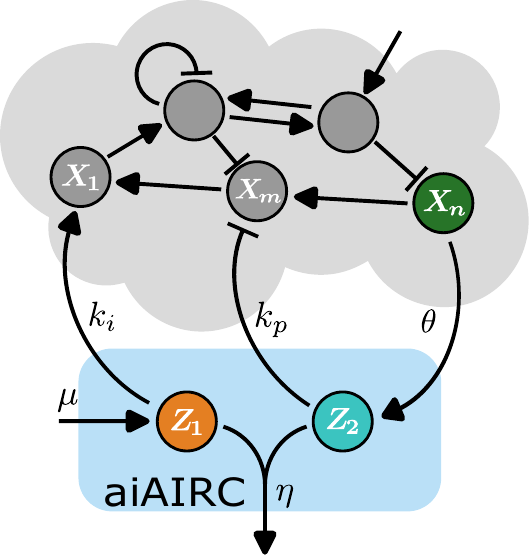}}\\
    \midrule
    {Inhibition\\ \ \\ \includegraphics[width=.13\textwidth]{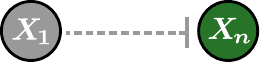}} & \parbox{5cm}{\includegraphics[width=.3\textwidth]{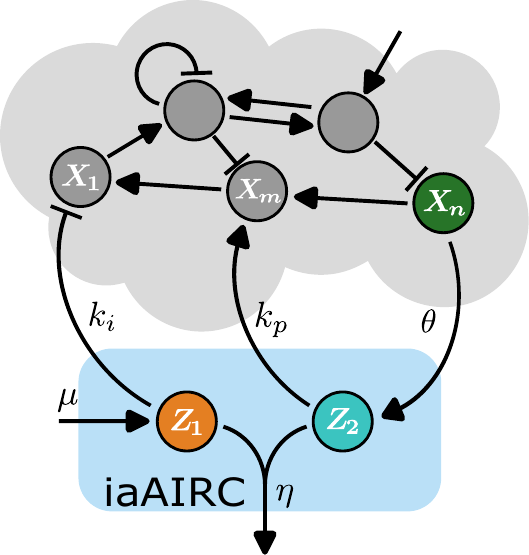}} & \parbox{5cm}{\includegraphics[width=0.3\textwidth]{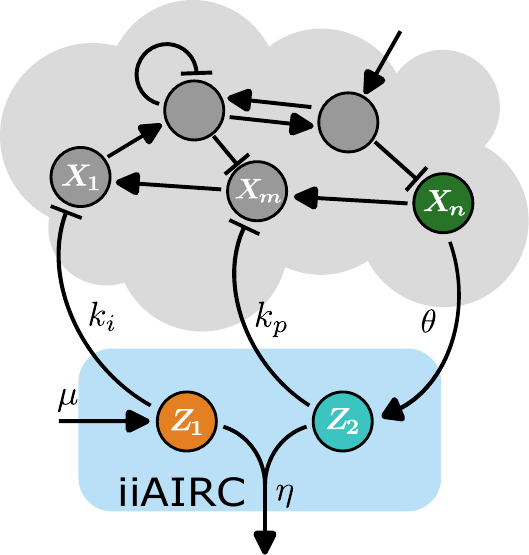}}\\
  \bottomrule[1.5pt]
  \end{tblr}
  \caption{Different structures of the Antithetic Integral Rein Controller depending on whether $\Z{1}$ and $\Z{2}$ activate or repress their respective actuated species. Once again, the key idea is that $\mu$ must activate $\X{n}$, which should in turn repress itself.}\label{fig:AIRCs_simple}
\end{table}

\subsubsection{Output Rein Control}

\black{A particular instance of the AIRC, depicted in Figure~\ref{fig:oaiAIC}, is when the output is directly repressed by the sensing species while the reference species indirectly activate the output through $\X{1}$. This specific instance, first considered in \cite{Gupta:19} and also discussed in \cite{Plesa:23}, has been shown to ensure better convergence properties in the case of a simple gene expression network. A reaction network implementation of the output aiAIRC is given by
\begin{equation}\label{eq:AIC:AIRC}
  \phib\rarrow{\mu}\Z{1},\ \Yz\rarrow{\theta}\Yz+\Z{2},\ \Z{1}+\Z{2}\rarrow{\eta}\phib,\  \Yz+\Z{2}\rarrow{k_p}\Z{2}, \Z{1}\rarrow{k_i}\Z{1}+\X{1}
\end{equation}
where $\Z{1}$ is both the reference and the activating species, $\Z{2}$ is both the sensing and inhibiting species, $\Yz$ is the measured/controlled species, and both $\X{1}$ and $\X{m}$ are actuated species. The parameters $\mu,\theta,\eta,k_p,k_i$ are all positive rate parameters of the reactions, which are all assumed to be mass-action. In compliance with the aiAIRC structure, it is assumed here that $\X{1}$ activates $\Yz$, possibly through a sequence of reactions. The deterministic model of this controller is also given by \eqref{eq:AIC:naAIC:det}.}\\

\begin{figure}[H]
  \centering
  \includegraphics[width=.3\textwidth]{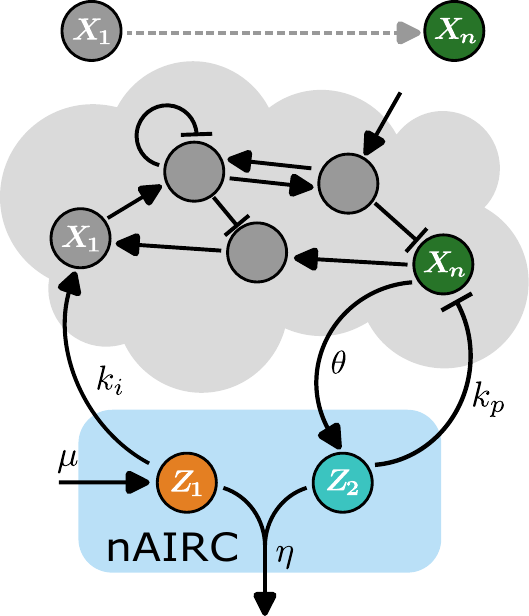}
  \caption{Output aiAIRC: The sensing species $\Z{2}$ directly represses the output $\Yz=\X{n}$ whereas the reference species $\Z{1}$ indirectly activate it by activating $\X{1}$.}\label{fig:oaiAIC}
\end{figure}

\black{We again illustrate the benefits of this controller structure on the previously introduced gene expression network to which we add now two control inputs $u_1$ and $u_2$ as
\begin{equation}
\begin{array}{rcl}
      \dot{x}_1(t)&=&-\gamma_1x_1(t)+b_0+u_1(t)\\
      \dot{x}_2(t)&=&k_2x_1(t)-\gamma_2x_2(t)-u_2(t)x_2(t)\\
      y(t)&=&x_2(t).
\end{array}
\end{equation}
It can be shown that at stationarity we have that $$y^*=\dfrac{k_2(b_0+u_1^*)}{\gamma_1(\gamma_2+u_2^*)},$$ which implies that for all $r>0$, there will exist $u_1^*,u_2^*\ge0$ such that $y^*=r$. This demonstrates that this specific controller structure enables a broader class of admissible set-point values than the naAIC and niAIC admissible set-points by combining them. The dynamical model of the interconnection of that controller and the gene expression network \eqref{eq:AIC:gene_exp} is now given by
\begin{equation}
\begin{array}{rcl}
      \dot{x}_1(t)&=&-\gamma_1x_1(t)+b_0+k_iz_1(t)\\
      \dot{x}_2(t)&=&k_2x_1(t)-\gamma_2x_2-k_pz_2(t)x_2(t)\\
      \dot{z}_1(t)&=&\mu-\eta z_1(t)z_2(t)\\
      \dot{z}_2(t)&=&\theta y(t)-\eta z_1(t)z_2(t).
\end{array}
\end{equation}
%It can be shown that the equilibrium value for $x_2^*$ is given by
%\begin{equation}
%  x_2^*=\dfrac{g_0+g_1k_iz_1^*}{1+g_nk_pz_2^*}
%\end{equation}
%where $\mu-\phi z_1^*z_2^*=0$, which shows that the equilibrium value for $x_2^*$ may take any nonnegative values.
Under some conditions, it is possible to show that $\lim_{t\to\infty}x_2(t)=\frac{\mu}{\theta}$ will hold under some stability conditions and with no restriction on the set-point (i.e. the admissible set is the set of nonnegative numbers). This is illustrated in Figure~\ref{fig:nAIRC}.}

\begin{figure}[h]
  \centering
  \includegraphics[width=0.45\textwidth]{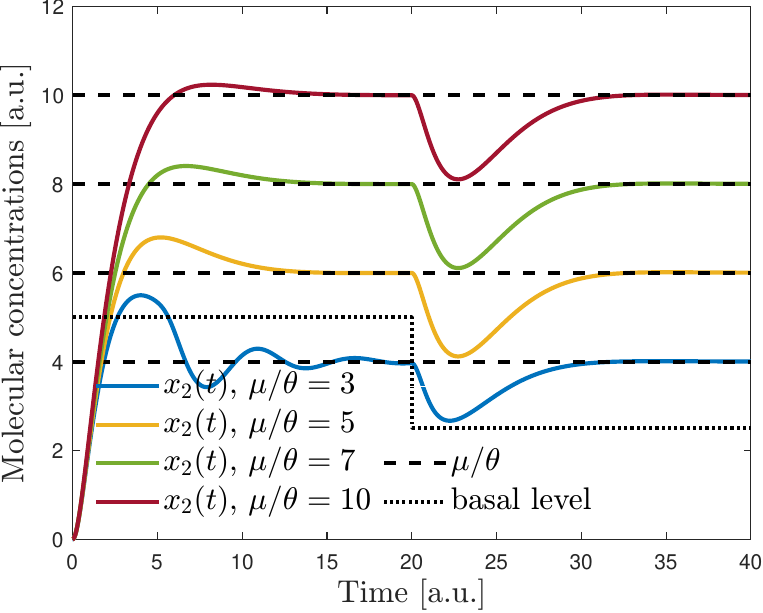}
  \caption{Simulation of the reaction network \eqref{eq:AIC:gene_exp}-\eqref{eq:AIC:AIRC} with the parameters $\gamma_1=1$, $\gamma_2=2$, $k_2=1$, $b_0=10$, $k_p=k_i=0.5$, $\eta=100$, $\theta=1$ and for various values for $\mu$. At time $t=30$, the value of $b_0$ is divided by 2. One can observe that the controlled species $\X{2}$ exhibits the robust perfect adaptation property for set-points that are both below and above the basal level, which illustrates the benefits of such a controller.}\label{fig:nAIRC}
\end{figure}
%
%\subsection{Comments???}
%
%
%
%\red{Say something for the AIC, eta,theta and k can be innocuous in the stoch. Determiistic no, SIAM paper, and show bifurcation diagram. Strict passivity provides structural stability. \\
%
%Passivity plays a role as also discuss in Cell Systens, ACS, SIAM, etc. Stoch vs deterministic, innocuousness. Why this is important, robustness,}

\section{Unimolecular mass-action networks}\label{sec:ptype:lin}

\subsection{The general problem and its properties}

\black{Let us consider a mass-action unimolecular reaction network $(\Xz,\mathcal{R})$ described by the following dynamical system
\begin{equation}\label{eq:mainsystL}
  \begin{array}{rcl}
    \dot{x}(t)&=&Ax(t)+b_0\\
    %y(t)&=&e_n^Tx(t)\\
    x(0)&=&x_0
  \end{array}
\end{equation}
where $x,x_0\in\mathbb{R}_{\ge0}^n$  are the state of the system and the initial condition, respectively. The matrix $A$ is Metzler (i.e. all its off-diagonal elements are nonnegative) while the vector $b_0$ is nonnegative. Connecting that network to the output aiAIRC network previously described yields the model
\begin{equation}\label{eq:mainsystCL}
  \begin{array}{rcl}
    \dot{x}(t)&=&Ax(t)+e_1k_iz_1(t)-e_n  x_n(t) k_pz_2(t)+b_0\\
    \dot{z}_1(t)&=&\mu-\eta z_1(t)z_2(t)\\
    \dot{z}_2(t)&=&\theta x_n (t)-\eta z_1(t)z_2(t)
  \end{array}
\end{equation}
where $z_1,z_2\in\mathbb{R}_{\ge0}$ are the states of the controller, and $\eta,k_p,k_i>0$ are the parameters of the controller. Again, the set-point for the equilibrium output is given by $r=\mu/\theta$. This setup is a generalization of the one considered in \cite{Gupta:19} where a less general class of systems was considered.}\\

%We consider the following Antithetic Integral Rein Controller
%\begin{equation}
%  \begin{array}{rcl}
%    \dot{z}_1(t)&=&\mu-\phi\eta z_1(t)z_2(t)\\
%    \dot{z}_2(t)&=&\theta x_n(t)-\phi\eta z_1(t)z_2(t)\\
%    u_1(t)&=&k_i z_1(t)\\
%    u_2(t)&=&k_p z_2(t)
%  \end{array}
%\end{equation}
% The resulting closed-loop system is then given by
%\begin{equation}\label{eq:mainsystCL}
%  \begin{array}{rcl}
%    \dot{x}(t)&=&Ax(t)+e_1k_iz_1(t)-e_n  x_n(t) k_pz_2(t)+b_0\\
%    \dot{z}_1(t)&=&\mu-\phi\eta z_1(t)z_2(t)\\
%    \dot{z}_2(t)&=&\theta x_n (t)-\phi\eta z_1(t)z_2(t)
%  \end{array}
%\end{equation}
%where we can see that the species $z_1$ of the controller acts on the production of $x_n$ by acting on $x_1$ whereas $z_2$ acts on the degradation of the controlled species.
%
%We need the following definition of admissibility of a set-point:
%\begin{define}
%  The set-point $r:=\mu/\theta$ is set to be admissible for the network \eqref{eq:mainsystCL} if there exists a nonnegative equilibrium point $(x^*,z_1^*,z_2^*)$ for \eqref{eq:mainsystCL} such that $x_n^*=r$.
%\end{define}
%
%This admissibility property simply defines what set-points can be achieved. It was previously considered in \cite{Briat:19:Logistic} as a necessary condition for the regulation of the output around a desired set-point. Non-admissibility of a set-point is often due to the presence of saturations and/or basal levels which will place upper and lower limits on the equilibrium molecular levels.

%Interestingly, the following result holds
%
%\begin{proposition}
%  Assume that the set-point $r$ is not
%\end{proposition}

\black{The proposition below states that all positive set-points are admissible for the system \eqref{eq:mainsystCL}:
\begin{proposition}\label{prop:eqpt}
  Assume that $A$ is Hurwitz stable, that $r=\mu/\theta>0$, and that $-e_n ^TA^{-1}e_1\ne0$. Then, the equilibrium point $(x^*,z_1^*,z_2^*)$ of the closed-loop system \eqref{eq:mainsystCL} is unique and nonnegative. It is, moreover, given by
  \begin{equation}
    \begin{array}{rcl}
      x^*&=&-A^{-1}\left(e_1k_iz_1^*-\dfrac{e_n  r k_p\mu}{\eta z_1^*}+b_0\right)\\
      z_2^*&=&\dfrac{\mu}{\eta z_1^*}
    \end{array}
  \end{equation}
  where $z_1^*$ is the unique positive root to the polynomial
  \begin{equation}\label{eq:P}
  P_1(z_1):=\eta g_1 k_iz_1^{2}+(g_0-r)\eta z_1-g_n  k_p\mu r,
\end{equation}
  where $r:=\mu/\theta$, $g_1:=-e_n ^TA^{-1}e_1$, $g_n :=-e_n ^TA^{-1}e_n $, and $g_0:=-e_n ^TA^{-1}b_0$.\\

  \noindent Alternatively, the equilibrium point for the controller species can be characterized as $z_1^*=\mu/(\eta z_2^*)$ where $z_2^*$ is the unique positive root of the polynomial
  \begin{equation}
  P_2(z_2) = -\eta g_n  k_p rz_2^{*2}+(g_0-r)\eta z_2^*+g_1 k_i\mu.
\end{equation}
\end{proposition}}
\black{\begin{proof}
  The equilibrium points solve the expressions
\begin{equation}\label{eq:jdskjdksdjk}
  \begin{array}{rcl}
    Ax^*+e_1k_iz_1^*-e_n  x_n ^* k_pz_2^*+b_0&=&0,\\
    \mu-\eta z_1^*z_2^*&=&0,\\
    \theta x_n ^*-\eta z_1^*z_2^*&=&0
  \end{array}
\end{equation}
which implies that $ z_1^*z_2^*\ne0$. Subtracting the two last rows in \eqref{eq:jdskjdksdjk} yields $x_n ^*=r$ and
  \begin{equation}
  \begin{array}{rcl}
    Ax^*+e_1k_iz_1^*-e_n  x_n^* k_p z_2^*+b_0&=&0,\\
    \mu-\eta z_1^*z_2^*&=&0.
  \end{array}
\end{equation}
This implies that $x^*$ can be written as
\begin{equation}
  x^*=-A^{-1}\left(e_1k_iz_1^*-e_n  k_p r z_2^*+b_0\right),
\end{equation}
which implies
\begin{equation}
\begin{array}{rcl}
    r&=&e_n^Tx^*\\
    &=&-e_n^TA^{-1}\left(e_1k_iz_1^*-e_n  r k_p z_2^*+b_0\right)\\
    &=& g_1k_iz_1^*-g_nr k_p z_2^*+g_0.
\end{array}
\end{equation}
Substituting $z_2^*=\mu/(\eta z_1^*)$ and multiplying the resulting expression by $\eta z_1^*$ yields
\begin{equation}
  \eta g_1 k_iz_1^{*2}+(g_0-r)\eta z_1^*-g_n  k_p\mu r=0\textnormal{ and }z_2^*=\mu/(\eta z_1^*).
\end{equation}
Using the fact that $g_1,g_n,r >0$ (see Lemma \ref{lemma:gains}), one can see that there is exactly one change of sign in the coefficients of the polynomial. From Descartes' rule of sign \cite{Henrici:88}, this implies that there exists one positive root to this polynomial. The alternative statement with $z_2^*$ is proven analogously.\\

We need to prove now that the equilibrium point is nonnegative. First of all, $z_1^*$ is positive, which implies that $z_2^*$ is positive as well. Noting now that the equilibrium state value $x^*$ can be written as
\begin{equation}
  x^*=-(A-k_pe_ne_n^Tz_2^*)^{-1}\left(e_1k_iz_1^*+b_0\right)
\end{equation}
where $z_1^*,z_2^*>0$. Using now the fact that $A$ is Hurwitz stable, then so is $A-k_pe_ne_n^Tz_2^*$ and, therefore, $(A-k_pe_ne_n^Tz_2^*)^{-1}\le0$ (see Proposition \ref{prop:MHS}). This implies that $x^*\ge0$ since $e_1k_iz_1^*+b_0\ge0$, which proves that the equilibrium point is nonnegative.  Since there are no assumptions on the value for the set-point $r$ besides its positivity, we can conclude that all positive set-points are admissible.
\end{proof}}

\black{The closed-loop network exhibits an interesting switching behavior in the large $\eta$ limit as stated in the following result:
\begin{proposition}\label{prop:switching}
  The following statements hold:
  \begin{enumerate}
  \item If $r>g_0$, then we have that $(z_1^*,z_2^*)\rarrow{\eta\to\infty}(u_1^*/k_i,0)$ where $u_1^*=(r-g_0)/g_1$.
  \item If $r<g_0$, then we have that $(z_1^*,z_2^*) \rarrow{\eta\to\infty} (0,u_2^*/k_p)$ where $u_2^*=(g_0-r)/(g_n  r)$.
  \item If $r=g_0$, then
  \begin{equation}
    z_1^*=\sqrt{\dfrac{g_n  k_p\mu r}{\eta g_1k_i}}\textnormal{ and }z_2^*=\sqrt{\dfrac{\theta g_1k_i}{\eta g_n  k_p}}
  \end{equation}
  and they both tend to 0 when $\eta\to\infty$ while keeping the product $\eta z_1^*z_2^*$ constant and equal to $\mu$.
\end{enumerate}
\end{proposition}
\begin{proof}
  Assuming that $r>g_0$ and letting $\eta\to\infty$ in \eqref{eq:P} make its positive root go to $z_1^*=u_1^*/k_i$. The other statements are proven analogously.
\end{proof}}

\begin{figure}
  \centering
  \includegraphics[width=0.45\textwidth]{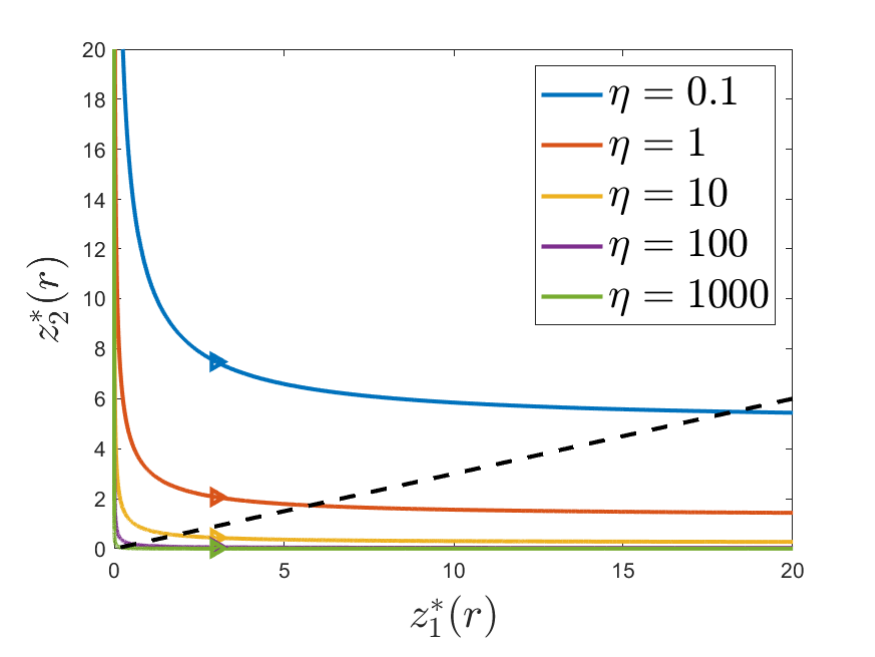}
  \caption{Illustration of the switching behavior described in Proposition \ref{prop:switching}. The curves describe the equilibrium values $(z_1^*(r),z_2^*(r))$ of the controller species parameterized in terms of the set-point $r$. Each curve corresponds to a specific value for $\eta$. The dashed line corresponds to the values of $(z_1^*(g_0),z_2^*(g_0))$. The values above that line corresponds to the case $r<g_0$ whereas the values below to the case $r>g_0$. One can observe that as $\eta$ increase we converge to a switching behavior where one of the equilibrium concentrations is zero, except in the particular case $r=g_0$ (not visible).}\label{fig:switching}
\end{figure}

\black{This result can be interpreted as follows: when the set-point $r$ is set beyond the basal level $g_0$ and $\eta$ is large enough, the naAIC part of the controller dominates its niAIC part whereas the opposite holds whenever the set-point is set below the basal level $g_0$. When the set-point is exactly equal to the basal level -- a rather pathological case -- both components act simultaneously in a rather fragile and unstable manner. This switching behavior is illustrated in Figure \ref{fig:switching} where one can observe the evolution of the equilibrium values of the states of the controller for different values for $\eta$ as the set-point $r$ increases.}\\

\black{Interestingly this result can be connected to an optimization problem at steady-state. To show this, consider the following adaptation of the system \eqref{eq:mainsystLu} to which we add two control inputs $u_1,u_2$ to mimic how the rein controller acts on the system in \eqref{eq:mainsystCL}:
\begin{equation}\label{eq:mainsystLu}
  \begin{array}{rcl}
    \dot{x}(t)&=&Ax(t)+b_0+e_1u_1(t)-e_nx_nu_2(t)\\
    %y(t)&=&e_n^Tx(t)\\
    x(0)&=&x_0.
  \end{array}
\end{equation}
Assuming that $A$ is Hurwitz stable, the steady-state relationship between the inputs and the output is given by
\begin{equation}
  y^*=g_0+g_1u_1-g_ny^*u_2,
\end{equation}
from which is immediate to see that for any given $y^*>0$, there is an infinite number of pairs $(u_1,u_2)\in\mathbb{R}_{\ge0}^2$ that satisfy this equation. Therefore, it is natural to seek to find a pair that minimize the overall control effort that leads to that output $y^*$. This problem can be formulated as the following linear programming problem
\begin{equation}
\begin{array}{rl}
  \min_{(u_1,u_2)\in\mathbb{R}^2}& u_1+u_2\\
  \mathrm{s.t.} & y^*=g_0+g_1u_1-g_ny^*u_2\\
                            & u_1,u_2\ge0.
\end{array}
\end{equation}
The unique solution to that problem is given by
\begin{equation}\label{eq:limit_solution}
  (u_1^*,u_2^*)=\left\{\begin{array}{ccl}
        \left(\dfrac{y^*-g_0}{g_1},0\right), &&\textnormal{if }y^*>g_0\\
        \left(0,\dfrac{g_0-y^*}{g_1y^*}\right), &&\textnormal{if }y^*<g_0\\
        \left(0,0\right), &&\textnormal{if }y^*=g_0
  \end{array}\right.
\end{equation}
and coincides with the conclusions of Proposition \ref{prop:switching}. Interestingly, the output rein controller can be shown to solve a relaxed version of the above problem by restricting the solution to lie on a surface of the form $u_1u_2=c$, where $c\ge0$. To show this, consider the feasibility problem
\begin{equation}
\begin{array}{rl}
  \mathrm{find}& (u_1,u_2)\in\mathbb{R}^2\\
  \mathrm{s.t.} & y^*=g_0+g_1u_1-g_ny^*u_2,\\
                            & u_1u_2=c,\\
                            &  u_1,u_2\ge0.
\end{array}
\end{equation}
which has the unique solution
\begin{equation}
  u_1^*=\dfrac{g_0-y^*+\sqrt{(y^*-g_0)^2+4g_1g_ny^*c}}{2g_1},\ u_2^*=\dfrac{c}{u_1^*}.
\end{equation}
When $c\to0$, which is analogous to $\eta\to\infty$, the solution of the above problem converges to \eqref{eq:limit_solution}. In this case, the parameter $c$ (or $\eta$, equivalently) allows one to tune the optimality of the choice for the input.}\\

The discussions above show that in the case of large sequestration parameter, the output rein controller either behaves like an naAIC or an niAIC depending on the value of the set-point. The class of naAIC's has been extensively studied since its introduction \cite{Briat:15e,Qian:18,Briat:19:Logistic,Olsman:19,Olsman:19b} and there is no reason to analyze it further here. On the other hand, the output niAIC is a structure that has not been thoroughly analyzed so far.

\subsection{Reduced problem - The niAIC with output inhibition}

\black{The objective of this section is to provide a thorough analysis of the output niAIC part of the considered rein controller with output inhibition
\begin{equation}\label{eq:AIC:niAIC_output}
  \phib\rarrow{\mu}\Z{1},\ \Yz\rarrow{\theta}\Yz+\Z{2},\ \Z{1}+\Z{2}\rarrow{k_p\eta}\phib,\  \Yz+\Z{2}\rarrow{k_p}\Z{2}
\end{equation}
where $\Z{1}$ is the reference species, $\Z{2}$ is both the sensing and actuating species, $\Yz$ is the measured, controlled and actuated species. The parameters $\mu,\theta,\eta,k_p$ are all positive rate parameters of the reactions, which are all assumed to be mass-action. The interconnection of this network with the (unimolecular) network to be controlled \eqref{eq:RN} yields the dynamical model
\begin{equation}\label{eq:mainsystCL2}
  \begin{array}{rcl}
    \dot{x}(t)&=&Ax(t)-e_n  x_n (t) k_pz_2(t)+b_0\\
    \dot{z}_1(t)&=&\mu-k_p\eta z_1(t)z_2(t)\\
    \dot{z}_2(t)&=&\theta x_n (t)-k_p\eta z_1(t)z_2(t)
  \end{array}
\end{equation}
where we have slightly changed the structure of the AIC by letting $\eta\leftarrow \eta k_p$. This modification does not change the nature of the results while simplifying their derivation; see e.g. \cite{Briat:19:Logistic}. The following result states under what condition the set-point is admissible:
\begin{proposition}\label{prop:eqpt}
  Assume that $A$ is Hurwitz stable. Then, the equilibrium point $(x^*,z_1^*,z_2^*)$ of the closed-loop system \eqref{eq:mainsystCL2} given by
  \begin{equation}
    (x^*,z_1^*,z_2^*)=\begin{pmatrix}
      -A^{-1}(-e_nru_*+b_0), & \dfrac{\mu}{\eta u_*}, & \dfrac{u_*}{k_p}
    \end{pmatrix},\ u_*=\dfrac{g_0-r}{g_n  r}
  \end{equation}
  is unique and positive if and only if $r<g_0$.
 \end{proposition}
\begin{proof}
The equilibrium points solve the following system of equations
\begin{equation}\label{eq:dskldskdls;dsld;}
  \begin{array}{rcl}
    0&=&Ax^*-e_n  x_nx^*k_pz_2x^*+b_0\\
    0&=&\mu-k_p\eta z_1x^*z_2x^*\\
    0&=&\theta x_n^*-k_p\eta z_1x^*z_2x^*.
  \end{array}
\end{equation}
From the two last rows, we obtain that $x_n^*=r$, $z_1^*=\mu/(k_p\eta z_2^*)$, and hence
  \begin{equation}
    Ax^*-e_nrk_pz_2^*+b_0=0.
  \end{equation}
  Therefore, we get
  \begin{equation}\label{eq:xstarCL2}
    x^*=-A^{-1}(-e_nrk_pz_2^*+b_0)
  \end{equation}
  which implies that $r=-g_nrk_pz_2^*+g_0$ and
  \begin{equation}
    z_2^*=\dfrac{g_0-r}{rk_pg_n}=\dfrac{u_*}{k_p}.
  \end{equation}
  Substituting this into the expression $z_1^*=\mu/(k_p\eta z_2^*)$ and \eqref{eq:xstarCL2}  yields $z_1^*=\mu/(\eta u_*)$ and the final expression for the unique equilibrium point.\\

  For $z_2^*$ to be positive, it is necessary and sufficient to have $g_0-r>0$, which then implies that $z_1^*$ is also positive. The first row of \eqref{eq:dskldskdls;dsld;} can be written as $(A-u^*e_ne_n^T)x^*+b_0=0$, which implies that $x^*=-(A-u^*e_ne_n^T)^{-1}b_0$ which implies that $x^*\ge0$ since $b_0\ge0$ and $A-u^*e_ne_n^T$ is Metzler and Hurwitz stable, as $A$ is Hurwitz stable by assumption. The proof is completed.
\end{proof}}

%\red{This section is devoted to the local structural stability analysis of the equilibrium point of the system \eqref{eq:mainsystCL2}. Some existence results are first derived and a full solution is then provided.}

\subsection{Existence results for the niAIC with output inhibition}

The following result states that the equilibrium point of the closed-loop network \eqref{eq:mainsystCL2} is locally exponentially stable for any sufficiently small $k_p>0$.
\black{\begin{proposition}\label{prop:small_kp}
  Assume that $\bar A=A-e_n  e_n ^Tu_*$ is Metzler and Hurwitz stable, and let $\eta,\mu,\theta>0$ be given and such that $r<g_0$. Then, the equilibrium point of the closed-loop network \eqref{eq:mainsystCL2} is locally exponentially stable for any sufficiently small $k_p>0$.
\end{proposition}
\begin{proof}
The linearized dynamics of the system about that equilibrium point is given by
\begin{equation}\label{eq:A}
  \begin{bmatrix}
    \dot{\tilde{x}}(t)\\
    \dot{\tilde{z}}_1(t)\\
    \dot{\tilde{z}}_2(t)
  \end{bmatrix}=\begin{bmatrix}
    \bar{A} & 0 & -e_n  k_pr\\
    0 & -\eta u_* & -\mu k_p/u_*\\
    \theta e_n ^T & -\eta u_* & -\mu k_p/u_*
  \end{bmatrix}\begin{bmatrix}
    \tilde{x}(t)\\
    \tilde{z}_1(t)\\
    \tilde{z}_2(t)
  \end{bmatrix}
\end{equation}
where $\bar A:=A-e_n  e_n ^Tu_*$. The matrix of the linearized dynamics can be decomposed as
\begin{equation}
  \begin{bmatrix}
    \bar{A} & 0 & 0\\
    0 & -\eta u_* & 0\\
   \theta  e_n ^T & -\eta u_* & 0
  \end{bmatrix}+k_p\begin{bmatrix}
    0 & 0 & -e_n  r\\
    0 & 0 & -\mu /u_*\\
    0 & 0 & -\mu /u_*
  \end{bmatrix}.
\end{equation}
The matrix to the left is marginally stable with one eigenvalue at 0. To show the existence of a $k_p$ that makes the matrix in \eqref{eq:A} Hurwitz stable, we use a perturbation argument \cite{Seyranian:03} and check at what happens when we slightly increases $k_p$ from 0 to (small) positive values. The normalized left- and right-eigenvectors associated with the zero eigenvalue are given by
\begin{equation}
  u = \begin{bmatrix}
    -\theta e_n ^T\bar{A}^{-1} & -1 & 1
  \end{bmatrix}^T\ \textnormal{and}\ v = \begin{bmatrix}
    0 & 0 & 1
  \end{bmatrix}^T.
\end{equation}
Therefore, the zero eigenvalue bifurcates according to the expression \cite{Seyranian:03}
\begin{equation}
  \lambda_0(k_p)=0+k_pu^T\begin{bmatrix}
    0 & 0 & -e_n  r\\
    0 & 0 & -\mu /u_*\\
    0 & 0 & -\mu /u_*
  \end{bmatrix}v+o(k_p)= k_p\mu e_n^T\bar{A}^{-1}e_n+o(k_p).
\end{equation}
Since $\bar A$ is Metzler and Hurwitz stable, then $\bar{A}^{-1}\le0$ with negative diagonal entries, and we have that $e_n^T\bar{A}^{-1}e_n<0$. Therefore, the zero eigenvalue moves inside the open left half-plane when positively perturbing $k_p$ from the 0 value, which implies that for any sufficiently small $k_p>0$ the matrix in \eqref{eq:A} is Hurwitz stable.
\end{proof}}

Similarly, the following result states that the equilibrium point of the closed-loop network \eqref{eq:mainsystCL2} is locally exponentially stable for any sufficiently small $\eta>0$.
\black{\begin{proposition}\label{prop:small_eta}
 Assume that $\bar A=A-e_n  e_n ^Tu_*$ is Metzler and Hurwitz stable, and let $\eta,\mu,\theta>0$ be given and such that $r<g_0$. Then, the equilibrium point of the closed-loop network \eqref{eq:mainsystCL2} is locally exponentially stable for any sufficiently small $\eta>0$.
\end{proposition}
\begin{proof}
Similarly, one can rewrite  the matrix in \eqref{eq:A} as
\begin{equation}
  \begin{bmatrix}
    \bar{A} & 0 & -e_n  k_pr\\
    0 & 0 & -\mu k_p/u_*\\
    \theta e_n ^T & 0 & -\mu k_p/u_*
  \end{bmatrix}+\eta\begin{bmatrix}
    0 & 0 & 0\\
    0 & - u_* & 0\\
    0 & -u_* & 0
  \end{bmatrix}.
\end{equation}
Again the matrix on the left is marginally stable with one eigenvalue at zero. The normalized left- and right-eigenvectors associated with the zero eigenvalue are given by
\begin{equation}
  u = \begin{bmatrix}
    * & 1 & *
  \end{bmatrix}^T\ \textnormal{and}\ v = \begin{bmatrix}
    0 & 1 & 0
  \end{bmatrix}^T
\end{equation}
where $*$ means that the entries are unimportant here. The zero eigenvalue  bifurcates according to the expression \cite{Seyranian:03}
\begin{equation}
  \lambda_0(\eta)=0+\eta u^T\begin{bmatrix}
  0 & 0 & 0\\
    0 & - u_* & 0\\
    0 & -u_* & 0
  \end{bmatrix}v+o(k_p)= -u_*\eta +o(\eta).
\end{equation}
Since $r<g0$, then we have that $u^*>0$ and, as a result, the matrix in \eqref{eq:A} is Hurwitz stable for any sufficiently small $\eta>0$.
\end{proof}}

Conversely, the following result states that the equilibrium point of the closed-loop network \eqref{eq:mainsystCL2} is locally exponentially stable for any sufficiently large $\eta>0$ provided that an extra condition is met.
\black{\begin{proposition}\label{prop:large_eta}
 Assume that $\bar A=A-e_n  e_n ^Tu_*$ is Metzler and Hurwitz stable, and let $\eta,\mu,\theta>0$ be given and such that $r<g_0$. Assume further that the matrix
 \begin{equation}
   \begin{bmatrix}
    \bar{A} & -e_n  k_p\mu\\
    e_n ^T & 0
  \end{bmatrix}
 \end{equation}
 is Hurwitz stable.  Then, the equilibrium point of the closed-loop network \eqref{eq:mainsystCL2} is locally exponentially stable for any sufficiently large $\eta>0$.
\end{proposition}
\begin{proof}
The matrix in \eqref{eq:A} can be rewritten as
\begin{equation}
  \dfrac{1}{\eps}\left(\begin{bmatrix}
    0 & 0 & 0\\
    0 & - u_* & 0\\
    0 & -u_* & 0
  \end{bmatrix}+\eps\begin{bmatrix}
    \bar{A} & 0 & -e_n  k_pr\\
    0 & 0 & -\mu k_p/u_*\\
    \theta e_n ^T & 0 & -\mu k_p/u_*
  \end{bmatrix}\right)
\end{equation}
where $\eps:=1/\eta$. Therefore, one can study the stability of the above matrix in brackets for small positive values which are close to $\eps=0$. The matrix on the left is marginally stable with a semisimple eigenvalue at 0 of multiplicity $n+1$. The normalized left- and right-eigenvectors associated with the zero eigenvalue are given by
\begin{equation}
  u = \begin{bmatrix}
    I & 0 & 0\\
    0 & -1 & 1
  \end{bmatrix}^T\ \textnormal{and}\ v = \begin{bmatrix}
    I & 0 & 0\\
    0 & 0 & 1
  \end{bmatrix}^T.
\end{equation}
Therefore, the zero eigenvalues bifurcate according to the expression \cite{Seyranian:03}
\begin{equation}
\begin{array}{rcl}
  \lambda_0(\eta)&=&0+\eps \lambda_i\left(u^T\begin{bmatrix}
    \bar{A} & 0 & -e_n  k_pr\\
    0 & 0 & -\mu k_p/u_*\\
   \theta  e_n ^T & 0 & -\mu k_p/u_*
  \end{bmatrix}v\right)+o(\eps)\\
  &=&\eps\lambda_i\left(\begin{bmatrix}
    \bar{A} & -e_n  k_pr\\
    \theta e_n ^T & 0
  \end{bmatrix}\right)+o(\eps)\\
&=&\eps\lambda_i\left(\begin{bmatrix}
    \bar{A} & -e_n  k_p\mu\\
    e_n ^T & 0
  \end{bmatrix}\right)+o(\eps)
\end{array}
\end{equation}
where $\lambda_i(\cdot)$ denotes the $i$-th eigenvalue of the matrix and where the last matrix has been obtained by pre- and post-multiplying the previous matrix by $\diag(\theta I,1)$ and $\diag(\theta^{-1} I,1)$, respectively. Therefore, under the condition of the result. the real part of the eigenvalues of the matrix are negative and all the zero eigenvalues bifurcate inside the open left half-plane when increasing $\eps$ from 0 to small positive values. As this corresponds to moving $\eta$ from infinity to large positive values, the result is proven.
\end{proof}}

\subsection{Structural stability analysis for the niAIC with output inhibition - Stable case}\label{sec:uni:niAIC:stable}

This result is instrumental in proving the main result of the section:
\black{\begin{theorem}\label{th:Hn}
  Assume that $M$ is Metzler and Hurwitz stable. Then, the transfer function $H(s):=e_n^T(sI-M)^{-1}e_n$ is strictly positive real.
\end{theorem}
\begin{proof}
 Since $M$ is Hurwitz stable, then the poles of $H(s)$ have all negative real part, and the condition \eqref{st:SPR:1} of Definition \ref{def:PRSPR} is met.\\

Using now Lemma \ref{lemma:zeros}, the zeros are given by the eigenvalues of $M_{11}$ which have all negative real part, and we have $H(s)=N(s)/D(s)$ with $D(s)=\det(sI-M)$ and $N(s)=\det(sI-M_{11})$ together with $\deg(D)=n$ and $\deg(N)=s$. As a result, we have that $H(0)>0$ and $H(\infty)=0$. So, we need to check the second property of the statement \eqref{st:SPR:3} of Definition \ref{def:PRSPR}. To this aim, we use Lemma \ref{lemma:infinity} and note that all the assumptions of this result are verified with $K=1$. Therefore, we get that
  \begin{equation}
    \lim_{\omega\to\infty}\omega^2\Re[H(j\omega)]=\trace(M_{11})-\trace(M)=-e_n^TMe_n.
  \end{equation}
  Since $M$ is Metzler and Hurwitz stable, then $-e_n^TMe_n>0$ and the conclusion follows; i.e. the condition \eqref{st:SPR:3} of Definition \ref{def:PRSPR} is verified.\\

  We now check the condition \eqref{st:SPR:2} of Definition \ref{def:PRSPR}. To this aim, we consider Proposition \ref{prop:LMI}. Adapted to our problem, we need to find  a matrix  $P\in\mathbb{S}_{\succ0}^d$ such that $M^TP+PM+2\eps e_ne_n^T\prec0$ and $Pe_n -e_n=0$. The equality constraint imposes that $P=\diag(P_1,1)$ where $P_1\in\mathbb{S}^{n-1}_{\succ0}$. Clearly, we have that if $M^TP+PM\prec0$ holds, then $M^TP+PM+2\eps e_ne_n^T\prec0$ holds for any sufficiently small $\eps>0$. Using now the fact that $M$ is Hurwitz stable and Metzler, then this is equivalent \cite{Farina:00} to say that there exists a diagonal matrix $D$ with positive diagonal entries such that $M^TD+DM\prec0$. Dividing this inequality by $e_n^TDe_n>0$ and letting $P_s:=D/e_n^TDe_n$, we get that $M^TP_s+P_sM\prec0$. Observing that this $P_s$ satisfies the conditions of the result shows that the condition \eqref{st:SPR:2} of Definition \ref{def:PRSPR} is met. Therefore, the transfer function $H(s)$ is SPR.
\end{proof}}

We can now state the main result of the section:
\black{\begin{theorem}\label{th:structstab1}
  Let $\mu,\theta>0$ be given and such that $r<g_0$ and assume that $A$ is Metzler and Hurwitz stable. Then, the unique equilibrium point of the system \eqref{eq:mainsystCL2} is locally exponentially stable for all $\eta,k_p>0$.
\end{theorem}
\begin{proof}
%The linearized dynamics about the unique equilibrium point is given by
%  \begin{equation}
%  \begin{bmatrix}
%    \dot{x}\\
%    \dot{z}_1\\
%    \dot{z}_2
%  \end{bmatrix}=\begin{bmatrix}
%    \bar A & 0 & -e_nk_pr\\
%    0 & -\eta c & -\mu k_p/u_*\\
%    e_n^T & -\eta c & -\mu k_p/u_*
%  \end{bmatrix}  \begin{bmatrix}
%    x\\
%    z_1\\
%    z_2
%  \end{bmatrix}.
%\end{equation}
To analyze the stability of the system in \eqref{eq:A} for all $k_p>0$, we can reformulate the problem as the stability analysis of a negative feedback interconnection of two systems. To see this, we can consider the following equivalent reformulation of \eqref{eq:A}
\begin{equation}
  \begin{array}{rcl}
    \begin{bmatrix}
    \dot{x}(t)\\
    \dot{z}_1(t)
    \end{bmatrix}&=&\begin{bmatrix}
      \bar{A} & 0\\
       0&-\eta u_*
    \end{bmatrix}\begin{bmatrix}
    x(t)\\
    z_1(t)
    \end{bmatrix}+\begin{bmatrix}
      e_nr\\
      \dfrac{\mu}{u_*}
    \end{bmatrix}w(t)\\
  v(t)&=&\begin{bmatrix}
     \theta e_n^T & -\eta u_*
    \end{bmatrix}\begin{bmatrix}
    x(t)\\
    z_1(t)
    \end{bmatrix}+\dfrac{\mu}{u_*}w(t)\\
    w(t)&=&k_p\int_0^tv(s)\ds.
  \end{array}
\end{equation}
The transfer from $v$ to $w$ is simply given by $k_p/s$  whereas the transfer function from $w$ to $v$ is given by
\begin{equation}
\begin{array}{rcl}
  G_{\eta}(s)&:=&\begin{bmatrix}
   \theta e_n^T & -\eta u_*
  \end{bmatrix}\begin{bmatrix}
    sI-\bar A & 0\\
    0 &  s+\eta u_*
  \end{bmatrix}^{-1}\begin{bmatrix}
     e_nr\\
    \mu/u_*
  \end{bmatrix}+\mu/u_*\\
  &=& H_n(s)\mu+\dfrac{-\eta\mu}{s+\eta u_*}+\dfrac{\mu}{u_*}\\
  &=& H_n(s)\mu+G(s)
\end{array}
\end{equation}
where
\begin{equation}
  \begin{array}{rcl}
    G(s)&:=&\dfrac{\mu s}{u_*(s+\eta u_*)}\\
    H_n(s)&:=&e_n^T(sI-\bar A)^{-1}e_n,
  \end{array}
\end{equation}
which results in the interconnection is depicted in Figure~\ref{fig:interconnection}. Since, $k_p/s$ is positive real for all $k_p>0$, if we can prove that $G_{\eta}(s)$ is (weakly) strictly positive real, then from Theorem \ref{th:interconnection}, the feedback interconnection will be stable for all $k_p>0$. From Theorem \ref{th:Hn}, the transfer function $H_n(s)$ is SPR. One can also observe that $G(s)$ is stable for all $\eta>0$, positive real since $\Re[G(j\omega)]=\omega^2/(\omega^2+\eta^2u_*^2)\ge0$ for all $\omega\ge0$, and such that $G(\infty)=\mu/u_*>0$. As a result, the sum of those transfer functions is SPR for all $\eta>0$ and, therefore, the feedback interconnection is stable for all $k_p,\eta>0$. This proves the result.
\end{proof}}

This result is quite powerful in the sense that the only assumption we have on the system is its stability and the only one on the controller is that the set-point is admissible. Interestingly, this result readily generalizes to uncertain systems using the same ideas as in \cite{Briat:20:Structural}.\\

\begin{figure}
  \centering
  \includegraphics[width=.5\textwidth]{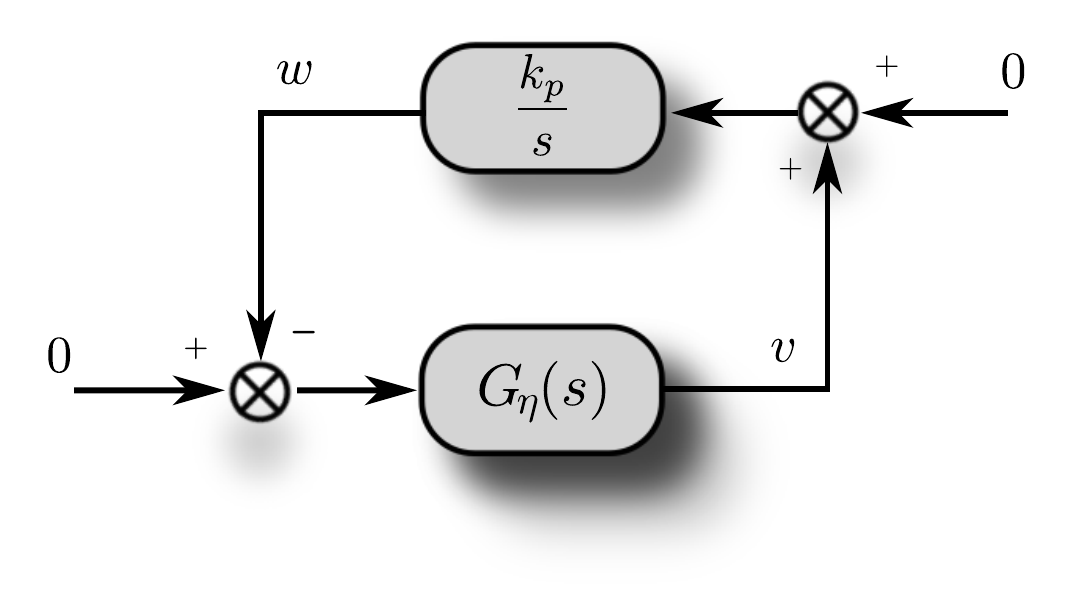}
  \caption{Equivalent negative interconnection}\label{fig:interconnection}
\end{figure}

\black{We illustrate Theorem \ref{th:structstab1} through the following simple example:
\begin{example}-
  Let us consider the following mass-action gene expression network with protein maturation and positive feedback:
  \begin{equation}\label{eq:RN:maturation}
  \begin{matrix}
    \phib&\rarrow{k_0}&\X{1} & \X{1}&\rarrow{k_{21}}&\X{1}+\X{2} & \X{i}&\rarrow{\gamma_i}\phib,i=1,2,3\\
    \X{2}&\rarrow{k_{32}}&\X{3}& \X{3}&\rarrow{k_{13}}&\X{1}+\X{3}
  \end{matrix}
  \end{equation}
  where $\X{1},\X{2}$, and $\X{3}$ are the mRNA, protein, and matured protein species, respectively, and where the parameters of the reactions of positive real numbers. We assume here that we would like to control the maturated species, that is, $\Yz=\X{3}$. The model of the system is given by
\begin{equation}
\begin{bmatrix}
  \dot{x}_1\\
    \dot{x}_2\\
    \dot{x}_3
\end{bmatrix}=\begin{bmatrix}
    -\gamma_1 & 0 & k_{13}\\
    k_{21} & -\gamma_2 & 0\\
    0 & k_{32} & -\gamma_3
  \end{bmatrix}\begin{bmatrix}
    x_1\\
    x_2\\
    x_3
  \end{bmatrix}+\begin{bmatrix}
    k_0\\
    0\\
    0
  \end{bmatrix}
\end{equation}
where $x_1,x_2,x_3$ are the mRNA, protein, and matured protein concentrations, respectively. The matrix describing the dynamics of the network is Metzler, by construction, and is Hurwitz stable provided that $\gamma_1\gamma_2\gamma_3-k_{13}k_{32}k_{21}>0$ (Routh-Hurwitz criterion). This yields the following result which is immediate application of Theorem \ref{th:structstab1}

\begin{theorem}
  Assume that the matrix describing the dynamics of the network is Hurwitz stable, then the closed-loop network \eqref{eq:RN:maturation}, \eqref{eq:AIC:niAIC_output} is locally exponentially stable for all $0<\mu/\theta<g_0$ and all $k_p>0,\eta>0$ where
  \begin{equation}
  g_0=\dfrac{k_0k_{21}k_{32}}{\gamma_1\gamma_2\gamma_3-k_{21}k_{32}k_{13}}>0.
\end{equation}
\end{theorem}
\end{example}}

%\begin{figure}[H]
%  \centering
%  \includegraphics[width=0.50\textwidth]{./Matlab/pAIC_Unimolecular_Gene_expression_maturation_stable/pAIC_stable_heatmap.pdf}
%  \caption{Spectral abscissa of the system associated with the \eqref{eq:RN:maturation}, \eqref{eq:AIC:niAIC_output} with the parameters $\gamma_1=1$, $\gamma_2=\gamma_3=2$, $k_{21}=1$, $k_{32}=2$, $b_0=10$, $k_{13}=1$, $\eta=100$, and $\theta=1$ and for various values for $\mu$ and $k_p$. In this case, we have that $g_0=10$. Simple calculations show that the spectral abscissa of $A$, is $\alpha(A)=-0.3044$ while the spectral abscissa of the closed-loop system may reach smaller values, which indicates that this controller is able to improve the convergence properties of the system near the equilibrium point. }
%\end{figure}

\begin{figure}[H]
  \centering
\begin{minipage}[h]{0.45\linewidth}
    \begin{center}
  \includegraphics[width=\textwidth]{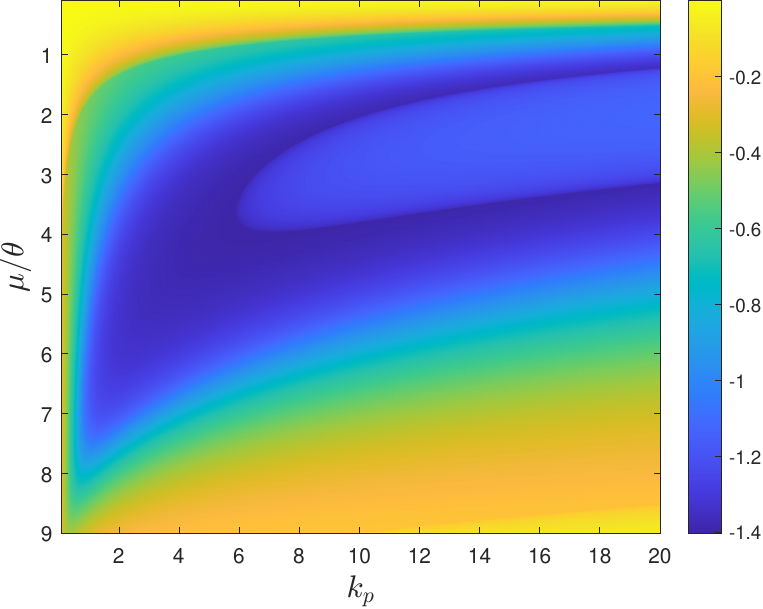}
  \end{center}
\end{minipage}
\begin{minipage}[h]{0.45\linewidth}
    \begin{center}
  \includegraphics[width=\textwidth]{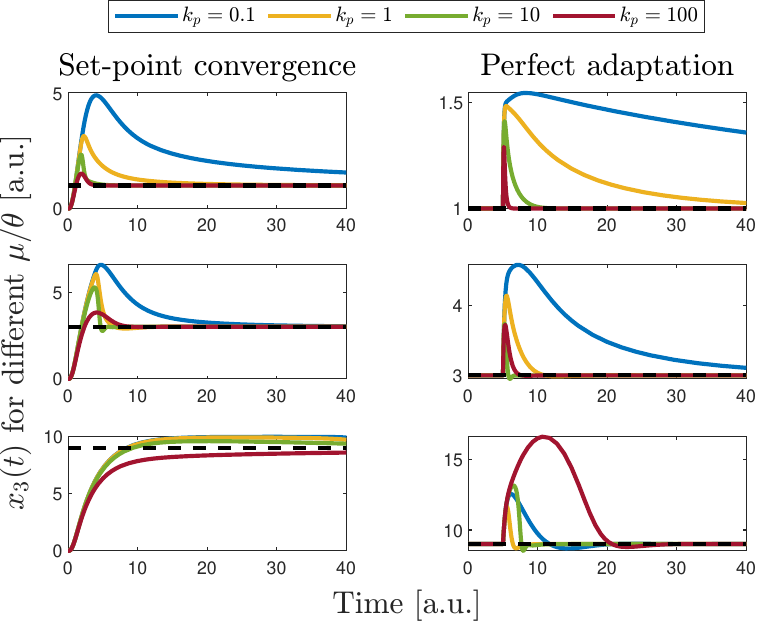}
  \end{center}
\end{minipage}
\caption{\black{\textbf{Left.} Spectral abscissa of the system associated with the \eqref{eq:RN:maturation}, \eqref{eq:AIC:niAIC_output} with the parameters $\gamma_1=1$, $\gamma_2=\gamma_3=2$, $k_{21}=1$, $k_{32}=2$, $b_0=10$, $k_{13}=1$, $\eta=100/k_p$, and $\theta=1$ and for various values for $\mu$ and $k_p$. In this case, we have that $g_0=10$. Simple calculations show that the spectral abscissa of $A$, is $\alpha(A)=-0.3044$ while the spectral abscissa of the closed-loop system may reach smaller values, which indicates that this controller is able to improve the convergence properties of the system near the equilibrium point. \textbf{Right.} Time domain evolution of the concentrations of $\X{3}$ various values for the set-point $\mu/\theta$ and controller gains $k_p$. The left column depicts simulation results under zero initial conditions whereas the right column depicts the response of the closed-loop network when the parameter $k_{32}$ changes from 2 to 3 at $t=5$.}}
\end{figure}

\subsection{Structural stability analysis for the niAIC with output inhibition - Unstable case}\label{sec:uni:niAIC:unstable}

We relax in this section the condition that $A$ is Hurwitz stable and derive a sequence of results to address the more general case where $A$ is allowed to be unstable.

\subsubsection{Preliminaries}

The results obtained in the previous section relies on the assumption that $A$ is Hurwitz stable. However, this assumption appears to be too strict since only the matrix $\bar A=A-e_ne_n^Tu_*$ would need to be Hurwitz stable. As it will be shown later, the class of matrices $A$ that can actually be considered is that of \emph{output-unstable matrices}, which is defined below:
\black{\begin{define}
  We say that a matrix $M\in\mathbb{R}^{n\times n}$ partitioned as
  \begin{equation}
    M=:\begin{bmatrix}
      M_{11} & M_{12}\\
      M_{21} & M_{22}
    \end{bmatrix}
  \end{equation}
  where $M_{11}\in\mathbb{R}^{(n-1)\times (n-1)}$, $M_{12}\in\mathbb{R}^{n-1}$, $M_{21}^T\in\mathbb{R}^{n-1}$, and $M_{22}\in\mathbb{R}$ is output unstable if
  \begin{enumerate}[(a)]
    \item $M_{11}$ is Hurwitz stable, and
    \item $M_{22}-M_{21}M_{11}^{-1}M_{12}>0$.
  \end{enumerate}
  %\begin{itemize}
%    \item output unstable if
%  \begin{enumerate}[(a)]
%    \item $M_{11}$ is Hurwitz stable, and
%    \item $M_{22}-M_{21}M_{11}^{-1}M_{12}>0$.
%  \end{enumerate}
%    \item strictly output unstable if
%  \begin{enumerate}[(a)]
%    \item $M_{11}$ is Hurwitz stable, and
%    \item $M_{22}>0$.
%  \end{enumerate}
%  \end{itemize}
\end{define}}

\black{The following result characterizes the sign pattern of the inverse matrix $M^{-1}$ as well as the sign of the gains $g_n$ and $g_0$ defined in Proposition \ref{prop:eqpt}:
\begin{lemma}\label{lem:inversion}
  Assume that the matrix $M$ is Metzler, output unstable, and nonsingular. Then, the following statements hold:
  \begin{enumerate}[(a)]
    \item The inverse matrix exhibits the sign-pattern $S^TM^{-1}e_n\ge0$, $e_n^TM^{-1}S\ge0$, and $e_n^TM^{-1}e_n>0$ where $S:=\begin{bmatrix}
      I_{n-1} \\ 0\end{bmatrix}$; and
    \item The gains $g_0$ and are $g_n$ are such that $g_0:=-e_n^TM^{-1}b_0\le0$ and $g_n:=-e_n^TM^{-1}e_n<0$.
  \end{enumerate}
\end{lemma}
\begin{proof}
  From the block matrix inversion formula, we have that
  \begin{equation}
    M^{-1}=\begin{bmatrix}
      \star & -M_{11}^{-1}M_{12}(M_{22}-M_{21}M_{11}^{-1}M_{12})^{-1}\\
      -(M_{22}-M_{21}M_{11}^{-1}M_{12})^{-1}M_{21}M_{11}^{-1} & (M_{22}-M_{21}M_{11}^{-1}M_{12})^{-1}
    \end{bmatrix}.
  \end{equation}
  Since $M$ is Metzler, output unstable, and nonsingular, then we have both $M_{11}^{-1}\le0$ and $M_{22}-M_{21}M_{11}^{-1}M_{12}>0$. This implies that $S^TM^{-1}e_n\ge0$, $e_n^TM^{-1}S\ge0$, and $e_n^TM^{-1}e_n>0$. Since $e_n^TM^{-1}\ge0$, then we have that  $e_n^TM^{-1}b_0\ge0$, and the result follows.
\end{proof}}

\black{The following result provides a condition for which  the matrix $\bar{A}=A-e_ne_n^Tu_*$ to be Hurwitz stable.
\begin{lemma}\label{lem:HSuns}
  Assume that the matrix $A$ is Metzler and output unstable and that $r>0$. Then, $\bar A=A-e_ne_n^Tu_*$ is Hurwitz stable if and only if $g_0<0$.
\end{lemma}
\begin{proof}
%A necessary condition for the Hurwitz stability of the  $\bar A$ matrix is that $e_n^T\bar{A}e_n<0$, which is equivalent to saying that $A_{22}-u_*<0$. Using the explicit expression for $u_*$, this condition rewrites as $r(1+g_nA_{22})-g_0>0$. Noting that we have $g_0=-e_n^TA^{-1}b_0\le$ and
%\begin{equation}
%  g_n=-e_n^TA^{-1}e_n=-(A_{22}-A_{21}A_{11}^{-1}A_{12})^{-1}>A_{22}^{-1}.
%\end{equation}
%Hence, $1+g_nA_{22}>0$, and we have that $r(1+g_nA_{22})-g_0>0$ for all $r>0$ and all $g_0\le0$. However, when $A_{21}A_{11}^{-1}A_{12})^{-1}=0$, then  $1+g_nA_{22}=0$ and we need $g_0<0$ for the condition to hold for all $r>0$.
%
The control law is stabilizing if and only if there exist a vector $v_1\in\mathbb{R}_{>0}^{n-1}$ and scalar $v_2>0$ such that
\begin{equation}
  \begin{bmatrix}
    A_{11} & A_{12}\\
    A_{21} & A_{22}-u_*
  \end{bmatrix}\begin{bmatrix}
    v_1\\
    v_2
  \end{bmatrix}=-\begin{bmatrix}
    w_1\\
    w_2
  \end{bmatrix},
\end{equation}
for some $w_1\in\mathbb{R}_{>0}^{n-1}$ and $w_2>0$. Solving for the first row yields
\begin{equation}
  v_1=-A_{11}^{-1}(w_1+A_{12}v_2)
\end{equation}
where $v_1>0$ since $A_{11}$ is Metzler and Hurwitz stable (so its inverse is a nonpositive matrix). Substituting the value for $v_1$ in the second row yields
\begin{equation}
  (A_{22}-A_{21}A_{11}^{-1}A_{12}-u_*)v_2=A_{21}A_{11}^{-1}w_1-w_2\le-w_2<0
\end{equation}
where we have used the fact that $A_{11}^{-1}\le0$. The above inequality can only hold if $A_{22}-A_{21}A_{11}^{-1}A_{12}-u_*<0$.\\

Noting that $A_{22}-A_{21}A_{11}^{-1}A_{12}=-1/g_n$ the stability condition can be rewritten as $-1/g_n-u_*<0$ or, equivalently,
\begin{equation}
  \dfrac{-1}{g_n}-\dfrac{g_0-r}{g_nr}=\dfrac{-g_0}{g_nr}<0
\end{equation}
where the inequality holds provided that $g_0<0$ since $g_n<0$. This proves the result.
\end{proof}}

The following result states under what conditions the network has a unique, nonnegative equilibrium point:
\begin{theorem}\label{lem:positiveunstable}
  Assume that the matrix $A$ is Metzler, output unstable, and nonsingular, and that $g_0<0$, then $u^*>0$, any set-point $r>0$ is admissible for the system \eqref{eq:mainsystCL2}, and the associated equilibrium point is nonnegative and unique.
\end{theorem}
\begin{proof}
  The equilibrium point satisfies the expressions
  \begin{equation}
    -(A-e_ne_n^Tk_pz_2^*)x^*+b_0=0, x_n^*=\mu/\theta,\ z_2^*=u_*/k_p,\ \textnormal{and\ }z_1^*=\dfrac{\mu}{\eta u_*}
  \end{equation}
  Under the assumptions of the result, the matrix $A-e_ne_n^Tu_*$ is Metzler and Hurwitz stable, which shows the existence and uniqueness of a nonnegative equilibrium point for all $r>0$. The positivity of $u^*$ follows from the expression $u^*=(g_0-r)/(g_nr)$.
\end{proof}

\subsubsection{Persistent external excitation}\label{sec:linear:unstable:persistent}

\black{Based on the results of the previous section, we can state the main stability result in the case where $g_0<0$:
\begin{theorem}\label{th:main:unstable}
  Assume that $A$ is Metzler, output unstable, and nonsingular and that $g_0\ne 0$. Then, the unique equilibrium point of the system \eqref{eq:mainsystCL2} is locally exponentially stable for all $\eta,k_p,\mu,\theta>0$.
\end{theorem}
\begin{proof}
   Under the condition that $r>0$ then the equilibrium point is positive by virtue of Lemma \ref{lem:positiveunstable}. Under the extra condition that $g_0\ne0$, we have that $g_0<0$ (see Lemma \ref{lem:inversion}), then the matrix $\bar A=A-e_ne_n^Tu_*$ is Metzler and Hurwitz stable from Lemma \ref{lem:HSuns}. The rest of the proof is then identical to that of Theorem \ref{th:structstab1} as we are now in exactly the same setting.
\end{proof}}

\black{The underlying assumption that $A_{11}$ be Hurwitz stable is not restrictive in this case since this is a necessary condition for the matrix $A-e_ne_n^Tu_*$ to be Hurwitz stable. Additionally, the strict positively realness property of the transfer function $H_n(s)$ also requires the zeros to be stable, which is equivalent to saying that $A_{11}$ be Hurwitz stable. Interestingly, the set of admissible set-points is much larger in the unstable case than in the stable case. This phenomenon can be easily explained by the fact that in the unstable case, the controller can always let the output increase until it exceeds the set-point before starting repressing it to stabilize it around the desired set-point. For this to happen, the system needs to be persistently excited, which is the meaning of the condition that $g_0\ne0$. The case where $g_0=0$ will be treated in the next section.}\\

\black{We consider now the following illustrative example:
\begin{example}
  Consider the following gene expression network with maturation, positive feedback ($k_{13}$) and autocatalytic reaction ($\nu$)
\begin{equation}\label{eq:RN:maturation2}
\begin{bmatrix}
  \dot{x}_1\\
    \dot{x}_2\\
    \dot{x}_3
\end{bmatrix}=  \begin{bmatrix}
    -\gamma_1 & 0 & k_{13}\\
    k_{21} & -\gamma_2 & 0\\
    0 & k_{32} & \nu -\gamma_3
  \end{bmatrix}\begin{bmatrix}
    x_1\\
    x_2\\
    x_3
  \end{bmatrix}+\begin{bmatrix}
    k_0\\
    0\\
    0
  \end{bmatrix}
\end{equation}
where $x_1,x_2,x_3$ are the mRNA, protein, and matured protein concentrations, respectively. All the model parameters are positive except, possibly, $k_{13}$ and $\nu$. This matrix is output unstable if and only if
\begin{equation}
  \gamma_1\gamma_2(\nu -\gamma_3)+k_{21}k_{32}k_{13}>0,
\end{equation}
which yields the following result
\begin{theorem}
  Assume that $\gamma_1,\gamma_2>0$ that
\begin{equation}
  \gamma_1\gamma_2(\nu -\gamma_3)+k_{21}k_{32}k_{13}>0
\end{equation}
and
\begin{equation}
  g_0=\dfrac{-k_0k_{21}k_{32}}{\gamma_1\gamma_2(\nu-\gamma_3)+k_{21}k_{32}k_{13}}<0.
\end{equation}
 then the closed-loop network \eqref{eq:RN:maturation2}, \eqref{eq:AIC:niAIC} is locally exponentially stable for all $\mu,\theta,k_p>0,\eta>0$.
  % where
%  \begin{equation}
%  g_0=\dfrac{b_0k_{21}k_{32}}{\gamma_1\gamma_2\gamma_3-k_{21}k_{32}k_{13}}>0.
%\end{equation}
%provided that
%  \begin{equation}
%  \gamma_1\gamma_2\gamma_3-k_2k_3\alpha>0.
%  \end{equation}
\end{theorem}
\begin{proof}
  In particular, the autocatalytic parameter is supposed to exceed the degradation rate of the matured protein; i.e. $\nu -\gamma_3>0$. We have that $A_{11}$ is Metzler and Hurwitz stable and that $\nu -\gamma_3>0$. Therefore, the matrix is Metzler and output unstable. Computing $g_0$ yields
\begin{equation}
  g_0=-\dfrac{bk_2k_3}{k_2k_3\alpha+\gamma_1\gamma_2(\nu -\gamma_3)}<0.
\end{equation}
Therefore, from Theorem \ref{th:main:unstable}, the equilibrium point of the closed-loop system is structurally stable provided that $r>0$.
\end{proof}
This result is illustrated by simulation in Figure \ref{fig:persistent:1} and Figure \ref{fig:persistent:2} where we can observe the predictions of the above theorem.
\end{example}
}

\begin{figure}[H]
  \centering
  \begin{minipage}[h]{0.45\linewidth}
    \begin{center}
    \includegraphics[width=\textwidth]{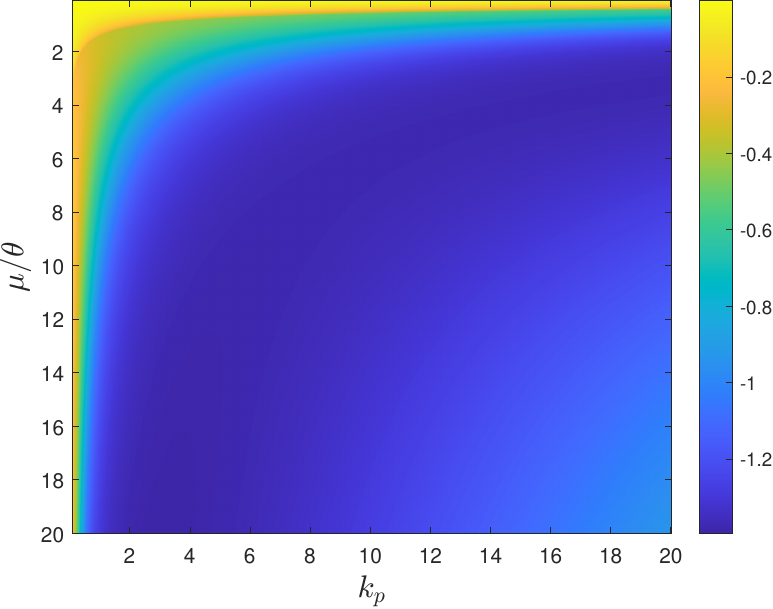}
  \end{center}
\end{minipage}
\begin{minipage}[h]{0.45\linewidth}
    \begin{center}
  \includegraphics[width=\textwidth]{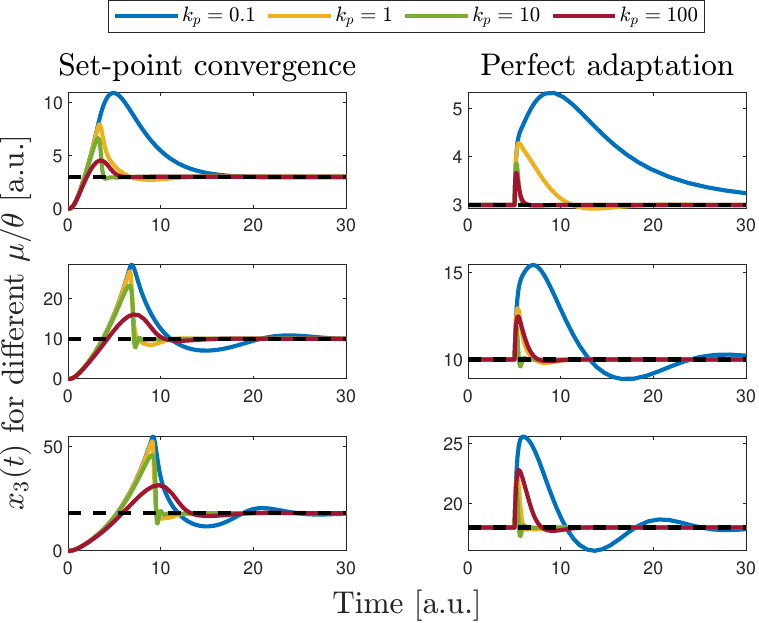}
  \end{center}
\end{minipage}
  \caption{\textbf{Left panel.} Spectral abscissa of the system associated with the \eqref{eq:RN:maturation2}, \eqref{eq:AIC:niAIC} with the parameters $\gamma_1=1$, $\gamma_2=\gamma_3=2$, $k_{21}=1$, $b_0=10$,  $k_{13}=3$, $\nu =0$, $\theta=1$, $\eta=100/k_p$ and for various values for $\mu$ and $k_p$. We have that which corresponds to $g_0=-10$, $g_n=-1$ and a spectral abscissa of $A$ given by $\alpha(A)=0.2188$. \textbf{Right panel.} State response from a zero initial condition (first column) and as a response to a change of $k_{21}$ from 2 to 3.}\label{fig:persistent:1}
\end{figure}

\begin{figure}[H]
  \centering
    \begin{minipage}[h]{0.45\linewidth}
    \begin{center}
   \includegraphics[width=\textwidth]{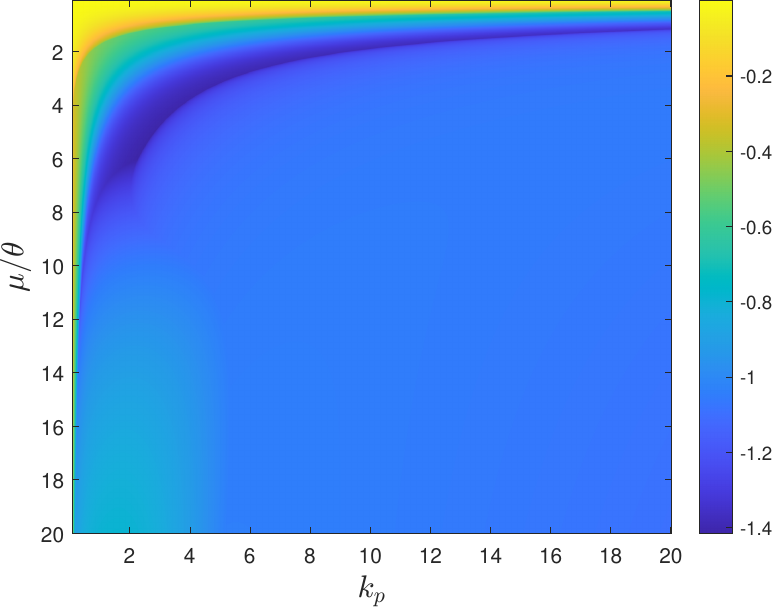}
  \end{center}
\end{minipage}
\begin{minipage}[h]{0.45\linewidth}
    \begin{center}
  \includegraphics[width=\textwidth]{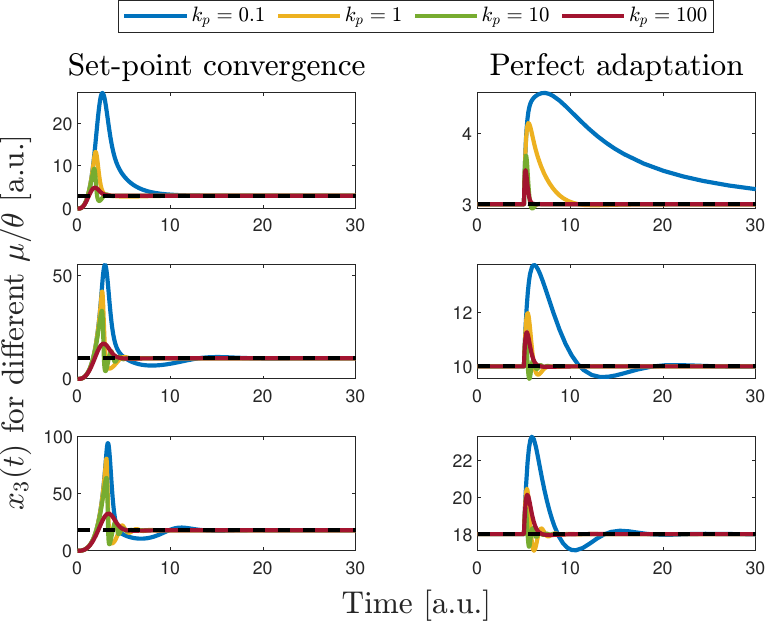}
  \end{center}
\end{minipage}
  \caption{\textbf{Left panel.} Spectral abscissa of the system associated with the \eqref{eq:RN:maturation2}, \eqref{eq:AIC:niAIC} with the parameters $\gamma_1=1$, $\gamma_2=\gamma_3=2$, $k_{21}=1$, $b_0=10$,  $k_{13}=1$, $\nu =3$, $\theta=1$, $\eta=100/k_p$ and for various values for $\mu$ and $k_p$. We have that which corresponds to
  $g_0=-5$, $g_n=-1/2$, and a spectral abscissa of $A$ given by $\alpha(A)=1.2695$ while the spectral abscissa of the closed-loop system reach smaller, negative values, which indicates that this controller is able to stabilize and improve the convergence properties of the system near the equilibrium point. \textbf{Right panel.} State response from a zero initial condition (first column) and as a response to a change of $k_{21}$ from 2 to 3.}\label{fig:persistent:2}
\end{figure}

\subsection{Partial structural stability analysis of the network controlled with an AIRC with output inhibition}

\black{Up to now all the results pertained to structural stability analysis of unimolecular networks controlled with an output-niAIC. We now connect the above results to the initial rein controller with output inhibition \eqref{eq:mainsystCL} for which we have now the following result:
\begin{theorem}
  Assume that $A$ is Metzler, output unstable, and nonsingular and that $g_0\ne 0$. Then, the unique equilibrium point of the system \eqref{eq:mainsystCL} is locally exponentially stable for all $\eta,k_p,\mu,\theta>0$ and all $k_i\in[0,\bar k_i)$.
\end{theorem}
\begin{proof}
Consider the general system \eqref{eq:mainsystCL} recalled below for simplicity:
  \begin{equation}\label{eq:kdlskdlskldksldksdkl}
  \begin{bmatrix}
   \dot x(t)\\
   \dot z_1(t)\\
   \dot z_2(t)
  \end{bmatrix}=F(x(t),z_1(t),z_2(t))+\begin{bmatrix}
    e_1\\
    0\\
    0
  \end{bmatrix}k_iz_1(t)
\end{equation}
where $k_i>0$ and
\begin{equation}\label{eq:kdlskdlskldksldksdkl2}
  F(x,z_1,z_2):=\begin{bmatrix}
    Ax+b_0-e_ne_n^Tk_pz_2\\
    \mu-\eta z_1z_2\\
    \theta x_n-\eta z_1z_2
  \end{bmatrix}.
\end{equation}

From Theorem \ref{th:main:unstable}, the unique equilibrium point of the above system for $k_i=0$ is structurally locally exponentially stable. We shall prove now that this is also the case for any sufficiently small $k_i>0$. Let $\bar{x}^*(k_i):=(x^*(k_i),z_1^*(k_i),z_2^*(k_i))$ be the equilibrium point of \eqref{eq:kdlskdlskldksldksdkl}-\eqref{eq:kdlskdlskldksldksdkl2}. Since the system is polynomial, this implies that the equilibrium point $\bar{x}^*(k_i)$ is a continuous function of $k_i$. In fact, by virtue of the implicit function theorem, we have that
\begin{equation}
  \dfrac{\d}{\d k_i}\begin{bmatrix}
    x^*(k_i)\\
    z_1^*(k_i)\\
    z_2^*(k_i)
  \end{bmatrix}=\tilde J(k_i)^{-1}\begin{bmatrix}
    e_1z_1^*(k_i)\\
    0\\
    0
  \end{bmatrix}
\end{equation}
where
\begin{equation}
  \tilde J(k_i)=\begin{bmatrix}
    A-e_ne_n^Tk_pz_2^*(k_i) & e_1k_i & -e_nk_pr\\
    0 & -\eta z_2^*(k_i) & -\eta z_1^*(k_i)\\
    \theta e_n^T & -\eta z_2^*(k_i) & -\eta z_1^*(k_i)
  \end{bmatrix},
\end{equation}
for all the $k_i$'s for which is it invertible. In fact, we proved that for $k_i=0$, the matrix $\tilde J(0)$ is Hurwitz stable for all $k_p,\eta>0$ provided that some conditions on $\mu,\theta>0$ and $A$ are satisfied. This implies that the equilibrium point $\bar{x}^*(k_i)$ is locally differentiable in $k_i$ around $k_i=0$. Now decompose $\tilde J(k_i)$ as
\begin{equation}
  J(k_i)=J_0+k_i J_1+(z_1^*(k_i)-z_1^*(0))J_2+(z_2^*(k_i)-z_2^*(0))J_3
\end{equation}
where
\begin{equation}
\begin{array}{rcl}
  J_0&=&\begin{bmatrix}
    A-e_ne_n^Tk_pz_2^*(0) & 0 & -e_nk_p\mu/\theta\\
    0 & -\eta z_2^*(0) & -\eta z_1^*(0)\\
    \theta e_n^T & -\eta z_2^*(0) & -\eta z_1^*(0)
  \end{bmatrix},\ J_1=\begin{bmatrix}
    0 & e_1 &0\\
    0 & 0 & 0\\
    0 & 0 & 0
  \end{bmatrix}\\
   J_2&=&\begin{bmatrix}
    0 & 0 &0\\
    0 & 0 & -\eta\\
    0 & 0 & -\eta
  \end{bmatrix},\ J_3=\begin{bmatrix}
    -e_ne_n^Tk_p & 0 &0\\
    0 & -\eta & 0\\
    0 &  -\eta & 0
  \end{bmatrix}.
\end{array}
\end{equation}
Therefore, we have that
\begin{equation}
  J(k_i)=J_0+k_i\left(J_1+J_2\left.\dfrac{\d z_1^*(k_i)}{\d k_i}\right|_{k_i=0}+J_3\left.\dfrac{\d z_2^*(k_i)}{\d k_i}\right|_{k_i=0}\right)+o(k_i)
\end{equation}
which shows that the eigenvalues are also locally continuous in $k_i$ at $k_i=0$. This implies that given all the other parameters of the system and the controller, there exists a $\bar k_i>0$ such that $\tilde J(k_i)$ is Hurwitz stable for all $k_i\in[0,\bar k_i]$.
\end{proof}}

\section{Structural stability of the niAIC - Nonlinear case}\label{sec:NL}

We provide immediate extensions of the results to other types of integral controllers that have been reported in the literature.

\subsection{Definitions, assumptions, and preliminary results}\label{sec:NL:prel}

\black{We now address the more general case where the network \eqref{eq:RN} is neither restricted to be unimolecular nor mass-action. The network is now described by the following model
\begin{equation}\label{eq:NL0}
\begin{array}{rcl}
  \dot{x}(t)&=&f(x(t))+b_0\\
  x(0)&=&x_0
%  y(t)&=&x_n(t)
\end{array}
\end{equation}
where $x(t),x_0\in\mathbb{R}^n$ are the state of the system and the initial condition, respectively. We assume here that the function $f$ is such that there exists a unique solution to that system which is defined for all times $t\ge0$. Additionally, since this model describes the evolution of the state of a reaction network, then the function $f_a$ is such that the solution $x(t)$ is nonnegative for all $t\ge0$ provided that $x(0),b_0$ are also nonnegative. Anticipating the interconnection of that network with an niAIC controller with output inhibition, we define the following extension
\begin{equation}\label{eq:NL}
\begin{array}{rcl}
  \dot{x}(t)&=&f(x(t))-e_nx_n(t)u(t)+b_0\\
  y(t)&=&x_n(t)\\
  x(0)&=&x_0
\end{array}
\end{equation}
where $u(t)\ge0$ is the input and $y(t)$ is the output of the system. It is immediate to see that this system is internally positive (i.e. for all $x(0),u(t),b_0\ge0$ we have that $x(t),y(t)\ge0$ for all $t\ge0$) since the solution to \eqref{eq:NL0} is nonnegative.}\\

\black{\begin{define}\label{def:sets}
  The geometry of equilibrium states and inputs is more complicated than in the unimolecular mass-action case and is described by the following sets:
\begin{itemize}
  \item the set of all possible equilibrium output values for a given input
  \begin{equation}
  %\mathscr{Y}:=\left\{e_n^Tx:\exists (x,u)\in\mathbb{R}_{\ge0}^n\times \mathbb{R}_{>0} \textnormal{ s.t. }f(x)-e_ne_n^Tu+b_0=0\right\},
  \mathscr{Y}(u):=\left\{e_n^Tx:\exists x\in\mathbb{R}_{\ge0}^n\textnormal{ s.t. }f(x)-e_ne_n^Tu+b_0=0\right\},\ u>0
\end{equation}
\item the set of all possible equilibrium output values across all possible input values
  \begin{equation}
  \mathscr{Y}:=\left\{e_n^Tx:\exists (x,u)\in\mathbb{R}_{\ge0}^n\times \mathbb{R}_{>0} \textnormal{ s.t. }f(x)-e_ne_n^Tu+b_0=0\right\}=\bigcup_{u>0}\mathscr{Y}(u),
\end{equation}
\item the set of equilibrium behaviors for a given equilibrium output
\begin{equation}
\mathscr{B}(y^\circ):=\left\{(x,u)\in\mathbb{R}^n_{\ge0}\times\mathbb{R}_{>0}: f(x)-e_ne_n^Tu+b_0=0, e_n^Tx=y^\circ\right\},\ y^\circ\in\mathscr{Y}
\end{equation}
\item the set of equilibrium inputs achieving a given value for the equilibrium output
\begin{equation}
\begin{array}{rcl}
  \mathscr{U}(y^\circ )&:=&\left\{u\in\mathbb{R}_{>0}:\exists x\in\mathbb{R}_{\ge0}^n \textnormal{ s.t. }f(x)-e_ne_n^Tu+b_0=0, e_n^Tx=y^\circ\right\},\ y^\circ\in\mathscr{Y}\\
                                        &:=&\left\{u\in\mathbb{R}_{>0}:\exists x\in\mathbb{R}_{\ge0}^n\textnormal{ s.t. }(x,u)\in\mathscr{B}(y^\circ)\right\},\ y^\circ\in\mathscr{Y},
\end{array}
\end{equation}
\item the set of equilibrium states corresponding to a given equilibrium output
\begin{equation}
\begin{array}{rcl}
  \mathscr{X}(y^\circ )&:=&\left\{x\in\mathbb{R}_{\ge0}^n:\exists u\in\mathbb{R}_{>0} \textnormal{ s.t. }f(x)-e_ne_n^Tu+b_0=0, e_n^Tx=y^\circ\right\},\ y^\circ\in\mathscr{Y}\\
                                        &:=&\left\{x\in\mathbb{R}_{\ge0}^n:\exists u\in\mathbb{R}_{>0}\textnormal{ s.t. }(x,u)\in\mathscr{B}(y^\circ)\right\},\ y^\circ\in\mathscr{Y}.
\end{array}
\end{equation}
\end{itemize}
We also define the global sets
\begin{equation}
\mathscr{B}:=\bigcup_{y^\circ\in\mathscr{Y}}\mathscr{B}(y^\circ),\ \mathscr{U}:=\bigcup_{y^\circ\in\mathscr{Y}}\mathscr{U}(y^\circ),\ \textnormal{and }\mathscr{X}:=\bigcup_{y^\circ\in\mathscr{Y}}\mathscr{X}(y^\circ).
\end{equation}
\end{define}
}

%\red{as well as $S(X)$ to denote a connected a subset of $X$ where $X$ is any of the sets above.}

\black{We make now the following assumption:
\begin{assumption}\label{hyp:nonempty}
  The set $\mathscr{Y}$ is non-empty and connected, and for all $y^\circ\in\mathscr{Y}$, the sets are $\mathscr{U}(y^\circ )$ and $\mathscr{X}(y^\circ )$ are singletons.
\end{assumption}}
\black{This assumption means that for any equilibrium value for the output, there exists a unique associated equilibrium state and constant input. Note that this does not imply that the system necessarily has a unique equilibrium point and the system is still allowed to have multiple equilibrium points. However, only one of them will be associated with the specified value for the output. A simple example is the system $\dot{x}=x(x-u)$, $y=x$. This assumption is not required but dramatically simplifies the exposition of the results.}\\

%\begin{example}
%  Consider the network
%\begin{equation}
%  \dot{x}(t)=x^2-ux+b_0,\ y=x
%\end{equation}
%Assuming $u^2\ge b_0$, then we have that
%\begin{equation}
%  x^*(u)=\dfrac{u\pm\sqrt{u^2-4b_0}}{2}>0,
%\end{equation}
%which shows that for a given $u^2\ge b_0$, there are two equilibrium points associated with it. Now pick a set-point $r$, then the corresponding control input is $u_*(r)=r+b_0/r>0$, which is unique. However, one can see that we have
%\begin{equation}
%  x^*(u_*(r))\in\left\{r,\dfrac{b_0}{r}\right\}.
%\end{equation}
%Now, we can observe that
%\begin{equation}
%  J(x)=2x-u
%\end{equation}
%and, as a result,
%\begin{equation}
%  J(x^*(u_*(r))=2r-u_*(r)=r-b_0/r
%\end{equation}
%which we will only be negative if $r^2<b_0$. So, in summary
%\begin{enumerate}
%  \item Existence of at least one nonnegative equilibrium point: $u^2\ge b_0$.
%  \item Stability of the $r$ equilibrium point: $r^2<b_0$.
%\end{enumerate}
%\end{example}

\black{Nevertheless, this is not enough for our results which also require the existence of a function that maps elements of $\mathscr{U}$ to elements of $\mathscr{X}$. The (local) existence of such a function is guaranteed by the following assumption
\begin{assumption}\label{hyp:jacobian}
  The matrices
  \begin{equation}\label{eq:implicit}
    J(x) \textnormal{ and } J(x)-ue_ne_n^T
  %\left.\dfrac{\partial f}{\partial x}\right|_{(x,u)\in\mathscr{X}(y)\times\mathscr{U}(y)}
  \end{equation}
  are invertible for all $(x,u)\in\mathscr{E}$ where $J(x):=\partial f(x)/\partial x$.
\end{assumption}

This assumption can be relaxed to consider a subset $\mathscr{R}$ of $\mathscr{E}$ on which those matrices are invertible. As for the previous assumption, this is not considered so as to simplify the exposition of the results but all the subsequent results can be adapted to this case at the expense of a higher notational cost.  This leads to the following result:
\begin{proposition}\label{prop:monotonic}
Assume that the conditions in Assumption \ref{hyp:nonempty} and \ref{hyp:jacobian} hold. Then, there exists (at least locally) a function $\sigma:\mathscr{U}\mapsto\mathscr{X}$ such that $f(\sigma(u))-e_ne_n^T\sigma(u)u+b_0=0$. Moreover, we have that
\begin{equation}
  \dfrac{\d F(u)}{\d u}=\dfrac{-g_n(u)}{1+g_n(u)u}F(u)
\end{equation}
where  $g_n(u):=-e_n^TJ(\sigma(u))^{-1}e_n$ and $F:=\pi_n\circ\sigma$ where $\pi_n$ denotes the projection operator on the $n$-th component; observe that $y_n^*(u)=e_n^Tx^*(u)=e_n^T\sigma(u)=F(u)$, locally on $u\in\mathscr{U}$.\\

Moreover, if there exists an $c>0$ such that
$$\sup_{u\in\mathscr{U}}\left|\dfrac{-g_n(u)y^*(u)}{1+g_n(u)u}\right|\le c,$$
then the function $F(u)$ is globally defined on $u\in\mathcal{U}$.
\end{proposition}
\begin{proof}
The existence of the function $\sigma$ is a consequence of Assumption \ref{hyp:jacobian} and the implicit function theorem. The same theorem also yields the expression
\begin{equation}\label{eq:dydu}
  \dfrac{\d\sigma(u)}{\d u}=(J(\sigma(u))-e_ne_n^Tu)^{-1}e_ne_n^T\sigma(u)
\end{equation}
at all the points where it is defined. In particular, this implies that
\begin{equation}
  \dfrac{\d y^*(u)}{\d u}=e_n^T(J(\sigma(u))-e_ne_n^Tu)^{-1}e_ny^*(u).
\end{equation}
Using Sherman-Morrison formula, we get that
\begin{equation}
\begin{array}{rcl}
   (J(\sigma(u))-e_ne_n^Tu)^{-1}&=&J(\sigma(u))^{-1}-\dfrac{(-1)J(\sigma(u))^{-1}e_ne_n^TJ(\sigma(u))^{-1}u}{1+(-1)e_n^TJ(\sigma(u))^{-1}e_nu}\\
                                       &=& J(\sigma(u))^{-1}+\dfrac{J(\sigma(u))^{-1}e_ne_n^TJ(\sigma(u))^{-1}u}{1+g(u)u}
\end{array}
\end{equation}
where $g_n(u)=-e_n^TJ(\sigma(u))^{-1}e_n$. Therefore,
\begin{equation}
\begin{array}{rcl}
  e_n^T(J(\sigma(u))-e_ne_n^Tu)^{-1}e_n&=&e_n^T\left(J(\sigma(u))^{-1}+\dfrac{J(\sigma(u))^{-1}e_ne_n^TJ(\sigma(u))^{-1}u}{1+g(u)u}\right)e_n\\
                                    &=& -g_n(u)+\dfrac{g_n(u)^2u}{1+g_n(u)u}\\
                                    &=& -\dfrac{g_n(u)}{1+g_n(u)u},
\end{array}
\end{equation}
from which the first result follows. Finally, the global existence of $F$ on $\mathscr{U}$ comes from the linear growth of the right hand side of the differential equation and the boundedness of the Lipschitz constant on $\mathscr{U}$.
\end{proof}}

\black{Unlike in the unimolecular mass-action case, the "gain" $g_n(u)$ now depends on the value of the control input, which may lead to more complex behaviors than in the unimolecular mass-action case. Yet, it still exhibits useful regularity properties, as stated in the result below:
\begin{proposition}
  Assume that the conditions in Assumption \ref{hyp:nonempty} and \ref{hyp:jacobian} hold, and that $J(x)-ue_ne_n^T$ is Hurwitz stable for all $(x,u)\in\mathscr{B}$. Then, we have that
\begin{enumerate}[(a)]
  \item $\sign(1+g_n(u)u)=\sign(\det(-J(\sigma(u))))$, and
  \item $\sign(g_n(u))=\sign\left(\dfrac{\det(-J_{11}(\sigma(u)))}{\det(-J(\sigma(u)))}\right)$.
\end{enumerate}
Moreover, if $J(x)$ is Hurwitz stable for all $x\in\mathscr{X}$, we have that
\begin{enumerate}[(a)]
 \setcounter{enumi}{2}
  \item $1+g_n(u)u>0$ holds for all $u\in\mathscr{U}$, and
  \item $g_n(u)>0$ if and only if $\det(-J_{11}(\sigma(u)))>0$.
\end{enumerate}
\end{proposition}
\begin{proof}
   Since  $J(x)-ue_ne_n^T$ is Hurwitz stable on $\mathscr{B}$ and $J(x)$ invertible on $\mathscr{X}$, then we have that $0<\det(-J(x)+e_ne_n^Tu)=\det(-J(x))(1+g_n(u)u)$. Therefore $1+g_n(u)u$ has the same sign as $\det(-J(x))$ on $\mathscr{B}$. Similarly,
   \begin{equation}
   g_n(u)=-e_n^TJ(\sigma(u))^{-1}e_n=\dfrac{\det(-J_{11}(\sigma(u)))}{\det(-J(\sigma(u)))},
   \end{equation}
   which proves the second statement. Assuming that $J(\sigma(u))$ is Hurwitz stable on $\mathscr{U}$ implies that $\det(-J(\sigma(u)))>0$ on $\mathscr{U}$, from which the two last statements follow.
\end{proof}}

%$$\begin{tikzcd}
%\red{	{\mathscr{U}} &&& {\mathscr{X}} &&& {\mathscr{Y}}
%	\arrow["{\pi_n}", maps to, from=1-4, to=1-7]
%	\arrow["\textnormal{\tiny Definition \ref{def:sets}}"', maps to, from=1-4, to=1-7]
%	\arrow["\sigma", maps to, from=1-1, to=1-4]
%	\arrow["\textnormal{\tiny Proposition \ref{prop:monotonic}}"', maps to, from=1-1, to=1-4]
%	\arrow["F", curve={height=-30pt}, maps to, from=1-1, to=1-7]
%    \arrow["\textnormal{\tiny Proposition \ref{prop:monotonic}}"', curve={height=-30pt}, maps to, from=1-1, to=1-7]
%    \arrow["F^{-1}"', curve={height=-30pt}, maps to, from=1-7, to=1-1]
%    \arrow["\textnormal{\tiny kdlsdks}", curve={height=-30pt}, maps to, from=1-7, to=1-1]}
%\end{tikzcd}$$
%
%
%  \begin{equation*}
%    \mathscr{U}\xmapsto[\textnormal{}]{\sigma}\mathscr{X}\xmapsto[\textnormal{Definition \ref{def:sets}}]{\pi_n}\mathscr{Y}
%  \end{equation*}

\black{\begin{assumption}\label{hyp:nonlinear_gain}
  The nonlinear gain function  $F=\pi_n\circ\sigma:\mathscr{U}\mapsto\mathscr{Y}$, defined in Proposition \ref{prop:monotonic}, is bijective.
\end{assumption}

%This, together with the previous assumptions and results, mean that the functions $\sigma$ and $F$ are bijections from $\mathscr{U}\mapsto\mathscr{X}$ and $\mathscr{U}\mapsto\mathscr{Y}$, respectively.
This leads us to the following result:
\begin{proposition}
  Consider the system \eqref{eq:NL} which is assumed to satisfy Assumption \ref{hyp:nonempty}, \ref{hyp:jacobian}, and \ref{hyp:nonlinear_gain}. Then, the equilibrium point $(x^*(r), u_*(r))$ associated with the steady-state output $y^*=r\in\mathscr{Y}$  for the system \eqref{eq:NL} is given by
  \begin{equation}
  u_*(r)=F^{-1}(r),\textnormal{and\ } x^*(r)=\sigma(F^{-1}(r)).
\end{equation}
\end{proposition}
\begin{proof}
  This follows from Proposition \ref{prop:monotonic} and Assumption \ref{hyp:nonlinear_gain}.
\end{proof}}

%\subsection{Main results}

\subsection{General results}

\black{We now move on with the main results of the section. The closed-loop network consisting of the interconnection of the network \eqref{eq:RN} described by \eqref{eq:NL0} and the niAIC \eqref{eq:AIC:niAIC_output} is represented by the following system
\begin{equation}\label{eq:NLCL}
\begin{array}{rcl}
  \dot{x}(t)&=&f(x(t))-e_nz_\iota(t)x_n(t)+b_0\\%,\ x(0)=x_0\\
   \dot{z}_1(t)&=&\mu-\eta k_pz_1(t)z_2(t)\\%,\ z_1(0)=z_{0,1}\\
   \dot{z}_2(t)&=&\theta e_n^Tx-\eta k_pz_1(t)z_2(t),%,\ z_2(0)=z_{0,2}
\end{array}
\end{equation}
where $(x(0),z_1(0),z_2(0))=(x_0,z_{1,0},z_{2,0})$. Note that we have left free the choice of the actuating species which can either be $z_1$ or $z_2$ depending on whether $\iota=1$ or $\iota=2$, respectively. This leads to the following result:}

\black{\begin{proposition}
Assume that Assumption \ref{hyp:nonempty},  \ref{hyp:jacobian}, and \ref{hyp:nonlinear_gain} are satisfied for the systems \eqref{eq:NL0} and \eqref{eq:NL}. Then, the unique equilibrium point of the closed-loop system \eqref{eq:NLCL} is given by
\begin{equation}\label{eq:point3}
 x^*(r)=g(F^{-1}(r)),\ z_\iota^*(r)=\dfrac{\mu}{\eta  F^{-1}(r)}\ \textnormal{and }z_{3-\iota}^*(r)=\dfrac{F^{-1}(r)}{k_p}.
\end{equation}

Moreover, defining $J^*(r):=J(x^*(r))$, we define the transfer function
\begin{equation}
  H_n(s,r)=e_n^T(sI-(J^*(r)-F^{-1}(r)e_ne_n^T))^{-1}e_n.
\end{equation}

Finally, the linearized dynamics of the system \eqref{eq:NLCL} about the equilibrium point \eqref{eq:point3} when $\iota=1$ is governed by the matrix
\begin{equation}\label{eq:linearized:general:z1}
M_1(r):=\begin{bmatrix}
    J^*(r)-F^{-1}(r)e_ne_n^T & -k_pe_nr & 0\\
  0 & -\mu k_p/F^{-1}(r) & -\eta F^{-1}(r)\\
  \theta e_n^T & -\mu k_p/F^{-1}(r) & -\eta F^{-1}(r)
\end{bmatrix}
\end{equation}
whereas it is given by
\begin{equation}\label{eq:linearized:general}
M_2(r):=\begin{bmatrix}
    J^*(r)-F^{-1}(r)e_ne_n^T & 0 & -k_pe_nr\\
  0 & -\eta F^{-1}(r) & -\mu k_p/F^{-1}(r)\\
  \theta e_n^T & -\eta F^{-1}(r) & -\mu k_p/F^{-1}(r)
\end{bmatrix}
\end{equation}
when $\iota=2$.
\end{proposition}
\begin{proof}
 The proof follows from the previous results and assumptions and the standard linearization procedure.
\end{proof}}

\black{It is important to clarify under what conditions there exists a $k_p>0$ such that the equilibrium point \eqref{eq:point3} is locally asymptotically stable for the dynamics \eqref{eq:NLCL}. This is stated in the result below:
\begin{proposition}
   Assume that Assumption \ref{hyp:nonempty},  \ref{hyp:jacobian}, and \ref{hyp:nonlinear_gain} are satisfied for the systems \eqref{eq:NL0} and \eqref{eq:NL}. Assume further that $r$ is admissible (i.e. $r\in\mathscr{Y}$), that $J^*(r)-F^{-1}(r)e_ne_n^T$ is Hurwitz stable, and let $\iota\in\{1,2\}$. Then, the following statements hold:
   \begin{enumerate}[(a)]
     \item If $(-1)^\iota H_n(0,r)>0$, then there exists a sufficiently small $k_p>0$ such that the equilibrium point \eqref{eq:point3} is locally exponentially stable for the dynamics \eqref{eq:NLCL}
     \item If $(-1)^\iota H_n(0,r)<0$, then, for all $k_p>0$, the equilibrium point  \eqref{eq:point3} is unstable for the dynamics \eqref{eq:NLCL}.
   \end{enumerate}
\end{proposition}
\begin{proof}
  The proof follows from the same arguments as the proof of Proposition \ref{prop:small_kp} and relies on proving that there exists a $k_p>0$ such that the matrix \eqref{eq:linearized:general}  or \eqref{eq:linearized:general} is Hurwitz stable. In the current case, the zero-eigenvalue locally obeys the expression
  \begin{equation}
  \begin{array}{rcl}
      \lambda(k_p)     &=&      (-1)^\iota\mu e_n^T(J^*(r)-F^{-1}(r)e_ne_n^T)^{-1}e_n k_p+o(k_p)\\
                                    &=&       (-1)^{\iota-1}\mu H_n(0,r)k_p+o(k_p)\\
                                    &=&       (-1)^{\iota}\theta\left.\dfrac{\d F(u)}{\d u}\right|_{u=F^{-1}(r)}k_p+o(k_p)\\
                                    &=&       (-1)^{\iota-1}\theta\dfrac{g_n(F^{-1}(r))}{1+g_n(F^{-1}(r))F^{-1}(r)}k_p+o(k_p).
  \end{array}
  \end{equation}
One can observe that stabilization is only possible if $(-1)^{\iota}H_n(0,r)>0$, which proves the first statement of the result.\\

Now assume that $(-1)^\iota H_n(0,r)<0$. Then, using the Schur complement and the properties of the determinant, we have that
\begin{equation}
     \det(-M_\iota(r))=\xi(r)(-1)^\iota H_n(0,r),\ \iota=1,2,
\end{equation}
where $\xi(r):=\eta F^{-1}(r)k_p\mu\det(-J^*(r)+F^{-1}(r)e_ne_n^T)$. Since $J^*(r)-F^{-1}(r)e_ne_n^T$ is Hurwitz stable, then $\det(-J^*(r)+F^{-1}(r)e_ne_n^T)>0$ and, therefore, $\xi(r)>0$, which implies that $\det(-M_\iota(r))<0$ since $(-1)^\iota H_n(0,r)<0$, implying, in turn, that $M_\iota(r)$ has at least one eigenvalue with positive real part, implying that the equilibrium point is unstable. This proves the second statement.
\end{proof}}

\black{It seems important to discuss the meaning of that result. It was previously explained that it is necessary that the feedback loop from $\Yz$ to $\Yz$ be negative and the map from $\mu$ to $\Yz$ be increasing for the equilibrium point to be locally exponentially stable as illustrated in Figure~\ref{fig:zvsz2}.  For those properties to be satisfied, the function $F$ needs to be increasing (i.e. activating) whenever $\iota=1$ and decreasing (i.e. inhibiting) whenever $\iota=2$.}\\

\begin{figure}
  \centering
  \includegraphics[width=.5\textwidth]{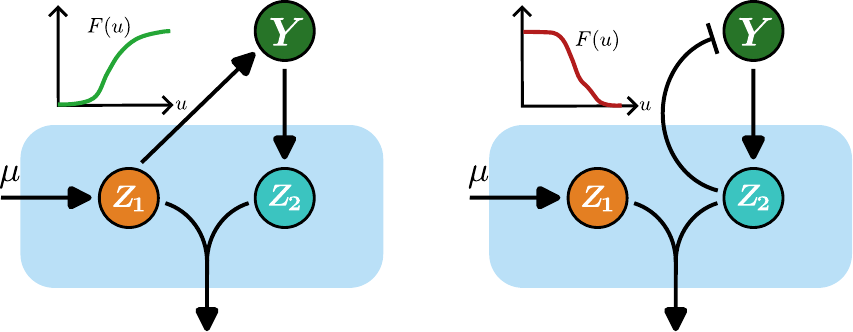}
  \caption{Valid control topologies. \textbf{Left.} Case $\iota=1$ where $\Z{1}$ is the actuating species. In that case, the map $F$ needs to be increasing for the map $\mu\mapsto z_1^*$ to be increasing and the feedback loop from $\Yz$ to $\Yz$ to be negative. \textbf{Left.} Case $\iota=2$ where $\Z{2}$ is the actuating species. In that case, the map $F$ needs to be decreasing for the map $\mu\mapsto z_1^*$ to be increasing and the feedback loop from $\Yz$ to $\Yz$ to be negative.}\label{fig:zvsz2}
\end{figure}

%In the case where $\iota=1$, we have the following diagram
%\begin{equation*}
%  \Yz\rarrow{\textrm{activation}}\Z{2}\rarrow{\ominus}\Z{1}\rarrow{F(\cdot)}\Yz
%\end{equation*}
%whereas when $\iota=2$ this diagram becomes
%\begin{equation*}
%\Yz\rarrow{\oplus}\Z{2}\rarrow{F(\cdot)}\Yz.
%\end{equation*}
%In the first scenario, $\iota=1$, the loop can only be negative if $F$ is an activating function, that is, it is an increasing function on $\mathscr{U}$ , which is equivalent to say that $H(0,r)<0$ for all $r\in\mathscr{Y}$, whereas in the second scenario, $\iota=1$, the loop can only be negative if $F$ is an inhibiting function, that is, it is a decreasing function on $\mathscr{U}$ , which is equivalent to say that $H(0,r)>0$ for all $r\in\mathscr{Y}$.

\black{We are now ready to state the main result of this section:
\begin{theorem}\label{th:generalNL}
Assume that Assumption \ref{hyp:nonempty},  \ref{hyp:jacobian}, and \ref{hyp:nonlinear_gain} are satisfied for the systems \eqref{eq:NL0} and \eqref{eq:NL}. Assume further that $r$ is admissible (i.e. $r\in\mathscr{Y}$), that $H_n(0,r)>0$, and that one of the following equivalent statements holds:
  \begin{enumerate}[(a)]
    \item There exist a matrix $P_1(r)=P_1(r)^T\succ0$ and an $\eps>0$ such that
    \begin{equation}\label{eq:LMI:NL}
        (J^*(r)-F^{-1}(r)e_ne_n^T)^TP(r)+P(r)(J^*(r)-F^{-1}(r)e_ne_n^T)+2\eps e_ne_n^T\prec0
    \end{equation}
    where
    \begin{equation}
        P(r) = \begin{bmatrix}
          P_1(r) & 0\\
          0 & 1
        \end{bmatrix}.
    \end{equation}
    \item There exists an $\eps>0$ such that the system $(J_{11}^*(r), J_{12}^*(r), -J_{21}^*(r), u_*(r)-J_{22}^*(r)-\eps)$ is SPR.
  \end{enumerate}
  Then, the equilibrium point \eqref{eq:point3} of the system \eqref{eq:NLCL} is locally exponentially stable for all $\eta,k_p>0$. If it holds for all $r\in\mathscr{Y}$, then the closed-loop system is also structurally stable for all admissible set-point.
\end{theorem}
\begin{proof}
  The equivalence between the statements follows from the fact that the condition \eqref{eq:LMI:NL} in the first statement can be written as
  \begin{equation}\label{eq:LMI:NL2}
  \begin{bmatrix}
    P_1J_{11}^*(r)+J_{11}^*(r)^TP_1(r) & P_1(r) J_{12}^*(r)+J_{21}^*(r)^T\\
    \star & 2(J_{22}^*(r)-u_*+\eps)
  \end{bmatrix}\prec0
\end{equation}
which is a time-domain characterization of the SPR condition; i.e. input strict passivity or strong strict positive realness. Now, let us prove that all the conditions for SPR are met. First of all, the condition \eqref{eq:LMI:NL} implies that $J^*(r)-F^{-1}(r)e_ne_n^T$ is Hurwitz stable (stable poles). Moreover, the condition \eqref{eq:LMI:NL2} implies that $J_{22}^*(r)-u_*<0$  and that $J_{11}^*(r)$ is Hurwitz stable (stable zeros). Those two properties all together implies that the condition at infinity holds. Finally, from the Kalman-Yakubovich-Popov Lemma we have that $\Re[H_n(j\omega,r)]\ge \eps |H_n(j\omega,r)|^2$. Since the zeros of the transfer function $H_n(s,r)$ are stable, then $ |H_n(j\omega,r)|\ne0$ for all $\omega\in\mathbb{R}$ and from the assumption that $H_n(0,r)>0$, we get that $\Re[H_n(j\omega,r)]>0$ for all $\omega\in\mathbb{R}$. The conclusion then follows.
\end{proof}}

\black{We have the immediate corollary:
\begin{corollary}\label{cor:generalNL_triangular}
  Assume that Assumption \ref{hyp:nonempty},  \ref{hyp:jacobian}, and \ref{hyp:nonlinear_gain} are satisfied for the systems \eqref{eq:NL0} and \eqref{eq:NL}. Assume further that $r$ is admissible (i.e. $r\in\mathscr{Y}$), that $H_n(0,r)>0$, and that one of the following statements holds:
  \begin{enumerate}[(a)]
    \item\label{cor:generalNL_triangular:st1} $J^*_{12}(r)=0$ and $J^*(r)$ is Hurwitz stable;
    \item\label{cor:generalNL_triangular:st2} $J^*_{21}(r)=0$ and $J^*(r)$ is Hurwitz stable;
  \end{enumerate}
  Then, the equilibrium point \eqref{eq:point3} of the system \eqref{eq:NLCL} is locally exponentially stable for all $\eta,k_p>0$.
\end{corollary}
\begin{proof}
From the conditions, we have that $J^*(r)$ is triangular with $J^*_{22}(r)<0$, which implies that $J^*_{22}(r)-u_*<0$. Therefore, the linear matrix inequality \eqref{eq:LMI:NL2} holds if and only if
  \begin{equation}\label{eq:kdsldklka}
     P_1(r)J^*_{11}(r)+J^*_{11}(r)^TP_1(r)+ \dfrac{1}{2(u_*-J^*_{22}(r))}(P_1(r) J^*_{12}(r)+J^*_{21}(r)^T)(J^*_{12}(r)^TP_1(r)+J^*_{21}(r))\prec0.
  \end{equation}

\textbf{Case \eqref{cor:generalNL_triangular:st1}.} Assume that $ J^*_{12}(r)=0$. Therefore, the inequality \eqref{eq:kdsldklka} reduces to
  \begin{equation}\label{eq:kdsldklka2}
     P_1(r)J^*_{11}(r)+J^*_{11}(r)^TP_1(r)+\dfrac{1}{2(u_*-J^*_{22}(r))}J^*_{21}(r)^TJ^*_{21}(r)\prec0.
  \end{equation}
  Since $J^*_{11}(r)$ is Hurwitz stable, then one can find a $P_1(r)$ such that the above inequality holds, which proves this case.\\

  \textbf{Case \eqref{cor:generalNL_triangular:st2}.} Assume that $J^*_{21}(r)=0$ and let $Q_1(r)$ be such that $Q_1(r)J^*_{11}(r)+J^*_{11}(r)Q_1(r)=-R_1(r)$ for some $R_1(r)\succ0$. Such a $Q_1(r)$ exists as $J^*_{11}(r)$ is Hurwitz stable. Let $P_1(r)=\eps(r) Q_1(r)$ for some $\eps(r)>0$. Then, the condition \eqref{eq:kdsldklka} reduces to
    \begin{equation}
    -\eps(r) R_1(r)+ \dfrac{\eps(r)^2}{2(u_*-J^*_{22}(r))}Q_1(r) J^*_{12}(r)J^*_{12}(r)^TQ_1(r)\prec0.
  \end{equation}
  Dividing by $\eps(r)>0$, this equivalent to say that
    \begin{equation}
    -R_1(r)+ \dfrac{\eps(r)}{2(u_*-J^*_{22}(r))}Q_1(r) J^*_{12}(r)J^*_{12}(r)^TQ_1(r)\prec0.
  \end{equation}
  This inequality is readily shown to be negative definite provided that $\eps(r)>0$ is small enough, which proves this case.
\end{proof}}

\black{Let us illustrate Theorem \ref{th:generalNL} through this simple example
\begin{example}
    Consider a gene expression network where the protein represses its own production described by
  \begin{equation}\label{eq:ex:feedback}
    \begin{array}{rcl}
              \dot{x}_1&=&-\gamma x_1+\dfrac{k_{12}}{1+x_2}+b_0,\\%\qquad       \dot{x}_2=kx_1-\gamma x_2-u
              \dot{x}_2&=&k_{21}x_1-\gamma x_2-ux_2
    \end{array}
  \end{equation}
  where $x_1$ denotes the mRNA concentration and $x_2$ the protein concentration. The parameters $\gamma,k_{12},k_{21},b_0$ are all positive. The admissibility set for this system is given by
  \begin{equation}
    \mathscr{Y}:=(0,r_\mn)
  \end{equation}
  where $r_{\mn}$ is the unique positive root of the polynomial
\begin{equation}
  Y(r):=\gamma^2r^2+r(\gamma^2-k_{21}b_0)-k_{21}(k_{12}+b_0).
\end{equation}
  Moreover, the sets
  \begin{equation}
    \mathscr{X}(y^\circ)=\left\{\left(\dfrac{1}{\gamma}\left[\frac{k_{12}}{1+y^\circ}+b_0\right], y^\circ\right)\right\}
  \end{equation}
  and
  \begin{equation}
    \mathscr{U}(y^\circ)=\left\{-\gamma+\dfrac{k_{21}}{\gamma y^\circ}\left[\dfrac{k_{12}}{1+y^\circ}+b_0\right]\right\},
  \end{equation}
  both consist of singletons, as required by Assumption \ref{hyp:nonempty}. The Jacobian of  the system \eqref{eq:ex:feedback} with $u=0$ is given by
  \begin{equation}
    J(x)=\begin{bmatrix}
        -\gamma & -\dfrac{k_{12}}{(1+x_2)^2}\\
      k_{21} & -\gamma
    \end{bmatrix}
  \end{equation}
  and is invertible for all $x_2\ge0$, and so is $J(x)-ue_ne_n^T$ for all $x\ge0$ and $u\ge0$. Therefore, the conditions in Assumption \ref{hyp:jacobian} are satisfied. As a result, maps $\sigma: \mathscr{U}\mapsto \mathscr{X}$ and $F: \mathscr{U}\mapsto \mathscr{Y}$ exist. Furthermore, we have that $g_n(u)=\gamma(1+x_2^*(u))^2/(\gamma^2+k_{12}k_{21})>\gamma/(\gamma^2+k_{12}k_{21})$, which implies, by virtue of Proposition \ref{prop:monotonic}, that the map $F$ is monotonically decreasing. This shows that the condition in Assumption \ref{hyp:nonlinear_gain} is also satisfied.\\

The Jacobian evaluated at the equilibrium point corresponding to $x_2^*=r$ is given by
  \begin{equation*}
    J^*(r)=\begin{bmatrix}
      -\gamma & -\dfrac{k_{12}}{(1+r)^2}\\
      k_{21} & -\gamma
    \end{bmatrix}
  \end{equation*}
  and the associated transfer function is $$H_n(s,r)=\dfrac{s+\gamma}{s^2+(2\gamma+u_*(r))s+\gamma(\gamma+u_*(r))+k_{21}k_{12}/(1+r)^2}.$$ From the Routh-Hurwitz criterion, the transfer function $H_n(s,r)$ is stable, has stable zeros and verifies $H(0,r)>0$. Therefore, we just need to show the existence of a $P(r)\succ0$ such that
  \begin{equation*}
    \begin{bmatrix}
    -2\gamma P(r) & -\dfrac{k_{12}}{(1+r)^2}P(r)+k_{21}\\
    -\dfrac{k_{12}}{(1+r)^2}P(r)+k_{21} & -2(\gamma+u_*(r))
  \end{bmatrix}\prec0.
  \end{equation*}
Alternatively, we may check whether the transfer function associated with the system
\begin{equation}
  \begin{array}{rcl}
    \dot{v}&=&-\gamma v-\dfrac{k_{12}}{(1+r)^2}w\\
    z&=&-k_{21}v+(\gamma+u_*(r))w
  \end{array}
\end{equation}
is strictly positive real. It can be seen that this is the case since it is a first order stable transfer function of relative degree 0 with stable zeros, with positive gain and positive feedthrough. Therefore, the closed-loop system consisting of \eqref{eq:ex:feedback} and the niAIC \eqref{eq:AIC:niAIC_output} is structurally stable with respect to $k_p,\eta>0$ and $r>r_{\mn}$.
\end{example}}

\begin{figure}[H]
  \centering
  \begin{minipage}[h]{0.45\linewidth}
    \begin{center}
    \includegraphics[width=\textwidth]{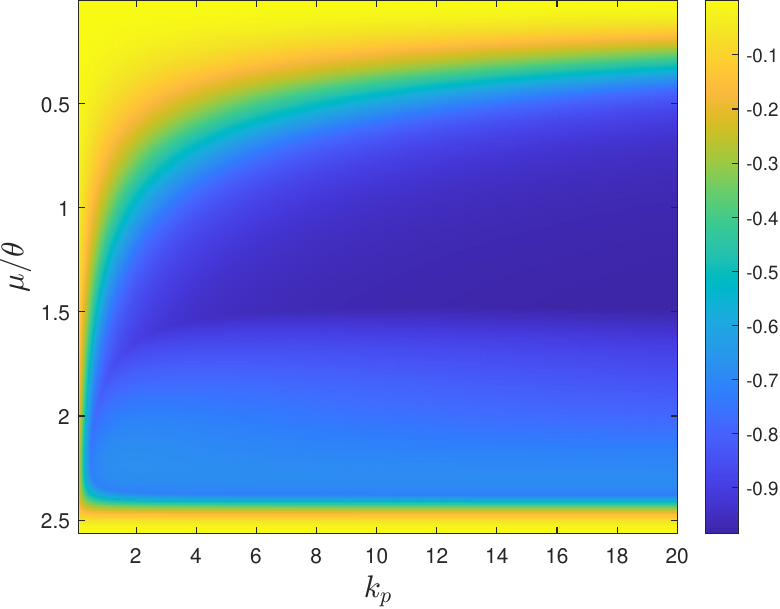}
  \end{center}
\end{minipage}
\begin{minipage}[h]{0.45\linewidth}
    \begin{center}
  \includegraphics[width=\textwidth]{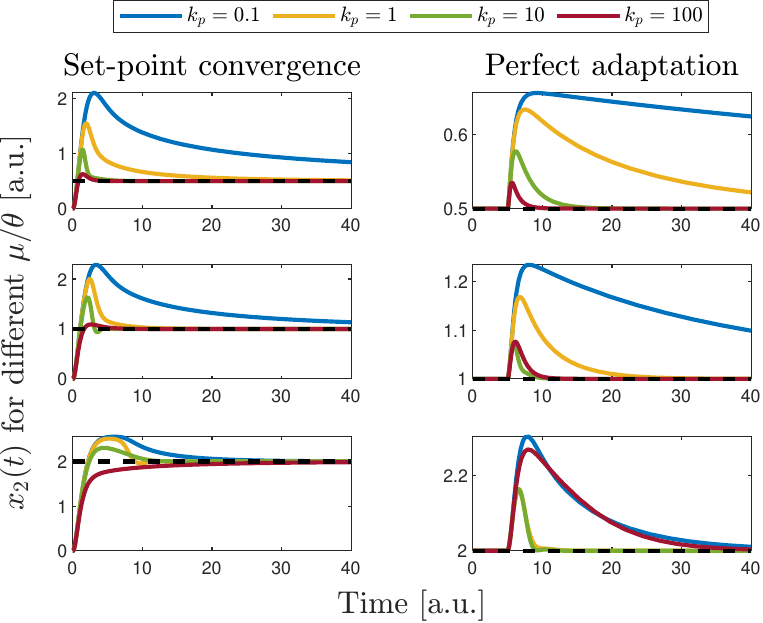}
  \end{center}
\end{minipage}
  \caption{\textbf{Left panel.} Spectral abscissa of the system associated with the \eqref{eq:ex:feedback}, \eqref{eq:AIC:niAIC} with the parameters $\gamma=1$, $k_{21}=2$, $b_0=1$,  $k_{21}=2$, $\theta=1$, $\eta=100/k_p$ and for various values for $\mu$ and $k_p$. We have that $r_{\mn}=2.5616$. \textbf{Right panel.} State response from a zero initial condition (first column) and as a response to a change of $k_{12}$ from 1 to 2.}
\end{figure}

\subsection{Cooperative and Michaelis-Menten networks}

\black{In some cases, the conditions in Theorem \ref{th:generalNL} can be shown to be automatically satisfied (as in the unimolecular case) and the only conditions that will need to be checked are those stated in the various assumptions introduced in Section \ref{sec:NL:prel}.\\

The class of cooperative systems benefits from such a simplified procedure, as stated in the result below:
\begin{proposition}\label{prop:cooperative}
  Assume that the set of admissible set-point $\mathscr{Y}$ is non-empty, that the function $f$ is cooperative\footnote{A function $f:X\mapsto X$, $X\subset\mathbb{R}^n$,  is cooperative if $\partial f/\partial x$ is Metzler for all $x\in X$.}  and such that $J^*(r)$ is Hurwitz stable for all $r\in\mathscr{Y}$. Then, the equilibrium point \eqref{eq:point3} of the system \eqref{eq:NLCL} is locally exponentially stable for all $\eta,k_p>0$ and all $r\in\mathscr{Y}$.
\end{proposition}
\begin{proof}
  Since $f$ is cooperative, then $J(x)$ is Metzler. Moreover, if $J^*(r)$ is Hurwitz stable for all $r\in\mathscr{Y}$, then so is $J^*(r)-u_*(r)e_ne_n^T$, and this implies that both matrices are invertible, which shows that the conditions in Assumption \ref{hyp:jacobian} are satisfied. Since $J^*(r)$ is Metzler and Hurwitz stable, then we have that $g_n(u)>0$ (see Lemma \ref{lemma:gains}) for all $u\in\mathscr{U}$, which implies that $F(u)$ is monotonically decreasing on $\mathscr{U}$ and shows that the condition in Assumption \ref{hyp:nonlinear_gain} is verified. From Lemma \ref{lemma:gains}, we also have that $H_n(0,r)>0$. Finally, using the fact that  $J^*(r)-u_*(r)e_ne_n^T$ is Metzler and Hurwitz stable, Proposition \ref{prop:MHS} allows us to state that there exists a diagonal matrix $P_1(r)\succ0$ that satisfies \eqref{eq:LMI:NL}. The result then follows from Theorem \ref{th:generalNL}.
\end{proof}}

\black{The class of cooperative networks is very similar to unimolecular ones, which are cooperative by construction. As a result, extensions to the case of output-unstable systems are rather immediate at the expense of additional notational burden stemming from the nonlinear nature of the network. For this reason, those results are omitted. Without entering into details, it is also possible to go beyond the class of cooperative systems by considering a more general class of systems which are cooperative with respect to a different cone than the nonnegative orthant; see e,g, \cite{Walcher:01}. Indeed, if one can find a diagonal matrix $S(r)$ with diagonal entries in $\{-1,1\}$ such that $S(r)J^*(r)S(r)$ is Metzler, one can show that there also exists a diagonal matrix $P_1(r)$ such that the inequality \eqref{eq:LMI:NL} holds. This implies that networks described by such matrices also exhibit the properties stated in Proposition \ref{prop:cooperative}.}\\

\black{Another important class of networks having beneficial properties is the class of networks governed by a combination of linear mass-action and Michaelis-Menten dynamics. The following lemma summarizes some important results proven in \cite{Blanchini:23}:
\begin{lemma}\label{lemma:MM}
  Assume that \eqref{eq:NL0} describes a reaction network with linear mass-action and Michaelis-Menten dynamics, that is, $f(x)=Ax+N(x)$ where $A$ is Metzler and Hurwitz stable and where $N(x)$ contains the Michaelis-Menten terms. Assume further that the graph of the network is strongly connected and at least one entry of $b_0$ is positive. Then,
  \begin{enumerate}[(a)]
    \item There exists a unique equilibrium point $x^*$ in the positive orthant.
    \item The Jacobian matrix $J(x^*)$ is row diagonally dominant with negative diagonal and is, therefore, Hurwitz stable, which implies that the unique equilibrium point $x^*$ is locally exponentially stable.
  \end{enumerate}
\end{lemma}
The condition that $b_0$ is nonzero is here to ensure that there at least one molecular species is constitutively produced. The strong connectedness condition of the graph of the network enforces a "mixing property" guaranteeing that all the species act on each other in a persistent way, making the equilibrium point positive. This latter condition is sufficient only for the positivity of the equilibrium point and weaker ones can be formulated, possibly on a case-by-case basis.\\}

\black{Based on the above result, we can state our main result regarding the control of Michaelis-Menten networks with unimolecular mass-action reactions using an output niAIC:
\begin{theorem}\label{th:stability:MM}
Assume that Assumption \ref{hyp:nonempty} and \ref{hyp:nonlinear_gain} are satisfied for the systems \eqref{eq:NL0} and \eqref{eq:NL}, which represents a Michaelis-Menten reaction networks with $f(x)=Ax+N(x)$ where $A$ Metzler and Hurwitz stable and $N(x)$ contains the nonlinear Michaelis-Menten terms.  Assume further that the graph of the network is strongly connected, that $b_0\ne0$, that $r$ is admissible (i.e. $r\in\mathscr{Y}$), that $\mathscr{X}\subset\mathbb{R}_{>0}^d$, and that $H_n(0,r)>0$.\\

Then, for all $r\in\mathscr{Y}$ the equilibrium point \eqref{eq:point3} of the system \eqref{eq:NLCL} is locally exponentially stable for all $\eta,k_p>0$.
\end{theorem}
\begin{proof}
From Lemma \ref{lemma:MM}, the matrix $J^*(r)$ is Hurwitz stable for all $r\in\mathscr{Y}$ because it is row diagonally dominant with negative diagonal. This also implies that both $J^*(r)-u_*(r)e_ne_n^T$ and $J_{11}^*(r)$ are row diagonally dominant with negative diagonal. This shows that the conditions in Assumption \ref{hyp:jacobian} are verified. It remains to show that one of the conditions in Theorem \ref{th:generalNL} is satisfied. The existence of a diagonal matrix $P_1(r)$ is automatically follows from \cite{Kaszkurewicz:00,Kushel:19} and the fact that the matrix $J_{11}^*(r)$ is row diagonally dominant with negative diagonal.
\end{proof}}

\black{\section{Extension to other types of integral controllers}

\subsection{Exponential integral controllers}\label{sec:exp}

We now consider the following class of exponential integral controller \cite{Briat:16a,Xiao:18b,Briat:19:Logistic}
\begin{equation}\label{eq:logistic}
  \Zz+\Yz\rarrow{\alpha}\Zz+\Yz+\Zz,\ \Zz\rarrow{\alpha r}\phib
\end{equation}
described as
\begin{equation}
  \dot{z}(t)=\alpha z(t)(y(t)-r)
\end{equation}
where $\alpha>0$ is a positive parameter of the controller and $r$ is the desired set-point. The closed-loop network is given in this case by
\begin{equation}\label{eq:mainsystCL3}
  \begin{array}{rcl}
    \dot{x}(t)&=&Ax(t)-e_n  x_n (t) k_pz(t)+b_0\\
    \dot{z}(t)&=&\alpha z(t)(y(t)-r).
  \end{array}
\end{equation}

Similar calculations as in the antithetic case yield the following result:
\begin{proposition}
  The closed-loop network \eqref{eq:mainsystCL3} has the following equilibrium points:
  \begin{enumerate}[(a)]
    \item One positive equilibrium point defined as
    \begin{equation}
  x^*=-\bar{A}^{-1}(-e_nk_pr z^*+b_0) \textnormal{ and }z^*=\dfrac{g_0-r}{g_nrk_p},
\end{equation}
where $r<g_0$.
\item One zero-equilibrium point given by
\begin{equation}
  x^*=-A^{-1}b_0\textnormal{ and }z^*=0.
\end{equation}
  \end{enumerate}
\end{proposition}
It is known from \cite{Briat:19:Logistic} that the zero-equilibrium point is structurally unstable. So, we just need to focus on proving the structural stability of the positive equilibrium point. This is stated in the following result:
\begin{theorem}
  Assume that $A$ is Metzler and Hurwitz stable. Then the unique positive equilibrium point of the system \eqref{eq:mainsystCL3} is locally exponentially stable for all $\alpha,k>0$ and $\mu<g_0$.
\end{theorem}
\begin{proof}
  The linearized dynamics are governed by the matrix
  \begin{equation}
  \begin{bmatrix}
    \bar{A} & -e_nk_p\mu\\
    \dfrac{\alpha(g_0-\mu)}{g_n\mu k_p}e_n^T & 0
  \end{bmatrix}.
  \end{equation}
  This can be interpreted as the negative interconnection of the transfer function $H_n(s)$ and the integrator $\alpha(g_0-\mu)/(g_ns)$. By virtue of the previously obtained results, we know that this interconnection is asymptotically stable for all $\alpha,k>0$ and $\mu<g_0$. This proves the result.
\end{proof}

We have the following immediate extension to the case where the matrix $A$ is Metzler and output unstable:
\begin{theorem}
    Assume that $A$ is Metzler, output unstable, and non-singular. Then, the unique positive equilibrium point of the system \eqref{eq:mainsystCL3} is locally exponentially stable for all $\mu,\alpha,k>0$.
\end{theorem}
\begin{proof}
  The proof follows from the same arguments as for the antithetic integral controller.
\end{proof}

\subsection{Logistic integral controllers}\label{sec:logistic}

We consider now the logistic integral controller \cite{Briat:19:Logistic}
\begin{equation}
  \dot{z}(t)=\dfrac{\alpha}{\beta} z(t)(\beta - z(t))(y(t)-r)
\end{equation}
where $\alpha>0$ is a parameter of the controller, $\beta>0$ is the saturation level, and $r$ is the desired set-point. The closed-loop network is given in this case by
\begin{equation}\label{eq:mainsystCL4}
  \begin{array}{rcl}
    \dot{x}(t)&=&Ax(t)-e_n  k_px_n(t)z(t)+b_0\\
    \dot{z}(t)&=&\dfrac{\alpha}{\beta} z(t)(\beta - z(t))(y(t)-r).
  \end{array}
\end{equation}

We have the following result:
\begin{proposition}
  The closed-loop network \eqref{eq:mainsystCL4} has the following equilibrium points:
\begin{itemize}
  \item One positive interior equilibrium point given by
\begin{equation}
  (x^*,z^*)=\left(-\bar{A}^{-1}(-e_n r  z^*+b_0),\dfrac{g_0-r}{g_n r}\right),
\end{equation}
where $r\in(g_0/(1+\beta k_pg_n),g_0)$.
\item One zero-equilibrium point given by
\begin{equation}
  (x^*,z^*)=(-A^{-1}b_0,0).
\end{equation}
\item One positive saturating equilibrium point given by
\begin{equation}
  (x^*,z^*)=-\left((A-e_ne_n^Tk_p\beta)^{-1}b_0,\beta\right).
\end{equation}
\end{itemize}
\end{proposition}
It is known from \cite{Briat:19:Logistic} that the zero and saturating equilibrium points are structurally unstable. Therefore, we just need to focus on proving the structural stability of the positive equilibrium. This is stated in the following result:
\begin{theorem}
  Assume that $A$ is Metzler and Hurwitz stable and that $\beta>0$. Then the unique positive equilibrium point of the system \eqref{eq:mainsystCL4} is locally exponentially stable for all $\alpha>0$ and all $$ r \in\left(\dfrac{g_0}{1+k_p\beta g_n},g_0\right).$$
\end{theorem}
\begin{proof}
  The linearized dynamics are governed by the matrix
  \begin{equation}
  \begin{bmatrix}
    \bar{A} & -e_nk_p r \\
    \dfrac{\alpha}{\beta}z^*(\beta-z^*)e_n^T & 0
  \end{bmatrix}.
  \end{equation}
  This can be interpreted as the negative interconnection of the transfer function $H_n(s)$ and the integrator $k\beta z^*(\beta-z^*)/(\beta s)$. Since the gain of the integrator is always positive, then the transfer function is positive real and the previously obtained apply.
\end{proof}

We have the following immediate extension to the case where the matrix $A$ is Metzler and output unstable:
\begin{theorem}
  Assume that $A$ is Metzler, output unstable, and non-singular. Then, the unique positive equilibrium point of the system \eqref{eq:mainsystCL4} is locally exponentially stable for all $k>0$ and all $ r > g_0/(1+\beta k_pg_n)$.
\end{theorem}
\begin{proof}
  The proof follows from the same argument as for the exponential and the antithetic controllers.
\end{proof}}

\end{document}